\documentclass[cmp]{svjour}  
\usepackage[centertags]{amsmath}
\usepackage{amsfonts}
\usepackage{amsmath}
\usepackage{amssymb}
\usepackage{tikz}
\journalname{Communications in Mathematical Physics}

\newtheorem{observation}{Observation}

\begin{document}
\newcommand{\lims}{{\overline{\lim}}}
\newcommand{\NN}{{\mathbb N}}
\newcommand{\RR}{{\mathbb R}}
\newcommand{\ZZ}{{\mathbb Z}}
\newcommand{\QQ}{{\mathbb Q}}
\newcommand{\CC}{{\mathbb C}}
\newcommand{\dv}{\operatorname{dv}}
\newcommand{\dz}{\operatorname{dz}}
\newcommand{\dxdy}{\operatorname{dxdy}}
\newcommand{\area}{{\rm area}}
\newcommand{\diam}{{\rm diam}}
\newcommand{\deriv}[2]{\frac{d #1}{d #2}}
\newcommand{\pderiv}[2]{\frac{\partial #1}{\partial #2}}
\newcommand{\dbyd}[1]{\frac{d}{d #1}}
\newcommand{\pdbyd}[1]{\frac{\partial}{\partial #1}}
\newcommand{\tensor}{\otimes}
\newcommand{\LL}{\mathcal{L}}
\newcommand{\p}{\prime}
\newcommand{\A}{\mathcal{A}}
\newcommand{\B}{\mathcal{B}}
\newcommand{\X}{\mathcal{X}}
\newcommand{\bigo}{\mathcal{O}}
\newcommand{\smallo}{\mathcal{o}}
\newcommand{\cG}{\mathcal{G}}
\newcommand{\cGo}{\mathcal{G}_\omega}
\newcommand{\cH}{\mathcal{H}_{l,\omega}}
\newcommand{\Hlo}{\mathcal{H}_{l,\omega}}
\newcommand{\bH}{\bf{H}}
\newcommand{\te}{\theta}
\newcommand{\g}{\gamma}
\newcommand{\cS}{\mathcal{S}}
\newcommand{\uu}{\textbf{u}}
\newcommand{\tu}{\tilde{u}^0}
\newcommand{\Y}{\mathcal{Y}}
\newcommand{\Z}{\mathcal{Z}}
\newcommand{\proj}{\mathrm {proj}}
\newcommand{\dive}{\mathrm {div}}

\begin{center}
\textbf{{Stability of Stationary Wave Maps from a Curved Background to a Sphere}}
\end{center}
\vspace*{2 cm}
\textbf{Author:} Sohrab M. Shahshahani\let\thefootnote\relax\footnotetext{Section de math$\acute{e}$matiques, $\acute{E}$cole Polytechnique F$\acute{e}$d$\acute{e}$rale de Lausanne, FSB-SMA, Station 8 - B$\hat{a}$timent MA, CH-1015 Lausanne, Suisse\\
E-mail address: sohrab.shahshahani@epfl.ch\\\\This work is part of the author's Ph.D. thesis.}\\\\\\\\

\abstract{ We study time and space equivariant wave maps from $M\times\RR\rightarrow S^2,$ where $M$ is diffeomorphic to a two dimensional sphere and admits an action of $SO(2)$ by isometries. We assume that metric on $M$ can be written as $dr^2+f^2(r)d\theta^2$ away from the two fixed points of the action, where the curvature is positive, and prove that stationary (time equivariant) rotationally symmetric (of any rotation number) smooth wave maps exist and are stable in the energy topology. The main new ingredient in the construction, compared with the case where $M$ is isometric to the standard sphere (considered by Shatah and Tahvildar-Zadeh \cite{ST1}), is the the use of triangle comparison theorems to obtain pointwise bounds on the fundamental solution on a curved background.}

\def\thesection{\arabic{section}}
\section{Introduction}

\begin{center}

\end{center}

\vspace*{0.5cm}
This work is a generalization of the work of Shatah and Tahvildar-Zadeh on the stability of equivariant wave maps between spheres \cite{ST1}. We consider wave maps (to be defined below) $U:~M\times \RR \rightarrow S^2,$ where the manifold $M$ is diffeomorphic to $S^2,$ and accepts an effective action of $SO(2)$ by isometries. This action will have exactly two fixed points, and we assume that the curvature is positive near these points. Moreover we assume that the metric on $M$ has the form $dr^2+f^2(r)d\theta^2,~r\in(0,R),~\theta\in S^1,$ away from the fixed points. $f$ behaves like $\sin(\frac{\pi r}{R})$ near the endpoints of $(0,R).$ In fact as shown in \cite{dC} p. 292, $f$ satisfies
\begin{align*}
&f(r) = r - \frac{k(r)}{6}r^3 + Y(r),\\
&\lim_{r\rightarrow 0}\frac{Y}{r^3} = 0.
\end{align*}
Shatah and Tahvildar-Zadeh treated the special case where $M=S^2.$ With $A=\bigl(\begin{smallmatrix} 0 & -1 & 0 \\ 1 & 0 & 0 \\ 0 & 0 & 0\end{smallmatrix}\bigr)$ denoting the infinitesimal generator of the action of $SO(2)$ on $\RR^3,$ our main result can be summarized as follows:\\\\

\textbf{\emph{Main Result:}} \textit{For every nonzero real number $\omega$ and nonzero integer $l,$ there exists a compact (in the energy topology) family of initial data leading to wave maps satisfying $U(t,r,\theta)=e^{A(\omega t+l\theta)}u(r),$ which is stable in the energy norm under small perturbations of the initial data in the same equivariance class.}\\\\

We refer the reader to Theorem 5 for a more precise statement.\\\\
\subsection{Background and History}
\par
Wave maps are by definition the critical points of the action
\begin{align*}
\mathcal{Q}(U) := \int_{M\times \RR}\Big<\partial_\alpha U,\partial^{\alpha}U\Big>d\mu_g,\quad\alpha=0,1,2,
\end{align*}
where $\partial^\alpha=g^{\alpha\beta}\partial_\beta$ for $g= \textrm{diag}(-1,1,f^2(r))$ the Lorentzian metric on $M\times\RR.$ Einstein's summation convention is in effect here, and $<,>$ denotes the standard inner product on $\RR^3.$ More concretely the action takes the form
\begin{align*}
\mathcal{Q}(U) = \frac{1}{2}\int_{S^2\times\RR}(-|U_t|^2 + |\nabla_{r,\theta} U|^2)f(r)dtdrd\theta,
\end{align*}
and a stationary point of this action satisfies
\begin{align*}
 \Box U = B(U)(DU,DU),
 \end{align*}
 if the target is embedded or
 \begin{align*}
 \Box U^c + \Gamma^c_{db}(U)\partial^\alpha U^d\partial_\alpha U^b = 0
 \end{align*}
 in local coordinates. Here $B$ denotes the second fundamental form of the embedding, $\Gamma^c_{bd}$ are the Christoffel symbols of the target, and $\Box$ is the wave operator on the domain $\partial_t^2-\triangle_M.$ We will use both the intrinsic and the extrinsic formulations. For readers familiar with harmonic maps, we mention that wave maps are the equivalent of harmonic maps in the Lorentizian setting much in the same way as the wave operator is the equivalent of the Laplacian. In other words, wave maps are defined exactly in the same way as harmonic maps with the sole distinction that the metric on the domain is assumed Lorentzian rather than Riemannnian.\\\\
 Wave maps have been studied extensively in the case of a flat background. The current work uses mainly the methods developed in the early works of Christodoulou, Shatah, and Tahvildar-Zadeh on wave maps with rotational symmetries. These are \cite{ST2} and \cite{CT} where the authors prove global regularity for the Cauchy problem of rotationally equivariant wave maps to rotationally symmetric Riemannian targets, under not very restrictive assumptions on the target. The theory for a curved background is less developed. The paper \cite{ST1} of Shatah and Tahvildar-Zadeh on wave maps from $S^2\times\RR$ to $S^2$ is the one the current work attempts to generalize. In \cite{ST1} The authors prove that between these spaces, space and time equivariant wave maps exist. Due to the time symmetry these are called stationary maps. Global well-posedness is also proved for spatially equivariant data, under small energy assumptions. Shatah and Tahvildar-Zadeh then combine this with the existence to prove a stability result in the energy norm for the stationary wave maps. The general outline of our work is the same as \cite{ST1}. The difference is that the domain sphere is equipped with a general rotationally symmetric metric which may be different from the standard one. This means that unlike \cite{ST1}, we do not have the Penrose transform at our disposal, and we are therefore led to consider the fundamental solution to the wave equation on a curved background derived in \cite{F}. See subsections 1.2 and 1.3 below for more details.\\

 Before we continue with the description of the results in the current paper, we digress to provide a brief history of the problem, mostly in the case of flat backgrounds. Using the same techniques as the works discussed in the preceding paragraph Christodoulou and Tahvildar-Zadeh \cite{CT2} proved scattering of wave maps with smooth radial data to targets satisfying similar conditions to those considered in \cite{CT}. For flat backgrounds, the issue of local well-posedness for the wave maps Cauchy problem in subcritical regularities $s>n/2$ was settled in \cite{KT}, \cite{KM} and \cite{KSel1}, \cite{KSel2}. At the critical regularity Tataru proved well-posedness for $n\geq2$ in the Besov space ${\dot{B}}^{n/2}_1$ in \cite{Tat1} and \cite{Tat2}. On the level of Sobolev spaces, in the breakthrough papers \cite{T1}, \cite{T2} Tao proved global regularity for data with small critical Sobolev norm, and target $S^2.$ Using similar methods, this was later extended to other targets and dimensions by Klainerman, Rodnianski \cite{KR}, Krieger \cite{Kr1}, \cite{Kr2}, \cite{Kr3}, Nahmod, Stephanov, Uhlenbeck \cite{NSU}, and Tataru \cite{Tat3}, \cite{Tat4}. For $n\geq 4$ Shatah and Struwe \cite{SS2} obtained similar results to those of \cite{T1} but using a simpler machinery. In dimension $4$ Lawrie \cite{Law} has recently generalized the results of \cite{SS2} to certain perturbations of the standard metric on the domain. Under assumptions of rotational symmetry, Struwe showed in \cite{St1} and \cite{St2} that blow up leads bubbling of non-constant harmonic maps, which in the radial context proves global regularity for compact targets, regardless of the size of the data. Using the methods of these papers Nahas \cite{N} extended the scattering result in \cite{CT2} to general compact targets. In \cite{KS} Krieger and Schlag were able to further develop the concentration compactness techniques introduced in \cite{BG}, \cite{KM1}, \cite{KM2}, to prove global regularity and scattering for large critical data and hyperbolic targets. Using different techniques Sterbenz and Tataru \cite{StT1}, \cite{StT2} proved global regularity for data with energy below that of the lowest energy harmonic map, for more general targets. Tao \cite{T3} has also been able to extend the methods of \cite{T1}, \cite{T2} to obtain global regularity for large data and hyperbolic targets. See \cite{CKM}, \cite{C} and \cite{LS} for other instances of applications of the ideas in \cite{KM1}, \cite{KM2} to wave maps. The construction of blow up solutions to the wave maps problem was initiated by Shatah \cite{S} and Cazenave, Shatah, Tahvildar-Zadeh \cite{CST}. Following the bubbling results of Struwe \cite{St1} and \cite{St2} formation of singularities were further studied in the works of Krieger, Schlag, Tataru \cite{KST}, Rodnianski, Sterbenz \cite{RS}, Raphael, Rodnianski \cite{RR}, and Carstea \cite{Ca}. The interested reader can consult the historical surveys in \cite{KS}, \cite{KST}, \cite{RR}, \cite{SS1}, and \cite{ST1} for more information on the history of the problem.\\\\

 \subsection{General Outline}

 \par

Following \cite{ST1} we divide our work into the following three parts. From here on we write $S^2$ for $M$ with the understanding that the metric is not necessarily the standard one.\\\\

Section 2: Existence of Stationary Wave Maps.\\\\

We will sometimes refer to this part as the elliptic part, because the methods used here come mainely from the theory of elliptic PDE. Here we make the basic time and space equivariance assumptions
\begin{align*}
&U(t,x)=e^{A\omega t}u(x),\\
&u(r,\theta)=e^{Al\theta}u(r,0),
\end{align*}
where $A=\bigl(\begin{smallmatrix} 0 & -1 & 0 \\ 1 & 0 & 0 \\ 0 & 0 & 0\end{smallmatrix}\bigr)$ is the infinitesimal generator of the action of $SO(2)$ on $\RR^3,$ and $l\in \ZZ,~\omega\in\RR_{+}$ are fixed. Such maps will be referred to as stationary and the aim of this section is to prove the existence of stationary wave maps.
 The time equivariance assumption reduces the problem of finding minimizers $U(t,x)$ for $L$ to that of finding minimizers $u(x)$ for a reduced functional $\cGo.$ The main theorem of this section then states that if we restrict $\cGo$ to $H^1$ maps of degree $l$ satisfying the space equivariance condition, every minimizing sequence converges strongly sub-sequentially in $H^1$ to a smooth map of degree $l.$ This establishes the existence of smooth stationary wave maps. Note that since we are in two space dimensions, according to a theorem of Schoen and Uhlenbeck (see \cite{St3} page 253, or \cite{SU} section $IV$) the homotopy class of a map in $H^1(M,S^2)$ is well-defined and we can therefore talk about the degree of such maps.\\\\

 Sections 3 and 4: Holder Continuity, and Higher Regularity.\\\\

This part will occasionally be referred to as the hyperbolic part, because here we treat the Cauchy problem for wave maps in our setting. This is the longest section. We drop the time equivariance assumption and prove that the Cauchy problem for spatially equivariant data of low energy is well posed. Specifically suppose we are given initial data $(U(0),U_t(0))=(f,g),$ where $f$ and $g$ satisfy the second equivariance assumption above and $\|df\|_{{H}^1}+\|g\|_{L^2}<\epsilon.$ Then the main theorem states that if $\epsilon>0$ is chosen sufficiently small, the Cauchy problem
 \begin{align*}
 \Box U = B(U)(DU,DU)
 \end{align*}
 with initial data $(f,g)$ has a global smooth solution. Local existence is standard and the difficulty lies in proving that blow up cannot occur. Because of the symmetry the first possible blow up will be at the origin of space-time (after an appropriate time shift). This part is divided into two subparts:\\\\

 a. Holder continuity:\\\\

 Here we prove Holder estimates on the local solution by bounding its derivatives along the characteristic directions $\eta=t+r$ and $\xi=t-r.$ We conclude that we have a globally defined Holder continuous solution.\\\\

 b. Higher regularity:\\\\

 Here we prove that the solution found in subpart a is smooth. This is done by covariantly differentiating the equation and then applying the methods of part a.\\\\

 Section 5: Stability\\\\

 We combine the results of the previous two parts with conservation methods to show that the stationary maps found in the elliptic section are stable under equivariant perturbations. To state the results precisely, we introduce the product norm topology on $H^1\times L^2(S^2),$ and to any minimizer $v$ for the functional $\cGo$ of the elliptic section we associate an element $\tilde{v}=(v,A\omega v)$ in $H^1\times L^2(S^2).$ Let $S$ be the family of such elements. $S$ is compact by the results of the elliptic section. Similarly given an $H^1$ map $U(x,t)$ we let $U^t:=(U(t,.),U_t(t,.))\in H^1\times L^2(S^2).$  We denote by $d$ the distance function associated with the norm above. Fix a stationary solution $v$ corresponding to $\tilde{v}\in H^1\times L^2.$ The main theorem states that there is a small $\eta>0$ (independent of $v$) such that given $\epsilon \in(0,\eta)$ we can find $\delta(\epsilon)$ such that if $u$ is the classical solution on $[0,T)\times S^2$ with initial data $(f,g)$ satisfying $d((f,g),\tilde{v})<\delta,$ then $d(u^t,S)\leq \epsilon$ for all $t<T$ and $T$ is in fact infinity.\\\\

 \subsection{Detailed Outline}

 \par

 We now proceed to provide the outlines of our arguments in the elliptic and hyperbolic parts, highlighting the differences with previous works. The stability part is almost word by word the same as the corresponding section in \cite{ST1} and requires no further elaboration. Before continuing we remark that except in the section on higher regularity (section 4 in the text) the target is assumed to be a surface of revolution diffeomorphic to $S^2.$ Even the arguments for higher regularity are expected to work for this more general setting with mild modifications, but we carried them out in the simpler case for additional clarity.\\\\
 Section 2: Existence of Stationary Wave Maps.\\\\
 This is the elliptic part, and the arguments follow closely those of \cite{ST1}. In our case the metric in polar coordinates takes the form $dr^2+f^2(r)d\theta^2$ instead of $d\alpha^2+\sin^2\alpha~d\beta^2,$ but since $f(r)$ behaves like $\sin r$ near the end points most of the arguments from \cite{ST1} carry over easily. The key point here is that using the geometry of the problem we can derive an expansion near $r=0$ of the form
 \begin{align*}
 f(r)=r-\frac{k(r)}{6}r^3+Y(r),
 \end{align*}
 (and a similar one near the other fixed point of the $SO(2)$ action, i.e. the other end point of the definition of the radial variable $r$) where $k$ is the Gaussian curvature and $Y=o(r^3)$ as $r$ approaches zero.\\
 We also take the target to be a general surface of revolution in $\RR^3$ diffeomorphic to the standard sphere, and satisfying the same geometric conditions as the domain. This allows us to write the metric on the target as $d\alpha^2+g^2(\alpha)d\beta^2$ in polar coordinates, where $g$ satisfies a similar expansion near the end points of the domain of definition of $\alpha.$\\
 We want to find minimizers of the wave maps functional
 \begin{align*}
\mathcal{Q}(U) = \frac{1}{2}\int_{S^2\times\RR}(-|U_t|^2 + |\nabla U|^2)f(r)dtdrd\theta .
 \end{align*}
 Just as in \cite{ST1} writing $U(t,x)=e^{A\omega t}u(x)$ we reduce our task to finding minimizers for the functional
 \begin{align*}
 \cGo(u)=\frac{1}{2}\int_{S^2}(|\nabla u|^2 - \omega^2| Au|^2)f(r)drd\theta,
 \end{align*}
 defined on $H^1(S^2;S^2).$ Such a minimizer $u$ will satisfy the elliptic Euler Lagrange equations
 \begin{align*}
\Delta_{S^2}u  = B(u)(\nabla u, \nabla u) + \omega^2(A^2u - < A^2u, \nu(u)>\nu(u)),
 \end{align*}
 where $\nu$ is the unit outward-pointing normal vector on the target. We will show that if $u\in H^1$ is spatially equivariant then it is Holder continuous away from the end points of $r$ (sometimes referred to as the poles from now on), with Holder norm depending on the $H^1-$norm of $u,$ and it can be written as $u(r,\theta)=(\varphi(r),\zeta(r)\theta^l).$ Here we are using the polar coordinates $(\alpha,\beta)\in (0,H)\times S^1$ on the target. $\varphi:[0,R]\rightarrow [0,H]$ is continuous and sends end points to end points and $\zeta:[0,R]\rightarrow S^1$ is of little consequence. In fact $\zeta$ can be set equal to $1$ without loss of generality. Moreover the degree of $u$ is $\pm l$ or $0$ depending on the image of the end points of $[0,R]$ under $\varphi,$ and $\varphi(0)=0,~\varphi(R)=H$ if $\deg u =l.$ This degree computation is the first point where our work differs from \cite{ST1}. There $\varphi$ has the entire real line as its target and the degree of $u$ may be any multiple of $l.$ A more complicated degree formula is also derived for $u$ depending on $\varphi$ (see \cite{ST1} page 236). We show that the only possible degrees are $l,~-l$ and $0,$ and that the the range of $\varphi$ is $[0,H]$ ($H=\pi$ if the target is the standard sphere). This is proven analytically, but a topological argument is also provided in a footnote.\\
 We denote the set of equivariant $H^1$ maps of degree $l$ by $X_l.$ The main theorem of this section states that if $\{u_n\}$ is a minimizing sequence for $\cGo$ in $X_l$ then there is a subsequence that converges strongly in $H^1$ to a smooth map in $X_l.$\\
 The representation above for $u$ allows us to reduce the problem to the study of the minima of the functional
 \begin{align*}
 \Hlo(\varphi):=\pi\int_0^R[(\varphi^{'})^{2}+ (\frac{l^2}{f^2}-\omega^2)g^2(\varphi)]fdr,
 \end{align*}
 over the set of continuous $\varphi:[0,R]\rightarrow [0,H]$ satisfying the boundary conditions $\varphi(0) =0$ and $\varphi(R)= H.$ Specifically it will be sufficient to show that if $\varphi_n$ is a minimizing sequence of continuous functions satisfying the right boundary conditions, then a subsequence converges strongly in $H^1$ to a $C^{1,\gamma}$ function satisfying the same boundary conditions. This will allow us to go from the minimizing sequence $u_n$ for $\cGo$ to a minimizing sequence for $\Hlo$ which has to converge to a minimizer $\varphi$. This is then used to construct a minimizer $u$ for $\cGo.$ Higher regularity of $u$ follows easily from the ellipticity of the equation satisfied for $u.$ This general outline is identical to that of \cite{ST1}.\\
 The study of minimizers of $\Hlo$ has two steps: existence and regularity. The regularity argument in \cite{ST1} is a clever argument using elliptic theory. It involves thinking of the minimizer whose existence is guaranteed by the existence part as a radial function on a higher dimensional sphere. This allows one to use better Sobolev embeddings which combined with elliptic regularity yield the desired regularity. Our proof follows the same guideline, but keeps track of the error resulting from having $f(r)$ instead of $\sin\alpha$ in the metric.\\
 For the existence, suppose $\{\varphi_n\}$ is a sequence of minimizers for $\Hlo.$ A standard argument shows that there is a $\varphi$ such that along some subsequence $\{\varphi_n\}$ converges uniformly on compact subintervals and weakly in $H^1$ to $\varphi.$ \
There are three possibilities for the limit $\varphi$:\\
(i)$\varphi(0)=0$ and $\varphi(R)=H$\\
(ii)$\varphi \equiv 0$ or $\varphi \equiv H$\\
(iii)$\varphi(0) = \varphi(R)=0$ but $\varphi$ is not constantly $0$ or the same statement with $0$ replaced by $H$.\\
To establish existence a concentration compactness argument is needed to eliminate the last two cases. If we were in case (ii) then the lower bound for $\cGo$ would have to be the same as the case where $\omega=0$ which corresponds to harmonic maps. This yields a contradiction, since a direct computation shows that the value of $\cGo$ at a harmonic map is strictly less than the $\omega=0$ case. The existence of harmonic maps for the case where the spheres are equipped with general metrics is shown in \cite{L} (see section 2 for more precise references). Eliminating case (iii) is more involved, and the most important difference between \cite{ST1} and the current work in the elliptic section is in this argument. In \cite{ST1} this is done by showing monotonicity of the members of a minimizing sequence which is in contradiction with case (iii). The reason the arguments from \cite{ST1} do not carry over is that $f(r)$ may have many more oscillations in its domain of definition that $\sin \alpha.$ We eliminate case (iii) in lemma 2, by modifying each $\varphi_n$ in the minimizing sequence in a way that $\Hlo(\varphi_n)$ is reduced by an amount independent of $n,$ contradicting the minimizing property of the sequence.\\\\

Section 3: Holder Continuity.\\\\

Here we prove Holder continuity. Local existence is standard, and due to symmetry it suffices to establish uniform Holder bounds on the local solution near $(t,r)=(0,0).$ We can therefore restrict attention to a neighborhood of one of the poles (the data here is given at $t=-1$). In \cite{ST1} this is done by using the conformal Penrose transform to convert the local problem into a wave equation in the Minkowski space. The fundamental solution derived in an earlier work of Chrisrodoulou and Tahvildar-Zadeh \cite{CT} for the flat radial wave operator is then used to find a convenient integral representation for the local solution, which is transformed back using the inverse Penrose transform. Combining the estimates on the fundamental solution from \cite{CT} with similar arguments to those in \cite{ST2}, Shatah and Tahvildar-Zadeh then prove estimates on the derivatives of the solution which yield Holder continuity. Specifically let $\eta=t+r$ and $\xi=t-r$ denote the characteristic coordinates. With $e$ and $m$ denoting the energy and the momentum respectively, define
\begin{align*}
\A^2 := f(r)(e+m)= \frac{f}{2}(|\partial_\eta u|^2 + \frac{l^2}{f^2}|Au|^2),\\
\B^2 := f(r)(e-m)= \frac{f}{2}(|\partial_\xi u|^2 + \frac{l^2}{f^2}|Au|^2).
\end{align*}
For $\tilde{t} \leq 0$ and a fixed $\delta \in (0,1/2)$ define
\begin{align*}
&Z(\tilde{t}) = \{(t^\p,r^\p)| -1 \leq t^\p \leq \tilde{t},~0\leq r^\p \leq \tilde{t} -t^\p\},\\
&\X(\tilde{t}) = \sup_{Z(\tilde{t})}\big\{(f(r))^{(\frac{1}{2}-\delta)}\A(t,r)\big\},\\
&\X = \X(0).
\end{align*}
$f(r)$ is replaced by $\sin\alpha$ in \cite{ST1}. Shatah and Tahvildar-Zadeh show that $\X$ is bounded if the bound $\epsilon$ on energy is sufficiently small. This is done by proving an estimate of the form $\X\lesssim 1+\epsilon\X.$ Deducing Holder bounds from the boundedness of $\X$ is not very difficult. We follow the same outline, but use different techniques to estimate $\partial_\eta u.$ Our point of deviation is that we no longer have the Penrose transform at our disposal. Instead, as suggested in \cite{ST1}, we use the fundamental solution for the wave equation on a curved background derived in \cite{F} to get the following representation for the solution $u$ near a pole
\begin{align*}
\int_{\xi}^{\eta}\int_{-2-\eta^{'}}^{\xi}\int_{\{\theta^{'}|(t-t^{'})^2 \geq d^2((r,\theta),(r^{'},\theta^{'}))\}}\frac{w^{+}Q(U)f(r^{'})}{\sqrt{(t-t^{'})^2- d^2((r,\theta),(r^{'},\theta^{'}))}}d\theta^{'}d\xi^\p d\eta^\p .
\end{align*}
Here $Q$ denotes the nonlinearity, $w^{+}$ is a smooth function whose existence is guaranteed in \cite{F}, and $d$ denotes the geodesic distance function on the domain. Note that the domain of integration is just the backward light cone of the new metric. Our arguments will be more similar to those of \cite{CT} than \cite{ST1}. The reason is that after the Penrose transform the wave equation in \cite{ST1} takes the form $\Box u +\frac{1}{r^2}u=nonlinearity.$ The fundamental solution derived in \cite{CT} for $\Box +\frac{1}{r^2}$ behaves slightly differently from that of $\Box.$ In particular for the latter the backward light cone is divided into two regions $K_1: |t-t^\p|\leq|r+r^\p|$ and $K_2:|t-t^\p|\geq|r+r^\p|$ to obtain optimal estimates. This division is present in our work also, but does not appear in \cite{ST1}. The main difficulty for us is that we do not have a simple law of cosines for our geometry and the singularity $((t-t^\p)^2-d^2)^{-1/2}$ is therefore more complicated to deal with. Our basic idea is to use triangle comparison theorems to compare $d$ with its flat and spherical counterparts. This will not always be easy as we have to consider derivatives of the singularity too. Furthermore the extra $\theta$ dependency prevents some of the arguments in \cite{CT} from being directly generalizable to our case. We can divide the results of this section into two parts:\\\\
a) Estimates on the distance function and its derivatives, and geometric comparison theorems: these consist of the last 7 lemmas of the section (lemmas 2-8) and corollary 1.\\\\
b) Adapting the arguments in \cite{CT} and \cite{ST1} to our setting, using the previous set of results: these include lemma 1 and the arguments in the main body of the work between the lemmas.\\\\
In order to motivate the results we have not made the above separation explicit in the actual write up. The lemmas are stated and proved when they arise in the context of results of type b. However, the reader can read lemmas 2-7 independently of the rest of the work, or skip their proofs and concentrate on the rest. As mentioned earlier the main ingredients in the proofs of our geometric lemmas are triangle comparison theorems. The most difficult of these is perhaps lemma 7 in that it uses all previous lemmas in its proof.\\
The rest of the argument goes as follows. We want an estimate of the form $|\partial_\eta u|\lesssim \epsilon r^{\delta-1}\X,$ so we differentiate the integral representation for $u$ to get
\begin{align*}
&\partial_\eta u = \\
&\lim_{\eta^{'} \rightarrow \eta} \int_{-2-\eta^{'}}^{\xi}\int_{\{\theta^{'}|(t-t^{'})^2 \geq d^2((r,\theta),(r^{'},\theta^{'}))\}}\frac{w^{+}Q(U)f(r^{'})}{\sqrt{(t-t^{'})^2- d^2((r,\theta),(r^{'},\theta^{'}))}}d\theta^{'}d\xi^\p +\\
& \int_{\xi}^{\eta}\int_{-2-\eta^\prime}^{\xi} \partial_\eta[\int_{\{\theta^{'}|(t-t^{'})^2 \geq d^2((r,\theta),(r^{'},\theta^{'}))\}}\frac{w^+Q(U)f(r^{'})}{\sqrt{(t-t^{'})^2- d^2((r,\theta),(r^{'},\theta^{'}))}}d\theta^{'}]d\xi^\p d\eta^\p\\
&=:I+II.
\end{align*}
The most important step in bounding the boundary term $I$ is lemma 1, which in turn relies on the geometric lemmas 2 and 3. This is treated somewhat differently in the flat case \cite{CT}, due to the existence of the law of cosines. Bounding $II$ is more difficult and we need to divide $II$ as $II_1 + II_2$ by splitting the region of integration into $K_1$ and $K_2.$ $II_1$ is more difficult to deal with. The method is similar to that of \cite{CT} but modifications are needed in several cases. The most important is the section between the end of the proof of lemma 4 and the beginning of lemma 7. The new geometric lemmas that are needed for bounding $II_1$ are lemmas 4 to 7.\\
Bounding $II_2$ is simpler because the region of integration is fixed (independent of $\eta$) and therefore the comparison lemma 8 allows us to use the results from the flat case with minor modifications.\\
The smallness factor comes from the important pointwise bound $|Q|\lesssim\frac{\A\B}{f(r)}$ on the nonlinearity, combined with a energy-flux identity which allows us to extract a smallness factor from the terms involving $\B.$\\\\

Section 4: Higher Regularity\\\\

We establish higher regularity for the Holder solution of the previous section. There are two possible approaches, both of which are based on differentiating the equation for $u$. One is to proceed as in \cite{CT} to find second derivative estimates on the fundamental solution and repeat the argument of the previous section with $u$ replaced by $v=u_r.$ This would be difficult for us because of the complicated representation of the fundamental solution in our case. The reason one would need to consider the second derivative of the fundamental solution, is that unlike for $u,$ the $\xi-$derivative of $v$ is not well behaved. In fact, in the previous section we use a flux-energy identity to deal with $u_\xi$ which gives rise to the smallness factor $\epsilon.$ Since the energy of $v$ is no longer bounded, we need to use an integration by parts to move one derivative (the $r-$derivative) from $v_\xi=u_{r\xi}$ to the fundamental solution. This proves to be tedious even in the flat case \cite{CT}. Instead we resort to an alternative method, which to the best of my knowledge was first used by Shatah and Struwe in \cite{SS1}. This consists of differentiating the equation for $u$ covariantly with respect to $r.$ The point is that when we differentiate covariantly the extra factor of $v_\xi$ no longer appears in the nonlinearity, and therefore we can use the estimates from the previous section to obtain Holder bounds on the gradient of $u,$ proving higher regularity.\\
There is, however, one more complication. The method just outlined works for the radial case, but the extra $\theta$ dependency makes the singularity worse at $r=0.$ This is easy to see if we consider the flat equation. If there is no $\theta$ dependency we have
\begin{align*}
u_{tt}-u_{rr}-\frac{1}{r}u_r=~nonlinearity.
\end{align*}
Differentiating and letting $v=u_r$ we get
\begin{align*}
v_{tt}-v_{rr}-\frac{1}{r}v_r +\frac{1}{r^2}v=~new~nonlinearity,
\end{align*}
which we can deal with. If there is a $\theta$ dependency however, the equation becomes
\begin{align*}
 e^{Al\theta}\Big(u_{tt}-u_{rr}-\frac{1}{r}u_r -\frac{l^2A^2u}{r^2}\Big)=~non-linearity.
\end{align*}
Differentiating this and letting $V=U_r=e^{Al\theta}u_r$ we get
\begin{align*}
\Box v +\frac{1}{r^2}V+\frac{2lA^2U}{r^3}= new~nonlinearity,
\end{align*}
and there is an extra negative power of $r$ in the last term on the left hand side. A careful examination of the arguments in \cite{CT} and \cite{SS1} shows that if we let $v=u_r,$ this singularity is too strong to allow us to deduce Holder estimates. What comes to our rescue is the following simple intertwining relation
\begin{align*}
\big[\partial_r+\frac{l}{r}\big]\big[\partial_r^2+\frac{1}{r}\partial_r-\frac{l^2}{r^2}\big]=\big[\partial_r^2+\frac{1}{r}\partial_r-\frac{(l-1)^2}{r^2}\big]\big[\partial_r+\frac{l}{r}\big].
\end{align*}
We can apply the operator $\partial_r+\frac{l}{r}$ to the equation and let $w=u_r+\frac{lu}{r}.$ We can then apply the arguments outlined above to obtain Holder estimates on $w$ and on the gradient of $u$ from there. Of course in our case we will need to replace $r$ by $f(r)$ which make the intertwining relation more complicated. We will also need a similar intertwining relation for the covariant derivative to be able to use the ideas of \cite{SS1}.\\
Finally we note again that this is the only section where we assume the target to be the standard sphere rather than a surface of revolution diffeomorphic to the standard $S^2.$ To deal with one of the terms arising after applying the covariant differentiation operator, we use the explicit formulas for the Christoffel symbols of the sphere. We expect a similar argument to work in the general case too. 
\section{Existence of Stationary Wave Maps}
\begin{center}

\end{center}

\vspace*{0.5cm}

Our goal is to prove existence and stability of smooth stationary wave maps $U: M \times \RR \rightarrow S^{2}$ where the domain sphere is equipped with a meteric supporting an effective action of $SO(2)$ by isometries, and the target sphere is isometerically embedded in $\RR^{3}.$ $M$ is a surface of revolution homeomorphic to $S^2$ with positive curvature near the poles. We assume the existence of polar coordinates $(r,\theta) \in (0,R) \times S^{1}$ on $M - \{p,q\},$ where $p$ and $q$ are the two poles. The curvature is assumed to be positive near the poles. The metric in these coordinates is given by $dr^{2} + f^{2}(r)d\theta^2.$ It follows from the computation of the Gaussian curvature that $f$ satisfies the following expansion (c.f. \cite{dC} p. 292)
\begin{align} \label{metric}
f(r) = r - \frac{k(r)}{6}r^3 + Y(r)
\end{align}
with
\begin{align*}
\lim_{r\rightarrow 0}\frac{Y}{r^3} = 0.
\end{align*}
Here $k$ is the Gaussian curvature. Similarly
\[f(R-r) = r - \frac{k(R-r)}{6}r^3 + X(r)\]
with
\[\lim_{r \rightarrow 0}\frac{X(r)}{r^3} = 0.\]

We now proceed as in \cite{ST1}, to prove the existence of stationary wave maps which satisfy a certain equivariance hypothesis. The expansion (\ref{metric}) will allow us to use the arguments in \cite{ST1} with minor modifications. The stationary (time equivariance) and symmetry (space equivariance) ansatze are respectively:
\begin{align}\label{time}
U(t,x) = e^{A\omega t}u(x)
\end{align}

\begin{align}\label{space}
u(r,\te) =\te^{l}\cdot u(r,1)
\end{align}

where $ u: S^2 \rightarrow S^2,~A$ is the matrix $\displaystyle  \bigl(\begin{smallmatrix} 0 & -1 & 0 \\ 1 & 0 & 0 \\ 0 & 0 & 0\end{smallmatrix}\bigr),~\omega$ is a fixed number in $\RR_{>0}$ and $l$ a fixed number in $\{1,2,3,\cdots \}.$
\begin{remark}
We assume that the action of $SO(2)$ on the target is given by the standard embedding of $SO(2)$ in $SO(3).$ Therefore in local coordinates equation (\ref{space}) reads $u(e^{As}x) = e^{Als}u(x).$ Equivariant functions are precisely the fixed point of the action of $SO(2)$ given by $(gv)(x):= g^{-l}v(gx).$
\end{remark}
Wave maps are critical points of the functional
\begin{align*}
\mathcal{Q}(U) = \frac{1}{2}\int_{S^2\times\RR}(-|U_t|^2 + |\nabla U|^2)dvol_{S^2\times\RR},
\end{align*}
and satisfy the Euler-Lagrange equations:
\begin{align}\label{UEL}
\Box U + (|U_t|^2 - |\nabla U|^2)U = 0.
\end{align}
In local coordinates $|\nabla U|^2 = |U_r|^2 + \frac{1}{f^2}|U_\te|^2$ and $\Box = {\partial_t}^2 - {\Delta}_{S^2} = {\partial_t}^2 -{\partial_r}^2 - \frac{f^{'}}{f}\partial_r -\frac{1}{f^2}{\partial_\te}^2.$\\
If the target is a surface of revolution homeomorphic to a sphere and gotten by rotating $c(\alpha):=(x,z)=(g(\alpha),h(\alpha)), ~\alpha \in [0,H],$ about the $z-$axis, the equation satisfied by $U$ becomes:
\begin{align}\label{UELgen}
\Box U = B(U)(DU,DU) := \sum_{i,j} \gamma^{ij}B(U)(\partial_iU,\partial_jU),
\end{align}
where B denotes the second fundamental form of the target and $DU$ is the total derivative of U. If $c$ is parametrized by arclength, the metric on the target in polar coordinates is $d\alpha^2 + g^{2}d\beta^2.$ As before $g$ satisfies equations similar to those satisfied by $f:$
\begin{align*}
&g(\alpha) = \alpha - \frac{K(\alpha)}{6}\alpha^3 + o(\alpha^3),\\
&g(H-\alpha) = \alpha - \frac{K(H-\alpha)}{6}\alpha^3 + o(\alpha^3),
\end{align*}
where $K$ denotes the Gaussian curvature function on the target.
Going back to the case where the target is the standard sphere, the stationary ansatz (\ref{time}) implies that $u$ satisfies
\begin{align}\label{uEL}
\Delta_{S^2}u + |\nabla u|^2u = \omega^2(A^2u + |Au|^2u).
\end{align}
The functional $\mathcal{Q}$ also reduces to the following functional defined on $H^1(S^2;S^2)$ whose critical points satisfy the Euler-Lagrange equation (\ref{uEL}):
\[\cG_\omega(u) = \frac{1}{2}\int_{S^2}(|\nabla u|^2 - \omega^2| Au|^2)dvol_{S^2}.\]
In the case of a general symmetric metric on the target this functional yields the Euler-Lagrange equations:
\begin{align}\label{uELgen}
\Delta_{S^2}u  = B(u)(\nabla u, \nabla u) + \omega^2(A^2u - < A^2u, \nu(u)>\nu(u)),
\end{align}
where $\nu$ denotes the unit normal vector to the target sphere. This equation can also be obtained by inserting the stationary ansatz (\ref{time}) into (\ref{UELgen}) (for further explanation see for instance \cite{St3}, p.233).
\begin{remark}
$\cGo$ is invariant under the action of $SO(2),$ in the sense that $\cGo(g(u)) = \cGo(u).$ This allows us to use equivariant variations to compute the Euler-Lagrange equations.
\end{remark}

The following lemma holds for general metrics on both the domain and the target spheres:
\begin{lemma}
Suppose $u \in H^1(S^2,S^2)$ satisfies the symmetry (\ref{space}) a.e.. Then with $N$ and $S$ denoting the north and the south poles on the target sphere we have:\\
(i) $u$ is Holder continuous away from the poles, that is $u \in C^{1/2}(V, \RR^3)$ for any $V\subset\subset S^2 - \{p,q\}$ and $u$ can be extended to a continuous map on all of $S^2$ sending $p$ and $q$ to poles. The Holder constant depends only on $V$ and the $H^1$ norm of $u.$\\
(ii) with $C^{N,S} := u^{-1}(\{N,S\}),$ there exists a bounded continuous function $\varphi: [0,R] \rightarrow [0,H],$ and a function $\zeta: (0,R) \rightarrow S^1 \subseteq \mathbb{C}$ continuous except possibly on the projection of $C^{N,S}$ on the first factor (in the $(r,\te)$ coordinates), so that $u$ is given by
\begin{equation}\label{phirep}
u(r,\te) = (\varphi(r),\zeta(r)\te^l).
\end{equation}
Moreover $deg~u = \pm l$ or $0,$ and if the degree is $l,~\varphi$ can be chosen so that $\varphi(R) = H$ and $\varphi(0) = 0.$
\end{lemma}

\begin{proof}
We use $\rho$ to denote the embedding of $SO(2) \cong S^1$ in $SO(3)$ given in local coordinates by $e^{As}.$ With this notation (\ref{space}) reads $u(r,\te) = \rho(\te)^lu(r,1),$ and since $\rho$ is Lipschitz we get $|u(r,\te_1) - u(r,\te_2)| \lesssim |\te_1-\te_2|.$ Also note that using the local representation for the metric on the domain sphere we have
\begin{align*}
&\|u\|^2_{L^2} = \int_0^{2\pi}\int_0^R|u(r,s)|^2f(r) dr ds = \int_0^{2\pi}\int_0^R|e^{Als}u(r,0)|^2f(r) dr ds\\
& = 2\pi \int_0^R|u(r,0)|^2f(r) dr.
\end{align*}
For $(r_j,\te_j)$ two points on $S^2$ we get
\begin{align*}
&|u(r_1,\te_1)-u(r_2,\te_2)| \leq  |u(r_2,\te_1)-u(r_2,\te_2)|+ |u(r_1,\te_1)-u(r_2,\te_1)|\\
&\lesssim |\te_1-\te_2| +|\int_{r_1}^{r_2} \pdbyd{r}u(r,\te_1)dr|\\
&\leq |\te_1-\te_2|+ \left(\int_{r_1}^{r_2} |\pdbyd{r}u(r,\te_1)|^2f(r) dr \right)^{1/2} \left (\int_{r_1}^{r_2}\frac{1}{|f(r)|} dr \right)^{1/2}\\
& \leq |\te_1-\te_2|+ \|u\|_{H^1}\{\min(|f(r_j)|)\}^{-1/2}|r_1-r_2|^{1/2}. 
\end{align*}
This establishes Holder continuity away from the poles. (To be precise, given $u$ we approximate it by a sequence $\{u_n\}$ of smooth functions which satisfy the above estimates.  These estimates allow us to conclude by Arzela-Ascoli that $\{u_n\}$ converges to a continuous function which by uniqueness of limits has to be the same as $u$ almost everywhere. We can then pass to the limits and conclude that the estimates are satisfied by $u$ as well.)\\
Now the equivariance of $u$ implies that any orbital circle (i.e. $\{\te \cdot x | \te \in SO(2)\}$) in the domain sphere is sent to a horizontal circle in the target sphere.  Therefore to show that $u$ can be continuously extended to a map sending poles to poles we need to show that $|u_1|^2+|u_2|^2 = |Au|^2 \rightarrow 0$ uniformly as $r \rightarrow 0,R$ and that $u_3$ is continuous. Again by equivariance of $u$ we have
\begin{align*}
&|\nabla u(r,s)|^2 = |\nabla e^{Als}u(r,0)|^2 = |\frac{lAe^{Als}u(r,0)}{f(r)}|^2 + |e^{Als}\pdbyd{r}u(r,0)|^2\\
&= \frac{l^2}{f^2(r)}|Au|^2 + |\pdbyd{r}u|^2.
\end{align*}
Since $\int_0^R|\nabla u|^2f(r)dr$ is finite and $\frac{l}{f(r)} \rightarrow \infty$ as $r \rightarrow 0,R,~|Au|^2$ must go to zero at least along some sequences $\{r_j^{1,2}\}$ going to $0$ and $R.$  Therefore to show that $|Au|^2 \rightarrow 0$ as $r \rightarrow 0,~R$ it suffices to prove $|Au(r +h,\te)|^2-|Au(r,\te)|^2 \rightarrow 0$ as $h \rightarrow 0:$
\begin{align*}
&|Au(r +h,\te)|^2-|Au(r,\te)|^2 = \int_r^{r + h}\pdbyd{\rho}|Au(\rho,\te)|^2 d\rho\\
&= 2\int_r^{r + h} \left < Au(\rho,\te),A\frac{\partial u}{\partial \rho}(\rho,\te) \right > d\rho\\
&\leq \int_r^{r + h}\left(\frac{|Au(\rho,\te)|^2}{f(\rho)} + |\frac{\partial u}{\partial\rho}|^2f(\rho) \right) d\rho \lesssim \int_r^{r + h}|\nabla u|^2f(\rho)d\rho = o(1).
\end{align*}
Continuity of $u_3$ follows from that of $u$:
\[|u_3(r+h,\te)- u_3(r,\te)|\leq |u(r+h,\te)- u(r,\te)| \rightarrow 0~\mathrm{as}~ h\rightarrow 0.~\]
For part $(ii)$ note that the equivariance of $u$ implies that for $(r,\te) \in ((0,R) \times S^1)-(C^{S} \cup C^{N}))$ we can write $u(r,\te) = \rho(\te)^lu(r,1) = (u_1(r,1), \rho(\te)^l \cdot u_2(r,1))$ (here $u_1$ and $u_2$ correspond to the angular coordinates $(\alpha,\beta)$ on the target, whereas up to here $u_1,~u_2,$ and $u_3$ corresponded to the standard coordinates in $\RR^3$). Therefore if we define $\varphi(r):= u_1(r,1)$ and $\zeta(r) := u_2(r,1)$ for $r \in \pi_1((0,R) \times S^1)-(C^{S} \cup C^{N}))$ and extend $\varphi$ to $\pi_1(C^{S} \cup C^{N})$ by setting $\varphi(\pi_1(C^{S})) = H$ and $\varphi(\pi_1(C^{N})) = 0$ we get the desired representation (\ref{phirep}).\\
To compute the degree of u, assume for the moment that $u$ is smooth and that $C^{N,S}$ have zero measures. Away from $C^{N,S}$ and in the $(r,\te)$ coordinates we have $u(r,\te) = (\varphi(r),\zeta(r)\cdot\te^{l}),$ so
\[u^*(dvol) = u^*(g(t)  dt d\beta) = lg(\varphi(r))\frac{d\varphi}{dr}drd\te.\]
It follows that
\begin{align*}
&deg(u) = \frac{1}{vol(S^2_{target})}\int_0^{2\pi} \int_0^{R} u^*dvol = \frac{1}{vol(S^2_{target})}\int_0^{2\pi}\int_0^{R}lg(\varphi(r))\frac{d\varphi}{dr}drd\te\\
&=\frac{2\pi l}{vol(S^2_{target})}[G(\varphi(R)) - G(\varphi(0))],
\end{align*}
where $G$ is an antiderivative for $g.$ But
\begin{align*}
&vol(S^2_{target}) = \int_0^{2\pi}\int_0^Hg(t)dtd\beta = 2\pi[G(H) - G(0)]
\end{align*}
and it follows that $deg(u)=\pm l$ or $0,$ depending on the images of $p$ and $q.$\\
For the general case we show that the subset of equivariant maps for which $C^{N,S}$ has zero measure is dense in the $H^1$ topology. For this, given a map $\varphi:[0,R]\rightarrow [0,H]$ and a small $\epsilon$ we need to find another map, $\psi:[0,R]\rightarrow[0,H]$ which satisfies the zero measure property above, and is $\epsilon$ close to $\varphi.$ We do this for the case where $\varphi(0)=0$ and $\varphi(R)=H,$ the other cases being similar. Let $\delta:[0,R] \rightarrow [0,H]$ be a smooth map which satisfies, $\delta(0)=0, ~\delta(R)=H, ~\delta(r)\leq \frac{H-\varphi(r)}{2},$ and $\delta(r)=\varphi(r)=0$ implies $r=0.$ We can now define $\psi$ as $\psi(r)=(1-\epsilon(R-r))\varphi(r)+\epsilon\delta(r).$ It follows that $\psi(0)=0$ and $\psi(R)=H,$ and that by making $\epsilon$ small $\psi$ can be made arbitrarily close to $\varphi.$ If for some $r,~ \psi(r)=0$ then we must have $\delta(r)=\varphi(r)=0$ which implies $r=0.$ If $\psi(r)=H,$ then we have $\epsilon\delta(r)=H-(1-(R-r))\varphi(r).$ But the left hand side of this equation is bounded above by $\epsilon\varphi(r)/2$ which implies a contradiction unless $r=R.$\\
Finally we relax the smoothness assumption on $u$ as follows. Note that being equivariant is the same thing as being invariant under a certain group action. Now given an equivariant $u,$ we approximate it by a sequence of smooth functions $u_n$. If we define $v_n := \frac{1}{vol~G}\int_Gg(u_n) dg$ then $v_n$ is smooth and invariant and since $u$ is also invariant we have
\begin{align*}
&\|u -v_n\|_{H^1} = \frac{1}{vol~G}\| \int_G g(u) - g(u_n) dg \|_{H^1}\\
&\leq \frac{1}{vol~G} \int_G \|g(u) - g(u_n)\|_{H^1} dg \rightarrow 0.
\end{align*}
This shows that we can approximate $u$ by smooth equivariant functions, and by the definition of the degree of an $H^1$ map we can conclude that our computations above are valid for any such map.\footnote{That the degree is $0,~l$or $-l$ can be seen topologically too. We present the argument for the case when $u$ sends $p$ to the north pole. The equivariance assumption implies that $u$ sends orbits (circles parallel to the equatorial circle) to orbits. Since we are assuming $u$ to be smooth we can find a regular value $P$ for it. Let $C$ be the orbit containing $p.$ Let $P_1$ be a point of highest altitude going to $P$ and let $C_1$ be the orbit containing it($P_1$ exists because $P$ is a regular value and therefore its pre-image is a finite number of points). Let $S$ and $S_1$ be small enough open squares around $P$ and $P_1$ respectively (with two sides parallel and two sides perpendicular to the equator), so that $u$ is a diffeomorphism from $S_1$ to $S.$ Since the north pole goes to the north pole, and $P_1$ is the closest point to the north pole that is mapped to $P,$ the upper edge of $S_1$ gets mapped to the upper edge of $S.$ Moreover by the regularity of $P$, possibly after making $S_1$ smaller, the lower edge of $S_1$ goes to the lower edge of of $S$. Otherwise the image of $S_1$ would have to fold on itself along the circle going through $P,$ so the image of $S_1$ would be a half-closed square, contradicting the fact that $u$ is a local diffeomorphism. This implies that the local degree of $u$ restricted to $S_1$ is one. Repeating the same argument for the other $l-1$ point on $C_1$ which get mapped to $P,$ we get a contribution of $l$ to the degree of $u.$ Now if there are any other points south of $P_1$ that go to $P,$ let $Q_1$ be one such point of highest altitude. By the same reasoning as above we see that a small square $T_1$ around $Q_1$ goes to a small square $T$ around $P,$ except that this time the southern edge of $T_1$ goes to the the northern edge of $T$ and vice versa. This means that the contribution of the $l$ pre-images of $P$ on the orbit through $Q_1$ to the degree of $u$ is $-l,$ making the total degree zero. Continuing in this way we conclude that the degree has to be either zero or $l.$}

\end{proof}

Define $X^0_l:=\{u\in H^1~ s.t.~u~satisfies ~(\ref{space})~a.e.\}.$ Note that if we restrict $\cGo$ to $u\in X^0_l$ of degree $l,$ the critical points still satisfy the Euler-Lagrange equations (\ref{uELgen}). The reason for this is that the degree map is continuous from $H^1$ to $\mathbb{N}$ and therefore the preimage of any integer is an open set. With this observation in mind we make the following definition\footnote{Note that $X_l$ here corresponds to $X_l^0$ and vice versa in \cite{ST1}.}
\begin{align*}
X_l := \{u \in X_l^0 ~ \mathrm{such}~ \mathrm{that}~ deg~u = l\} =\{u \in X_l^0 ~ \mathrm{such}~ \mathrm{that}~ u(p) = N,~u(q)= S\}.
\end{align*}
Since $u$ is only an $H^1$ map, to be precise we should write $\{u \in X_l^0 ~ \mathrm{such}~ \mathrm{that}~ \mathrm{if}~\mathrm{we}~\mathrm{define}~u(p) = N,~u(q)= S~\mathrm{then}~u~\mathrm{is}~\mathrm{continuous}\}.$ This is what we will always tacitly mean in the future. \\
From now on we will study the restriction of $\cGo$ to $X_{l}$ with $l>0.$\\
A standard computation shows that for $u \in X_l$
\[\frac{1}{2}\int_{S^2}|\nabla u|^2 dvol_{S^2}\geq lvol(S^2_{target}).\]
The right hand side is just the infimum of the functional for harmonic maps of degree $l$. It follows that
\[G^*:=\inf_{X_l}\cGo(u) \geq lvol(S^2_{target})-\sup_{u \in X_l}\frac{\omega^2}{2}\int_{S^2}|Au|^2 dvol_{S^2}.\]
Note that within each homotopy class of maps between spheres there exists a harmonic map which is an energy minimizer (see \cite{L}, page 64, theorem 8.4). It follows that the value of $\cGo$ at this function is strictly less than $lvol(S^2_{target}).$ Consequently
\[G^* < lvol(S^2_{target})\].
We are now ready to state the main theorem of this section.
\begin{theorem}\label{main}
For fixed $l$ and $\omega$ any minimizing sequence $\{u_n\}$ in $X_l$ for $\cGo$ has a subsequence which converges strongly in $H^1$ to a smooth map in $X_l.$
\end{theorem}
The first step in the proof is to restate the theorem in terms of the function $\varphi$ defined in Lemma 1. If we think of $u$ as a map to $\RR^3$ the representation (\ref{phirep}) can be written as
\[u(r,\theta) = (g(\varphi(r))\cos(\zeta(r)+l\te),g(\varphi(r))\sin(\zeta(r)+l\te),h(\varphi(r))).\]
Here $h$ and $g$ are as in the paragraph between equations (\ref{UEL}) and (\ref{UELgen}). To be more precise we should write $\arg(\zeta)+ l\arg(\te)$ instead of $\zeta + l\te,$ where $\arg$ is the argument function with values in $(0,2\pi).$ A computation then shows that
\begin{align*}
\cGo(u) = \pi\int_0^R[(\varphi^{'})^{2}+ (\frac{l^2}{f^2}-\omega^2)g^2(\varphi)]fdr + \pi\int_0^R(\zeta^{'})^{2}g^2(\varphi)fdr.
\end{align*}
Note that the second integral above is always positive. This implies that if we define $v(r,\te) := (\varphi(r),\te^l)$ then $v \in X_l$ and $\cGo(v) \leq \cGo(u).$ Therefore if we define
\[\Hlo(\varphi):=\pi\int_0^R[(\varphi^{'})^{2}+ (\frac{l^2}{f^2}-\omega^2)g^2(\varphi)]fdr,\]
the problem reduces to studying the minima of $\Hlo$ over the set of continuous $\varphi:[0,R]\rightarrow [0,H]$ satisfying the boundary conditions $\varphi(0) =0$ and $\varphi(R)= H.$ Note that $\omega =0$ corresponds to harmonic maps between spheres.\\
Suppose we can establish the sub-sequential convergence of minimizing sequences in $X_l.$ It will then follow from equation (\ref{uELgen}) and the ellipticity of the Laplacian that to prove the smoothness of the limit function it suffices to prove its Holder continuity. With this in mind we prove the following theorem, which is the key step in the proof of Theorem 1.
\begin{theorem}\label{key}
Fix $\omega \neq 0$ and let $\{\varphi_n\}$ be a minimizing sequence for $\Hlo$ consisting of continuous functions from $[0,R]$ to $[0,H]$ satisfying $\varphi_n(0) =0$ and $\varphi_n(R) = H.$ Then there is a subsequence of $\{\varphi_n\}$ that converges in $H^1$ to a $C^{1,\gamma}$ function $\varphi,$ for some $\gamma \in (0,1),$ with $\varphi(0) =0$ and $\varphi(R) = H.$
\end{theorem}

\begin{proof}
 By the bound on $\{\Hlo(\varphi_n)\}$ we have a uniform $L^2-$bound on $\{\varphi_n^{'}\}.$ It follows that $\{\varphi_n\}$ is uniformly bounded in $H^1((0,R))$ and by the Sobolev inequalities also in $C^{0,1/2}(V)$ for any $V \subset \subset (0,R)$ \footnote{Note that by the $H^1$ norm of a function we mean the $L^2$ norm of the function and its derivative with respect to the measure $fdr.$ Since $f$ vanishes at the endpoints of $(0,R)$ we have to exclude the end points when we want to apply the Sobolev inequalities. This is also the reason why we have to consider the three cases below for the limiting function $\varphi.$} By Arzela-Ascoli we can find a continuous $\varphi$ such that $\varphi_n \rightarrow \varphi$ uniformly on compact intervals, along some subsequence which we continue to denote by $\{\varphi_n\}.$  Boundedness of $\{\varphi_n\}$ in $H^1$ also implies that
\[\varphi_n \stackrel{H^1}{\rightharpoonup} \varphi.\]
Lower semicontinuity of weak limits gives
\[\Hlo(\varphi) \leq d:=\inf_{\stackrel{\psi(0)=0}{\psi(R)=H}}\Hlo(\psi).\]
Indeed, to see that $\int_0^R(\frac{l^2}{f^2}-\omega^2)g^2(\varphi)fdr \leq \lim \int_0^R(\frac{l^2}{f^2}-\omega^2)g^2(\varphi_n)fdr$ we write the integral as $\int_{\frac{R}{2}-M}^{\frac{R}{2}+M} + \int_{[0,R]-[\frac{R}{2}-M,\frac{R}{2}+M]}$ where $M$ is chosen so large that $\frac{l^2}{f^2}-\omega^2 \geq 0$ for $s \not\in [\frac{R}{2}-M,\frac{R}{2}+M].$  For the first integral we use the uniform convergence of $\varphi_n$ and for the second we use Fatou's lemma. In particular $\Hlo(\varphi) < \infty$ and therefore $\varphi(0)$ and $\varphi(R)$ must be $0$ or $H.$\\
There are three possibilities for the limit $\varphi$:\\
(i)$\varphi(0)=0$ and $\varphi(R)=H$\\
(ii)$\varphi \equiv 0$ or $\varphi \equiv H$\\
(iii)$\varphi(0) = \varphi(R)=0$ but $\varphi$ is not constantly $0$ or the same statement with $0$ replaced by $H$.\\
In case (i) since $\varphi$ satisfies the correct boundary conditions we must have $\Hlo(\varphi) = d = \lim_{n\rightarrow \infty}\Hlo(\varphi_n).$ But as indicated above $\int_0^R(\frac{l^2}{f^2}-\omega^2)g^2(\varphi)fdr \leq \lim \int_0^R(\frac{l^2}{f^2}-\omega^2)g^2(\varphi_n)fdr$ so we must have $\|\varphi^{'}\|_{L^2}\geq {\lim}\|\varphi_n^{'}\|_{L^2}$ and thus $\varphi_n^{'} \stackrel{L^2}{\rightarrow} \varphi^{'}.$ $L^2$ convergence of $\{\varphi_n\}$ follows from the poitwise convergence and the dominated convergence theorem, and hence $\varphi_n \stackrel{H^1}{\rightarrow} \varphi.$\\
Next, we eliminate cases (ii) and (iii). Because $\varphi$ appears only in the forms $g^2(\varphi)$ and $(\frac{d\varphi}{d\alpha})^2$ it suffices to consider (ii) and (iii) with $\varphi \equiv 0,$ etc. Suppose we are in case (ii). Since the infimum of $\Hlo$ is strictly less than $l\mathrm{vol}(S^2_{target}),$ we will have reached a contradiction if we show that
\[ \lim_{n \rightarrow \infty} \Hlo(\varphi_n) \geq l\mathrm{vol}(S^2_{target}).  \]
Since $\Hlo(\varphi_n) = \mathcal{H}_{l,0}(\varphi_n) - \int_{0}^R \omega^2g^2(\varphi_n)fdr,$ and since $l\mathrm{vol}(S^2_{target})$ is a lower bound for $\Hlo,$ it suffices to show that
\begin{align*}
\lim_{n \rightarrow \infty} \int_0^R g^2(\varphi_n)fdr = 0.
\end{align*}
Since $g^2(\varphi_n)f$ is up to a constant bounded by the integrable function $f$, the above equality follows from the dominated convergence theorem. This shows that (ii) is not a possibility.\\

The proof that (iii) is not a possibility is more technical and therefore for the sake of exposition, we isolate it in a lemma.
\begin{lemma}
 (iii) is not a possibility.
\end{lemma}
 \begin{proof}(of lemma 2)
 We assume that $\varphi(0) = \varphi(R) = 0$ and reach a contradiction. The case where $\varphi(0) = \varphi(R) = H$ is treated similarly. We will reach the desired contradiction by modifying each $\varphi_n$ in a way that $\cGo(\varphi_n)$ is reduced by an amount independent of $n.$ This will contradict the minimizing property of $\{\varphi_n\}.$ It is convenient to introduce a change of variables here. Let $s = s(r) = \int_\frac{R}{2}^r\frac{dt}{f(t)},$ so that $s$ goes from $-\infty$ to $\infty$ as $r$ goes from $0$ to $R.$
 In terms of this new variable the $\Hlo$ functional becomes
\[\Hlo(\varphi) = \int_{-\infty}^{\infty}[(\varphi^\prime(s))^2+(l^2 - \omega^2f^2(s)g^2(\varphi(s)))]ds.\]
To be precise we should write $r(s)$ instead of $s$ as the argument of the functions involved, however for simplicity of notation we ignore this issue and think of our functions as defined on all of $\RR$ (so for instance $f(s)$ really means $f(r(s))$). Keep in mind that our expansion for $f$ shows that $f$ is a decreasing function for $r$ near $R,$ or equivalently for large $s.$ We divide the proof into cases:\\\\
Case $1$: Suppose there exist intervals of the form $[a,\infty)$ where $\varphi\equiv0$ identically, and let $S_1$ be the smallest such $a.$\\\\
Case $1a$: $f(S_1)\not= l/\omega:$\\\\
Let $S<S_1$ be such that $\varphi(S)>0,$ $\varphi(S)=\max_{[S,S_1]}\varphi,$ $g$ is increasing on $[0,2\varphi(S)],$ and $l^2-\omega^2f^2$ doesn't change sign in $[S,S_1].$ Let $I:=\{x>S_1~s.t.~ l^2-\omega^2f^2\leq 0\},$ and let $b:=\max_I|l^2-\omega^2f^2|.$ Choose $\delta_1$ so small that $\delta_1|I|b<c\varphi^2(S)$ for a small $c$ to be chosen. Let $T_1>S_1$ be such that $l^2-\omega^2f^2>0$ and $f$ is decreasing on $[T_1,\infty)$ and choose $T>T_1$ such that $f(T)<f(x)$ for all $x \in [S,T_1].$ Choose $\delta_2\leq \varphi^2(S)/4$ so small that $g$ is increasing and smaller than $\delta_1$ on $[0,2\delta_2].$ Now let $N$ ne so large that for all $n\geq N,$ $|\varphi_n-\varphi|<\delta_2$ on $[S-1000,T+1000].$ Finally for each $n\geq N$ choose $S_n \in [S,S_1]$ such that $\varphi_n(S_n)=\max_{[S,S_1]}\varphi_n$ and $T_n>T$ such that $\varphi_n(S_n)=\varphi_n(T_n).$ \\\\
Case $1a^\prime$: $l^2-\omega^2f^2>0$ on $[S,S_1]:$\\\\
Define
\begin{align*}
\overline{\varphi_n}(x):=\left\{\begin{array}{rcl}\varphi_n(x)& \mbox{for}~ x\leq S_n\\ \varphi(x-S_n+T_n)& \mbox{for}~ x >S_n \end{array}\right\}.
\end{align*}
Then
\begin{align*}
&\int_{S_n}^\infty[{\overline{\varphi_n}^{\prime}}^2(x) + (l^2-\omega^2f^2(x))g^2(\overline{\varphi_n}(x))]dx\\
&= \int_{T_n}^{\infty}[{\varphi_n^{\prime}}^2(y) + (l^2-\omega^2f^2(y-T_n+S_n))g^2(\varphi_n(y))]dy\\
&\leq\int_{T_n}^{\infty}[{\varphi_n^{\prime}}^2(y) + (l^2-\omega^2f^2(y))g^2(\varphi_n(y))]dy.
\end{align*}
Therefore
\begin{align*}
&\Hlo(\varphi_n)-\Hlo(\overline{\varphi_n})\geq\int_{S_n}^{T_n}[{\varphi_n^{\prime}}^2(y) + (l^2-\omega^2f^2(y))g^2(\varphi_n(y))]dy\\ &\geq \int_{S_n}^{S_1}\varphi_n^2dx-b|I|g(\delta_2)\geq \frac{|\varphi_n(S_n)-\varphi_n(S_1)|^2}{S-S_n}-c\varphi^2(S)\\
&\geq\frac{|\varphi(S)-2\delta_2|^2}{S-S_1}-c\varphi^2(S)\geq(\frac{1}{4(S-S_1)}-c)\varphi^2(S)>C
\end{align*}
for some $C$ independent of $n,$ if we choose $c$ small enough.\\\\
Case $1a^{\prime\prime}$: $l^2-\omega^2f^2<0$ on $[S,S_1]:$\\\\
Define
\begin{align*}
\overline{\varphi_n}(x):=\left\{\begin{array}{rcl}\varphi_n(x)& \mbox{for}~ x\leq S_n\\ \varphi(S_n)& \mbox{for} ~S_n<x\leq S_1\\ \varphi(x-S_1+T_n)& \mbox{for}~ x>S_1 \end{array}\right\}.
\end{align*}
Since $(l^2-\omega^2f^2)g^2(\overline{\varphi_n})\geq(l^2-\omega^2f^2)g^2(\varphi_n)$ on $[S_n,S_1],$ by the same argument as before
\begin{align*}
&\Hlo(\varphi_n)-\Hlo(\overline{\varphi_n})\geq \int_{S_n}^{S_1}{\varphi_n^\prime}^2dx+\int_{S_1}^{T_n}[{\varphi_n^\prime}^2+(l^2-\omega^2f^2)g^2(\varphi_n)]dx>C.
\end{align*}
\\
Case $1b$: $f(S_1)=l/\omega$:\\\\
Define the constants used in case $1a$ with the following modifications: The requirement that the sign of $l^2-\omega^2f^2$ be constant on $[S,S_1]$ is replaced by $|l^2-\omega^2f^2|<c^\prime$ on $[S,S_1],$ for a small $c^\prime$ to be chosen. Note that since the curvature of the target is positive near the poles, we can choose $S$ so small that $g(x)\leq x$ on $[0,2\varphi(S)]$ (see the expansion for $g$ near $t=0$). Moreover, we require that $S_1-S<1.$ Now define $\overline{\varphi_n}$ in the same way as in case $1a^{\prime}.$ Note that
\begin{align*}
-\int_{S_n}^{S_1}|l^2-\omega^2f^2|g^2(\varphi_n)dx>-c^\prime\int_{S_n}^{S_1}\varphi_n^2dx\geq \frac{-5c^\prime}{4} \varphi^2(S).
\end{align*}
It follows by the same argument as before that
\begin{align*}
&\Hlo(\varphi_n)-\Hlo(\overline{\varphi_n})\geq (\frac{1}{4(S-S_1)}-c-\frac{5c^\prime}{4})\varphi^2(S),
\end{align*}
which is bounded below by a positive constant independent of $n$ if we choose $c$ and $c^\prime$ small enough.\\\\
Case $1b$: there is no $a$ such that $\varphi \equiv 0$ on $[a,\infty)$:\\\\
Choose $S_1$ so large that $f$ is increasing and $l^2-\omega^2f^2$ is positive on $[S_1,\infty),$ and such that $\varphi(S)\neq 0.$ Let $T_1>S_1$ be so that $\varphi(T_1)<\varphi(S_1)/20000.$ Choose $N$ so large that $|\varphi_n-\varphi|\leq \varphi(S_1)/1000$ in $[S_1-1000,T_1+1000]$ for $n\geq N,$ and for each such $n$ let $T_n>T_1$ be such that $\varphi(T_n)=\varphi(S_1).$ If we define $\overline{\varphi_n}$ in the same way as in case $1a^\prime$ with $S_n$ replaced by $S_1,$ it follows that $\Hlo(\varphi_n)-\Hlo(\overline{\varphi_n})\geq C$ for a constant $C$ independent of $n.$
\end{proof}
Next we obtain $C^{1,\gamma}$ bounds for $\varphi.$ Note that we are ultimately interested in showing that $u$ defined by $(\varphi(r), \te^l)$ is differentiable and for this it suffices to prove Holder estimates for $\frac{\varphi(r)}{r^l}$ (and for $\frac{\varphi(r)}{(R-r)^l}$ which will follow from the symmertry of the arguments to follow).To see this, on both spheres consider the charts identifying a neighborhood of the north pole with a ball in $\mathbb{C}.$ With $w$ denoting the complex coordinate $u$ is then given by $u(w) = \varphi(|w|)\frac{w^l}{|w|^l},$ which shows that we need to prove Holder continuity of $\frac{\varphi(r)}{r^l}.$ It follows from the expansion for $f$ that we may instead prove our Holder estimates for $\frac{\varphi}{f^l},$ and this is what we aim for.\\
Holder continuity of $u$ away from the poles proves Holder continuity of $\varphi$ away from the end-points, but the latter also follows from the elliptic Euler-Lagrange equations satisfied by $\varphi$:
\begin{equation}\label{phieq}
\varphi^{''} + \frac{f^{'}}{f}\varphi^{'} + (\omega^2 - \frac{l^2}{f^2})g(\varphi)g^{'}(\varphi) = 0.
\end{equation}

Ellpitic regularity in fact proves smoothnes of $\varphi$ away from the endpoints. Regularity at the end-points is more subtle. For this we think of $\frac{\varphi}{f^l}$ as a radial function defined on $B_\rho-\{0\},$ where $B_\rho$ is the ball of radius $\rho,$ to be determined later, in $\RR^{2l+2}.$  We equip $B_\rho - \{0\}$ with the metric $\mathrm{diag}(1,f^2(r),\cdots,f^2(r)).$ The motivation for this is that $\varphi$ already came from a function defined on $S^2.$ As will become clear momentarily, raising the dimension allows us to use appropriate Sobolev inequalities to deal with the singularity in $\frac{\varphi}{f^l}.$ We use $\RR^{2l+2}$ rather than $S^{2l+2}$ because smoothness is a local issue. With $\Delta$ denoting the Laplacian of our new modified metric on $B_{\rho}-\{0\},~\psi$ satisfies the equation

\begin{equation}\label{psieq}
\Delta \psi = \frac{l^2}{f^{l+2}}(g(\varphi)g^{\prime}(\varphi)-\varphi)+\frac{1}{f^l}((\frac{l^2(1-{f^{\prime}}^2)}{f^2}-\frac{lf^{''}}{f})\varphi - \omega^2 g(\varphi)g^{\prime}(\varphi)) := R(\psi)
\end{equation}

Thanks to the Sobolev embeddings in order to show that $\psi \in C^{0,\gamma}$ for some $\gamma,$ it suffices to show $\psi \in W^{2,p}$ with $p>l+1.$ The key point is that the expansion (\ref{metric}) for $f$ and the corresponding expansion for $g$ show that
\begin{equation}\label{errorbound}
R(\psi)\lesssim \frac{|\varphi|^3}{f^{l+2}}+|\psi| \lesssim |\psi|^{1+\frac{2}{l}} + |\psi|.
\end{equation}
Suppose we could prove that $R(\psi) \in L^{p_0}$ for some $p_0.$ Then by elliptic regularity $\psi \in W^{2,p_0} \hookrightarrow L^{q_0}$ where $\frac{1}{p_0} - \frac{1}{q_0} = \frac{1}{l+1}.$ By (\ref{errorbound}) this implies $R(\psi) \in L^{p_1}$ with $p_1 = lq/(2+l).$ We can continue recursively to get $R(\psi) \in L^{p_n}$ with
\begin{align*}
&\frac{1}{p_n}= \frac{l+2}{l}(\frac{1}{p_{n-1}}-\frac{1}{l+1}) = \cdots\\
&=(\frac{l+2}{l})^n(\frac{1}{p_0} -\frac{l+2}{2l+2}) + \frac{l+2}{2l+2}.
\end{align*}
If $p_0 > (2l+2)/(l+2)$ this means that $p_n$ can be made larger than $l+1$ which in turn, coupled with elliptic regularity, implies that $\psi \in W^{2,p_n}.$ The proof of the $c^{0,\gamma}$ bounds will therefore be complete if we can show that $R(\psi) \in L^{p_0}(B_\rho - \{0\}),~ p_0 = (2l+3)/(l+2).$ The factor of $f^{2l+1}$ in the integrals below comes from the volume form on $B_\rho - \{0\}$:
\begin{align*}
&\|R(\psi)\|_{L^{p_0}}^{p_0} \lesssim \int_0^\rho |\psi|^{p_0(\frac{l+2}{l})}f^{2l+1}dr\\
&= \int_0^\rho \frac{|\varphi|^{2+3/l}}{f^2}dr \lesssim \int_0^\rho \frac{|\varphi|^2}{f^2}dr.
\end{align*}
To bound this last integral we multiply (\ref{phieq}) by $\varphi$ and integrate on $[a,\rho],~0<a<\rho:$
\begin{align*}
&\int_a^\rho({\varphi^{'}}^2-\frac{f^{'}}{f}\varphi\varphi^{'}+(\frac{l^2}{f^2}-\omega^2)\varphi g(\varphi)g^{'}(\varphi))dr\\
& = \int_a^\rho(\varphi^{''}\varphi + {\varphi^{'}}^2)dr = \frac{1}{2}((\varphi^2)^{'}(\rho)-(\varphi^2)^{'}(a)).
\end{align*}
Since $\varphi$ is not constantly zero near the end-points (otherwise there is nothing to prove) we can find a sequence of small numbers $\{a_j\}$ tending to zero so that $\varphi^{'}(a_j) > 0$ for all $j.$ It follows that if we choose $\rho$ small enough, and upon replacinf $a$ by $a_j,$ the right hand side of the equation above becomes bounded by $C_1(\rho)$ for some constant $C_1$ independent of $a_j.$ We also have
\begin{align*}
{\varphi^{'}}^2-\frac{f^{'}}{f}\varphi\varphi^{'} = (\varphi^{'} -\frac{f^{'}}{2f}\varphi)^2 -\frac{{f^{'}}^2}{4f^2}\varphi^2 \geq -\frac{{f^{'}}^2}{4f^2}\varphi^2.
\end{align*}
Using the expansions for $g$ and $f$ we see that after possibly making $\rho$ even smaller
\begin{align*}
4l^2\varphi g(\varphi)g^{'}(\varphi) -{f^{'}}^2\varphi^2 = \varphi^2[(4l^2-f^2) + O(\varphi)] \geq C_2 \varphi^2.
\end{align*}
Putting the pieces together we have
\begin{align*}
\int_{a_j}^\rho \frac{\varphi^2}{f^2}dr \leq C_3 + \omega^2 \int_0^\rho \varphi g(\varphi)g^{'}(\varphi)dr \leq C(\rho).
\end{align*}
Sending $j$ to infinity we can conclude, using Fatou's Lemma, that
\begin{align*}
\int_{0}^\rho \frac{\varphi^2}{f^2}dr < \infty.
\end{align*}
as was to be shown. We have just obtained $C^{0,\gamma}$ bounds on $\psi.$ But once again the expansions for $f$ and $g$ imply that $R(\psi)$ is also Holder continuous and applying ellpitic regularity to equation (\ref{psieq}) we get the desired $C^{1,\gamma}$ estimates.
\end{proof}
We are now ready for the proof of Theorem 1:
\begin{proof} (of Theorem 1)
Let $\{u_n\}$ be a minimizng sequence for $\cGo$ in $X^l.$ There exists $u \in H^1(S^2)$ such that we can extract a subsequence of $\{u_n\},$ again denoted by $\{u_n\},$ satisfying
\[\nabla u_n \stackrel{L^2}{\rightharpoonup} \nabla u,~u_n\stackrel{L^2}{\rightarrow}u,~ u_n \stackrel{a.e.}{\rightarrow} u.\]
The almost-everywhere convergence follows from Arzela-Ascoli, because boundedness in $H^1$ and Lemma 1 provide us with locally uniform Holder bounds on $\{u_n\}$ away from the poles. $L^2$convergence follows as a consequence. Since $|u_n| \lesssim 1$ and by lower semicontinuity of norms $\cGo(u) \leq \lim \cGo(u_n) = G^*.$\\
$u$ is equivariant because for any $g \in SO(3)$
\[|u(gx)-g^lu(x)| \leq|u(gx)-u_n(gx)| + |g^l(u_n(x)-u(x))|.\]
which converges to zero. We can therefore find functions $\varphi$ and $\zeta$ as before such that $u(r,\te) = (\varphi(r),\zeta(r)\te^l).$ Also since $u_n \in X^l$ we can write $u_n(r,\te) = (\varphi_n(r),\zeta_n(r)\te^l)$ with $\varphi_n(0) = 0,~ \varphi_n(R)=H$ for all $n.$ Since $\{\varphi_n\}$ is minimizing, by Theorem 2 it converges strongly to a $C^{1,\gamma}$ function along some subsequence (for the same $\gamma$ as in the Theorem 2), which by uniqueness of limits has to be $\varphi.$  By our earlier degree computation $u$ has degree $l$ and hence belongs to $X^l.$ It follows that $\cGo(u) = G^*$ whence $\nabla u_n \stackrel{L^2}{\rightarrow} \nabla u$ and the $H^1$ convergence is strong. Finally if $\zeta$ is not constant we can reduce $\cGo(u)$ by replacing $\zeta$ by $1.$  It follows that $u(r, \te) = (\varphi(r),\te^l),$ which is $C^{1,\gamma}$ as argued in Theorem 2. Smoothness of $u$ follows from elliptic regularity applied to equation (\ref{uELgen}).
\end{proof} 
\section{Holder Continuity}
\begin{center}

\end{center}

\vspace*{0.5cm}

In this section and the next, we prove that the Cauchy problem for rotationally equivariant wave maps from $U: S^2\times \RR \rightarrow S^2$ (where the metrics on the spheres are as in the previous section) with smooth initial data of small energy has globally smooth solutions. The small time existence of solutions follows from the standard theory of nonlinear waves, so we need to prove that a solution can be smoothly extended to all times. In the current section we prove uniform Holder bounds on the solution which allow its extension as a Holder continuous map. In the next section we prove higher regularity. The Cauchy problem for $U$ is
\begin{equation}\label{cauchyproblem}
\left\{\begin{array}{rcl}\Box U = \sum \gamma^{ij}B(U)(\partial_i U, \partial_j U)) =: Q(U)\\ (U(x,t_0),U_t(x,t_0))=(U_0(x),U_1(x))\end{array}\right\},
\end{equation}
where $U_j$ satisfy the equivariance ansatz $U_j(r,\theta)=\theta^l\cdot U_j(r,0)$ and $\gamma$ is the metric on the domain. We want to prove the following
\begin{theorem}
There exists an $\epsilon >0$ such that if the energy of $(U_0,U_1)$ is smaller than $\epsilon$ then the Cauchy problem (\ref{cauchyproblem}) has a unique smooth solution which exists for all times.
\end{theorem}
As in \cite{ST1} and \cite{CT} we may assume that the initial data are given at time $t_0 =-1$ and that a smooth solution $U$ exists up to time $t=0.$ Using the rotational symmetry, we may also assume as in the aforementioned papers, that the first possible blow happens at $r=0.$ It follows that we can conclude the statement of the theorem if we can find uniform bounds on $U$ and its derivatives near the origin $(t,r)=(0,0).$ This and finite speed of propagation allow us to assume that we are working in a neighborhood of the origin in $\RR^2 \times \RR$ equipped with a non-standard metric. We shall occasionally refer to the origin as the north pole. The rest of this section will be devoted to finding uniform Holder bounds on $U.$\\
To this end, we use the integral representation derived in \cite{F} for the solution to the wave equation on a curved space-time. To be able to use this representation the domain has to be causal in the terminology of \cite{F}. A domain $D$ is called causal if\\
(1) D is contained in a geodesically convex domain $D_0$ such that\\
(2) with $J^{+}$ and $J^{-}$ denoting the forward and backward light cones respectively, $J^{+}(p)\cap J^{-}(q)$ is compact for any $p$ and $q$ in $D_0.$\\
It is proved in theorem 4.4.1 in \cite{F} (p. 147) that every point is contained in such a neighborhood (see \cite{F} section $4.4$ for further explanation). We may therefore assume that we are working in a neighborhood $V:=\Omega \times (T,0)$ which satisfies the desired convexity conditions. Let $K > \max_{p \in S^2}k(p).$  We may take $\Omega$ to be such that it contains a small ball around $r=0$ and so small that it has non-negative curvature and that any geodesic triangle in $\Omega$ has a comparison triangle in $S^2(\frac{1}{\sqrt{2K}}).$ For simplicity of notation we set $T=-1.$



\begin{remark}

By a comparison triangle we mean a triangle whose sides have the same lengths as the sides of the original triangle.The point is that we want to use triangle comparison theorems to obtain estimates on the fundamental solution, but the sphere of curvature $K$ may be very small. Therefore we need to work in a small neighborhood of the north pole where the comparison triangles fit on a small sphere in $\RR^3$. To be precise we use two kinds of comparison triangles. The first is when we take a triangle in a simply connected space of constant curvature, whose three sides have the same lengths as the sides of the original triangle. The other kind is when in the triangle in the constant curvature space, the two sides emanating from the pole (or the origin when the comparison space is $\RR^2$) have the same lengths as the corresponding sides in the original triangle and also they meet at the same angle in the two triangles. The corresponding comparison results are: 1. (\cite{B}, pp. 121-122) Suppose $p,q,r$ are three points in $M$ and that the curvature of $M$ is bounded below by $\delta$ (respectively above by $\Delta$). If we take three points $p^\p,q^\p,r^\p$ in the model space (constant curvature, simply connected) $M_{\delta}$ (respectively $M_\Delta$) of curvature $\delta$ (respectively $\Delta$) such that the triangle between them lies in a normal chart and $d(p,q)=d(p^\p,q^\p),~ d(p,r)=d(p^\p,r^\p)$ and the angles $A=A^\p$ at $p$ and $p^\p$ then $d_M(q,r)\leq d_{M(\delta)}(q^\p,r^\p)$ (respectively $d_M(q,r)\geq d_{M(\Delta)}(q^\p,r^\p)$). 2. (\cite{B}, pp. 121-122 or \cite{K}, section 2.7) Under the same assumptions as before, except that if instead of assuming $A=A^\p$ we assume $d(q,r)=d(q^\p,r^\p)$ then the angles in $M$ are greater than (less than) or equal to the respective angles in $M_{\delta}$ ($M_{\Delta}$).

\end{remark}

\begin{remark}
The expansion for $f$ near the origin in the previous section implies in particular that $f \sim r$ near $r=0.$ This will be used throughout this work without further mention, so we will often use $f$ and $r$ interchangeably in our estimates.
\end{remark}

Following \cite{F} we denote by $J^{-}(P)$ the backward light cone through $P \in V.$ With $(r,\theta, t)$ polar coordinates for $P$ on $S^2 \times \RR$ and $d$ denoting the geodesic distance function, Friedlander's representation of the solution is (up to a factor of $2\pi$)

\begin{align*}
U(r,\theta, t) = \int_{J^{-}(p)}\frac{W^{+}(r,\theta,t,r^{'},\theta^{'},t^{'})Q(U(r^{'},\theta^{'},t^{'}))}{\sqrt{(t-t^{'})^{2}- d^{2}((r,\theta),(r^{'},\theta^{'}))}}d\mu(r^{'},\theta^{'},t^{'}).
\end{align*}

Here $Q$ denotes the nonlinear term in the wave maps equation, $W^+$ is defined in the same way as in in \cite{F} and satisfies the asymptotic behavior obtained there (we only use the fact that $W^{+}$ is smooth, but see \cite{F}, p. 246 for the asymptotic expansion and further explanation). See \cite{F}, p. 249, equation (6.3.20) and the explanation following it for the derivation of this representation. See also theorems 6.3.1 and 6.2.1 in \cite{F}.\\

Now $J^{-}(P)$ is

\begin{align*}
&\{(r^{'},\theta^{'}, t^{'})| (t-t^{'})^2 \geq d^2((r,\theta),(r^{'},\theta^{'}))\} \\
&= \{(r^{'},\theta^{'}, t^{'})| (r^{'},t^{'}) \in K(r,t),~ \mathrm{and}~\theta^\p ~\mathrm{satisfies}~ (t-t^{'})^2 \geq d^2((r,\theta),(r^{'},\theta^{'}))\},
\end{align*}

where $K(r,t)$ is the backward light cone:

\begin{align*}
\{(r^{'},t^{'}) | t^{'} \leq t,~ r^{'} \geq 0,~ t-t^{'} \geq |r-r^{'}|\}.
\end{align*}

Friedlander's representation becomes

\begin{equation}\label{Urepresentation}
\int_{K(r,t)}\int_{\{\theta^{'}|(t-t^{'})^2 \geq d^2((r,\theta),(r^{'},\theta^{'}))\}}\frac{W^+Q(U)f(r^{'})}{\sqrt{(t-t^{'})^2- d^2((r,\theta),(r^{'},\theta^{'}))}}d\theta^{'}dr^{'}dt^{'}.
\end{equation}

Here $f$ is the function defining the metric on $S^2,~ds^2 = dr^2 + f^2(r)d\theta^2.$ Introducing characteristic coordinates

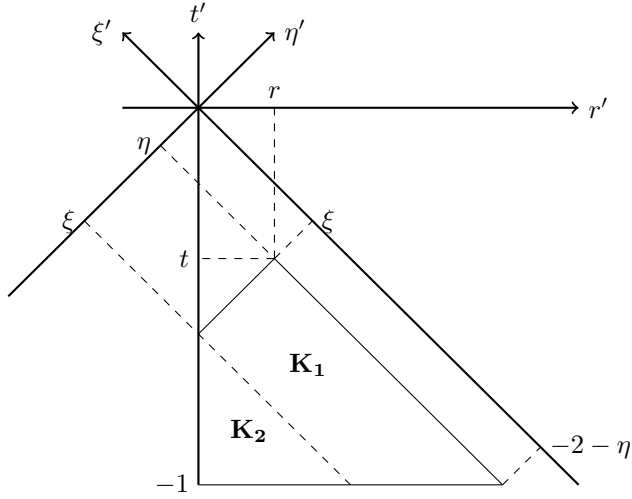
\begin{align*}
\xi = t-r,~ \eta = t+r
\end{align*}
\\
\begin{center}
\begin{tikzpicture}[scale=1]
\draw [<->,thick] (0,1) node (yaxis) [above] {$t^{\prime}$}
        |- (5,0) node (xaxis) [right] {$r^{\prime}$};
    \draw [thick](0,0)  -- (0,-5) ;
    \draw [thick](0,0)  -- (-1,0) ;
    \draw (1,-2)  -- (4,-5) ;
    \draw (1,-2)  -- (0,-3) ;
    \draw (0,-5) node[left] {$-1$}  -- (4,-5) ;
    \draw [dashed](-3/2,-3/2) node [left] {$\xi$}  -- (2,-5) ;
    \draw [dashed](-1/2,-1/2) node [left] {$\eta$}  -- (1,-2) ;
    \draw [dashed](1,-2) -- (3/2,-3/2) node [right] {$\xi$} ;
    \draw [dashed](4,-5) -- (9/2,-9/2) node [right] {$-2-\eta$} ;
    \draw [dashed](1,-2) -- (1,0) node [above] {$r$} ;
    \draw [dashed](1,-2) node [below =40, right=2] {$\mathbf{K_1}$}node [below =65, left = -1] {$\mathbf{K_2}$}-- (0,-2) node [left] {$t$} ;
    \draw [->,thick](-2.5,-2.5) --(1,1) node [right] {$\eta^{\prime}$};
    \draw [->,thick] (5,-5)-- (-1,1) node [left] {$\xi^{\prime}$};
\end{tikzpicture}
\\
\end{center}
as in \cite{ST1}, (\ref{Urepresentation}) becomes

\begin{align*}
\int_{\xi}^{\eta}\int_{-2-\eta^\p}^{\xi}\int_{\{\theta^{'}|(t-t^{'})^2 \geq d^2((r,\theta),(r^{'},\theta^{'}))\}}\frac{W^+Q(U)f(r^{'})}{\sqrt{(t-t^{'})^2- d^2((r,\theta),(r^{'},\theta^{'}))}}d\theta^{'}d\xi^\p d\eta^\p.
\end{align*}

By the equivariance assumption $U(r,\theta,t) = \theta^l \cdot U(r,0,t),$ so from now on we shall write $u(r,t)$ for $U(r,0,t)$ and $w^{+}(r,t,r^{'},t^{'},\theta^{'})$ for $W^{+}(r,0,t,r^{'},\theta^{'},t^{'})$. With this notation we have the following representation for $u(\xi,\eta)$ \footnote{From here on we should set $\theta = 0$ but we omit this point for aesthetic reasons.}
\begin{equation}\label{urepresentation}
\int_{\xi}^{\eta}\int_{-2-\eta^{'}}^{\xi}\int_{\{\theta^{'}|(t-t^{'})^2 \geq d^2((r,\theta),(r^{'},\theta^{'}))\}}\frac{w^{+}Q(u)f(r^{'})}{\sqrt{(t-t^{'})^2- d^2((r,\theta),(r^{'},\theta^{'}))}}d\theta^{'}d\xi^\p d\eta^\p.
\end{equation}

Since we want to obtain Holder estimates we will try to bound the derivatives of $u$ with respect to the characteristic coordinates near the north pole. We start by giving an overview of the plan of the proof, which is the same as in \cite{ST1}. With $e$ and $m$ denoting the energy and the momentum respectively, define
\begin{align*}
\A^2 := f(r)(e+m)= \frac{f}{2}(|\partial_\eta u|^2 + \frac{l^2}{f^2}|Au|^2),\\
\B^2 := f(r)(e-m)= \frac{f}{2}(|\partial_\xi u|^2 + \frac{l^2}{f^2}|Au|^2).
\end{align*}

For $\tilde{t} \leq 0$ and a fixed $\delta \in (0,1/2)$ define
\begin{align*}
&Z(\tilde{t}) = \{(t^\p,r^\p)| -1 \leq t^\p \leq \tilde{t},~0\leq r^\p \leq \tilde{t} -t^\p\}\\
&\X(\tilde{t}) = \sup_{Z(\tilde{t})}\big\{(f(r))^{(\frac{1}{2}-\delta)}\A(t,r)\big\}\\
&\X = \X(0).
\end{align*}
Our aim will be to prove that $\X$ is bounded provided the bound $\epsilon$ on energy is small enough. This will be accomplished by showing that $\X(t) \lesssim 1+ \epsilon\X(t)$ with a constant independent of $t.$ From this we can deduce Holder continuity as follows. For $\xi^\p=\xi=$constant
\begin{align*}
&|u(\xi, \eta_2)-u(\xi,\eta_1)|=\big|\int_{\eta_1}^{\eta_2}u_\eta(\xi,\eta^\p)d\eta^\p\big|\lesssim \X\int_{\eta_1}^{\eta_2}|\xi-\eta^\p|^{\delta-1}d\eta^\p\\
&\lesssim |r_2-r_1|^\delta.
\end{align*}
The $\eta^\p=\eta=$constant characteristic is a bit more involved. Using the boundedness of $\X$
\begin{align*}
&\int_0^r\B^2(t+r-r^\p,r^\p)dr^\p=\int_\xi^\eta\B^2(\eta,\xi^\p)d\xi^\p= \int_\xi^\eta\A^2(\eta^\p,\xi)d\eta^\p\\
&\lesssim\int_\xi^\eta|\eta^\p-\xi|^{2\delta-1}d\eta^\p \lesssim r^{2\delta}
\end{align*}
With $q:=u_\theta\cdot u_t$ denoting the charge density, the second inequality above comes from applying the divergence theorem to the divergence free vector-field $T=(e,-m,q)$ over the region $\{(\eta^\p,\xi^\p,\theta^\p)|\xi\leq\eta^\p,~\xi^\p\leq\eta\}$ in the backward light cone through the origin. Integrating by parts we see that for any $\delta^\p \in (0,\delta)$
\begin{align*}
&\int_{\xi_1}^{\xi_2}|\xi^\p-\eta|^{-2\delta^\p}\B^2(\eta,\xi^\p)d\xi^\p\sim\int_{r_1}^{r_2}r^{-2\delta^\p}\B^2(t_1+r_1-r,r)dr\\
&=\int_{r_1}^{r_2}r^{-2\delta^\p}d\big(\int_{r_1}^{r}\B^2(t_1+r_1-r^{\p\p},r^{\p\p})dr^{\p\p}\big)\\
&=r_2^{-2\delta^\p}\int_{r_1}^{r_2}\B^2(t_1+r_1-r^{\p\p},r^{\p\p})dr^{\p\p}\\
&+2\delta^\p\int_{r_1}^{r_2}r^{-1-2\delta^\p}\big(\int_{r_1}^{r}\B^2(t_1+r_1-r^{\p\p},r^{\p\p})dr^{\p\p}\big)dr\\
\end{align*}
\begin{equation}\label{fluxbound}
\lesssim r_2^{2\delta-2\delta^\p},
\end{equation}
because
\begin{align*}
&\int_{r_1}^{r}\B^2(t_1+r_1-r^{\p\p},r^{\p\p})dr^{\p\p}\leq \int_{0}^{r}\B^2(t_1+r_1-r^{\p\p},r^{\p\p})dr^{\p\p}\\
&=\int_{0}^{r}\B^2(t+r-r^{\p\p},r^{\p\p})dr^{\p\p}\lesssim r^{2\delta^\p}
\end{align*}
as argued above. It follows that
\begin{align*}
&|u(\xi_1,\eta)-u(\xi_2,\eta)|\leq\int_{\xi_1}^{\xi_2}\frac{|\B|}{\sqrt{|\xi^\p-\eta|}}d\xi^\p\\
&\lesssim \Bigg(\int_{\xi_1}^{\xi_2}|\xi^\p-\eta|^{-2\delta^\p}\B^2d\xi^\p\Bigg)^{1/2}\Bigg(\int_{\xi_1}^{\xi_2}|\xi^\p-\eta|^{2\delta^\p-1}d\xi^\p\Bigg)^{1/2}\\
&\lesssim |r_2|^{\delta-\delta^\p}|r_2^{2\delta^\p}-r_1^{2\delta^\p}|^{1/2}=r_2^\delta\big|1-(\frac{r_1}{r_2})^{2\delta^\p}\big|^{1/2}\\
&\leq r_2^{\delta}\big|1-(\frac{r_1}{r_2})^{2\delta}\big|^{1/2}= \big|r_2^{2\delta}-r_1^{2\delta}\big|^{1/2} \lesssim |r_2-r_1|^\delta,
\end{align*}
which finishes the proof of Holder continuity.\\ \\
Going back to the main argument, we start by recording an estimate on the nonlinearity $Q$ which will be used throughout. As computed in the previous section $Q$ is given by
\begin{align*}
\sum \gamma^{ij}B(U)(\partial_i U, \partial_j U) = B(U)(U_t-U_r,U_t + U_r) + \frac{1}{f^2}B(U)(lAU,lAU),
\end{align*}
where $B$ denotes the second fundamental form of the target sphere, and $A = \displaystyle  \bigl(\begin{smallmatrix} 0 & -1 & 0 \\ 1 & 0 & 0 \\ 0 & 0 & 0\end{smallmatrix}\bigr).$\\
We can therefore bound the nonlinearity as
\begin{equation}\label{nonlinearitybound}
|Q|\lesssim \frac{\A\B}{f(r)},
\end{equation}
which can be rewritten as
\begin{align*}
|Q(u(t^\p,r^\p))| \lesssim (f(r^\p))^{\delta - 3/2}\X\B.
\end{align*}
Next we compute the derivative of $u$ in the characteristic direction $\eta$ which appears in the definition of $\A$ and hence $\X:$
\begin{align*}
&\partial_\eta u = \\
&\lim_{\eta^{'} \rightarrow \eta} \int_{-2-\eta^{'}}^{\xi}\int_{\{\theta^{'}|(t-t^{'})^2 \geq d^2((r,\theta),(r^{'},\theta^{'}))\}}\frac{w^{+}Q(U)f(r^{'})}{\sqrt{(t-t^{'})^2- d^2((r,\theta),(r^{'},\theta^{'}))}}d\theta^{'}d\xi^\p +\\
& \int_{\xi}^{\eta}\int_{-2-\eta^\prime}^{\xi} \partial_\eta[\int_{\{\theta^{'}|(t-t^{'})^2 \geq d^2((r,\theta),(r^{'},\theta^{'}))\}}\frac{w^+Q(U)f(r^{'})}{\sqrt{(t-t^{'})^2- d^2((r,\theta),(r^{'},\theta^{'}))}}d\theta^{'}]d\xi^\p d\eta^\p
\end{align*}
\begin{equation}\label{I-II-def}
 =: I + II.
\end{equation}
To bound $I$ we need to control $\int_{\{\theta^{'}|(t-t^{'})^2 \geq d^2((r,\theta),(r^{'},\theta^{'}))\}}\frac{f(r^{'})}{\sqrt{(t-t^{'})^2- d^2((r,\theta),(r^{'},\theta^{'}))}}d\theta^{'}$ near $\eta^{'} =\eta.$ Note that $\eta = \eta^{'}$ means $t-t^{'} = r^{'}-r$ so as $\eta^{'} \rightarrow \eta$ the region of integration $\{\theta^{'}|(t-t^{'})^2 \geq d^2((r,\theta),(r^{'},\theta^{'}))\}$ shrinks to $\{\theta^{'} =\theta\}$ and on the other hand the integrand $\frac{f(r^{'})}{\sqrt{(t-t^{'})^2- d^2((r,\theta),(r^{'},\theta^{'}))}}$ blows up.
The first key step is the following
\begin{lemma}
$\lim_{\eta^{'} \rightarrow \eta} \int_{\{\theta^{'}|(t-t^{'})^2 \geq d^2((r,\theta),(r^{'},\theta^{'}))\}}\frac{f(r^{'})}{\sqrt{(t-t^{'})^2- d^2((r,\theta),(r^{'},\theta^{'}))}}d\theta^{'} \lesssim \sqrt{\frac{r^\p}{r}}$
\end{lemma}
\begin{proof}
We make the change of variable $y (\theta^{'})= d^2((r,\theta),(r^{'},\theta^{'})).$ Note that $y$ is a function of $r^{'}$ (and $r$) as well as $\theta^{'}.$ We also let $b = (t-t^{'})^2$ and $a = (r-r^{'})^2$ to simplify notation. In terms of the new variable $y$ we have:

\begin{align*}
&\int_{\{\theta^{'}|(t-t^{'})^2 \geq d^2((r,\theta),(r^{'},\theta^{'}))\}}\frac{f(r^{'})}{\sqrt{(t-t^{'})^2- d^2((r,\theta),(r^{'},\theta^{'}))}}d\theta^{'}\\
& = \int_{a}^{\min \{b,(r+r^{'})^2\}}\frac{f(r^{'})dy}{(\frac{\partial{y}}{\partial{\theta^{'}}})\sqrt{b-y}}.
\end{align*}
Since $\eta^{'} \rightarrow \eta,~ \min\{b,(r+r^{'})^2\} =b.$ Also as computed in \cite{F} Theorem 1.2.3, equation (1.2.13), pp.17-19, $<\nabla y,\nabla y> = 4y.$ Since $\nabla y=(y_{r^\p},\frac{y_{\theta^\p}}{f^2}),$ this implies $\partial_{\theta^{'}} y= f(r^{'})\sqrt{4y - (\partial_{r^{'}}y)^2 }.$ The last integral above is therefore equal to
\begin{equation}\label{boundarysingularkernel}
\int_a^b\frac{dy}{\sqrt{4y - (\partial_{r^{'}}y)^2 } \sqrt{b-y}}.
\end{equation}
The triangle connecting the north pole, $(r,0)$ and $(r^{'},\theta^{'})$ will be of special interest to us and from now on we will refer to it as $\Delta.$ The angle at the noth pole is $\theta^{'},$ the angle facing the edge of length $r$ will be called $\alpha,$ and the angle opposite the side of length $r^\p$ will be called $\beta$.\\ \\ \\
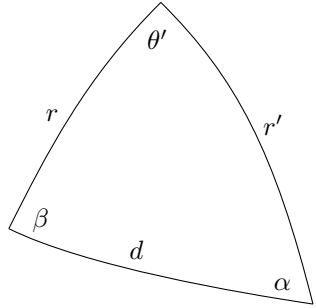
\begin{center}
\begin{tikzpicture}
 \draw (-1,0) .. controls (-2,-1) and (-2.5,-2) .. node [left = 1pt]{$r$} (-3,-3);
 \draw (-1,0) .. node[left = 35, above =25] {$\theta^\p$} controls (0,-1) and (0.5,-2) .. node [right = 1pt]{$r^\p$} node[below =60, right=5]{$\alpha$}(1,-4);
 \draw (-3,-3) ..  node[left = 45, above =13]{$\beta$} node [above= 10, left=3] {$d$} controls (-2,-3.5) and (0,-3.85) .. (1,-4);
\end{tikzpicture}
\end{center}
\begin{lemma}
$\partial_{r^{'}}y = 2d\cos\alpha$, where $d = \sqrt{y}.$
\end{lemma}
\begin{proof}(of lemma 2)
Given a small positive number $h$ let $\gamma$ be the angle between the geodesic connecting $(r^{'},\theta{'})$ to $(r^{'}+h,\theta{'})$ and the one connecting $(r^{'},\theta{'})$ to $(r,\theta),$ so that $\gamma = \pi - \alpha$ and therefore $\cos\gamma = - \cos\alpha.$\\ \\ \\
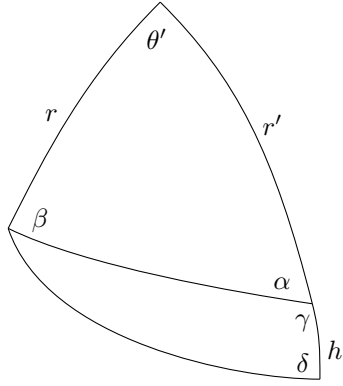
\begin{center}
\begin{tikzpicture}
 \draw (-1,0) .. controls (-2,-1) and (-2.5,-2) .. node [left = 1pt]{$r$} (-3,-3);
 \draw  (1,-4) .. controls (1.1,-4.4) and (1.1, -4.6) .. node [left = 6 pt, above=0.5] {$\gamma$} node[left= 6 pt, above = -15 pt]{$\delta$} node[left=-6 pt, above = -10 pt]{$h$} (1.1,-5);
 \draw (-1,0) .. node[left = 35, above =25] {$\theta^\p$} controls (0,-1) and (0.5,-2) .. node [right = 1pt]{$r^\p$} node[below =60, right=5]{$\alpha$} (1,-4);
 \draw (-3,-3) ..  node[left = 45, above =13]{$\beta$}  controls (-2,-3.5) and (0,-3.85) .. (1,-4);
 \draw (1.1,-5) ..  controls (0.,-5) and (-2.5,-4.5) .. (-3,-3);
\end{tikzpicture}
\end{center}

Since our surface has nonnegative curvature in the region we are considering, a classical triangle comparison result gives (this is also a special case of Topagonov's triangle comparison theorem, see \cite{B}, sections 3.2 and 6.4 or \cite{K} section 2.7)
\[ y(r^{'}+h,\theta^{'})- y(r^{'},\theta^{'}) \leq h^2 - 2hd\cos\gamma,\]
and therefore
\begin{align*}
\partial_{r^{'}}y(r^{'},\theta^{'}) = \lim_{h \rightarrow 0^+}\frac{y(r^{'}+h,\theta^{'})- y(r^{'},\theta^{'})}{h}\leq -\lim_{h \rightarrow 0^+} \frac{2hd\cos\gamma}{h} = 2d\cos\alpha.
\end{align*}
On the other hand if $\delta$ is the angle between the geodesic connecting the north pole to $(r^{'}+h, \theta^{'})$ and the one connecting $(r^{'}+h, \theta^{'})$ to $(r,\theta)$
\begin{align*}
&-\partial_{r^{'}}y(r^{'},\theta^{'}) = \lim_{h \rightarrow 0^{+}}\frac{y(r^{'},\theta^{'})-y(r^{'}+h,\theta^{'})}{h}\\
&\leq -\lim_{h \rightarrow 0^{+}} \frac{2h\sqrt{y(r^{'}+h,\theta^{'})}\cos\delta}{h} = -2d\cos\alpha.
\end{align*}
This together with the previous inequality imply the statement of lemma 2.
\end{proof}
It follows form the lemma that (\ref{boundarysingularkernel}) is equal to
\begin{align*}
\int_a^b\frac{dy}{2\sin\alpha \sqrt{y(b-y)}}.
\end{align*}
To bound $\sin\alpha$ we consider two cases: when $\alpha \leq \pi/2$ and when $\alpha \geq \pi/2.$  Let $\Delta_0$ and $\Delta_K$ be comparison triangles to $\Delta$ (i.e. with sides of lengths $r,r^{'},$ and $d$) in $\RR^2$ and $S^2(\frac{1}{\sqrt K})$ respectively and $\alpha_0,$ $\alpha_K,$ $\theta_0,$ $\theta_K,$ $\beta_0$ and $\beta_K$ the corresponding angles. We have from triangle comparison theorems that $\alpha_0 \leq \alpha \leq \alpha_K.$\\
\begin{lemma}
$\sin\alpha_0\sim\sin\alpha\sim\sin\alpha_K,$ with similar statements holding for the other angles $\beta$ and $\theta.$
\end{lemma}
\begin{proof}(of lemma 3)
the laws of cosines in the flat and spherical cases give
\begin{align*}
\cos \alpha_0 = \frac{d^2+{r^{'}}^2-r^2}{2dr^{'}}
\end{align*}
and
\begin{align*}
\cos{\alpha_K} = \frac{\cos(\sqrt{K}r)-\cos(\sqrt{K}r^{'})\cos(\sqrt{K}d)}{\sin(\sqrt{K}r^{'})\sin(\sqrt{K}d)}.
\end{align*}
In the following computation we keep in mind that $r,~r^{'}$ and $d$ are all very small and therefore we freely replace $\sin r$  by $r,$ etc. It is to be understood that we are omitting multiplication by a constant which we can make arbitrarily close to one by working in a small enough neighborhood of the north pole. It follows, using the trigonometric identity $\cos q-\cos p=2\sin(\frac{p+q}{2})\sin(\frac{p-q}{2})$, that
\begin{align*}
&\sin^2\alpha_K=\frac{(\cos \sqrt{K}r -\cos\sqrt{K}(d+r^\p))(\cos\sqrt{K}(d-r^\p)-\cos \sqrt{K}r)}{\sin^2\sqrt{K}d\sin^2\sqrt{K}r^\p}\\
&=\frac{4\sin(\sqrt{K}(\frac{d+r^\p+r)}{2})\sin(\frac{\sqrt{K}(d+r^\p-r)}{2})\sin(\frac{\sqrt{K}(r+d-r^\p)}{2})\sin(\frac{\sqrt{K}(r-d+r^\p)}{2})}{\sin^2\sqrt{K}d\sin \sqrt{K}r^\p}\\
&\sim\frac{((d+r^\p)^2-r^2)(r^2-(d-r^\p)^2)}{d^2{r^\p}^2}=\sin^2\alpha_0.
\end{align*}
It now suffices to prove that $\sin\alpha\sim\sin\alpha_0.$ That $\alpha_0\leq\alpha\leq\alpha_K$ is not sufficient to guarantee this because a priori we may have an $\alpha$ which is close to $\pi/2$ and a very small $\alpha_0$ and $\alpha_K=\pi-\alpha_0.$ We need to do a little more work. If $\alpha\leq \pi/2$ then
\begin{align*}
\frac{\sin \alpha_0}{\sin\alpha} \leq 1.
\end{align*}
If $\alpha >\pi/2$
\begin{align*}
\frac{\sin\alpha_0}{\sin\alpha}\lesssim\frac{\sin\alpha_K}{\sin\alpha}\leq 1,
\end{align*}
so it remains to show that $\frac{\sin\alpha}{\sin\alpha_0}\lesssim 1.$ If $\alpha_0\geq\pi/2$
\begin{align*}
\frac{\sin\alpha}{\sin\alpha_0}\leq 1.
\end{align*}
If $\alpha_K\leq\pi/2$
\begin{align*}
\frac{\sin\alpha}{\sin\alpha_0}\lesssim\frac{\sin\alpha}{\sin\alpha_K}\leq 1
\end{align*}
and again we are done. The difficulty is of course when $\alpha_0<\pi/2$ and $\alpha_K>\pi/2.$ Since $\frac{\sin\alpha}{\sin\alpha_0}\lesssim \frac{\sin\alpha}{\sin\alpha_K}$ it suffices to show that $\alpha_K$ is bounded away from $\pi$ if $\alpha_0<\pi/2.$ For this we need to show that $1+\cos\alpha_K$ stays away from zero. Note that since $\cos\alpha_0$ is positive when $\alpha_0<\pi/2$ we have $d^2+{r^\p}^2>r^2$ by flat the law of cosines. We set $K=1$ to simplify notation, but the argument for the general case is exactly the same:
\begin{align*}
&1+\cos\alpha_K=\frac{\sin d \sin r^\p -\cos d \cos r^\p +\cos r}{\sin d \sin r^\p}= \frac{\cos r -\cos(d+r^\p)}{\sin d \sin r^\p}\\
&= \frac{2\sin(\frac{(d+r^\p)-r)}{2})\sin(\frac{(d+r^\p)+r)}{2})}{\sin d \sin r^\p}\geq \frac{c}{2}\Big(\frac{(d+r^\p)^2-d^2}{dr^\p}\Big)\geq c,
\end{align*}
as desired.
\end{proof}
We will not need the following corollary in the proof of lemma 1, but we record it for future use. The proof is immediate from the flat and spherical laws of sines, and the lemma.
\begin{corollary}
$\frac{\sin\beta}{r^\p}\sim\frac{\sin\alpha}{r}\sim\frac{\sin\theta^\p}{d}.$
\end{corollary}
To finish the proof of lemma $1$ we consider two cases:\\
Case $1:r<r^\p.$
In this case $\sqrt{a}=r^\p-r.$ We use $\sin\alpha_0\sim\sin\alpha$ in this case to estimate
\begin{align*}
&\int_a^b\frac{dy}{2\sin \alpha \sqrt{y(b-y)}} \leq \int_a^b\frac{r^{'}dy}{\sqrt{b-y}\sqrt{(d+r^{'})^2-r^2)}\sqrt{(r^2-(d-r^{'})^2)}}\\
& \leq \int_{a}^{b} \frac{r^{'}dy}{\sqrt{b-y}\sqrt{r^2-(d-r^{'})^2}\sqrt{(r^{'}-r+r^{'})^2-r^2}}\\
 &= \frac{\sqrt{r^{'}}}{2\sqrt{r^{'}-r}}\int_{a}^{b} \frac{dy}{\sqrt{b-y}\sqrt{r^2-(d-r^{'})^2}}.
\end{align*}
Letting $z=\sqrt y,~\mu = \sqrt{a},~$ and $\nu = \sqrt{b}$ this last integral becomes:
\begin{align*}
&\frac{\sqrt{r^{'}}}{\sqrt{\mu}}\int_\mu^\nu\frac{z dz}{\sqrt{\nu^2-z^2}\sqrt{r^2 - (z-\mu -r)^2}}\\
&\leq\frac{\nu\sqrt{r^{'}}}{\sqrt{2}\mu}\int_\mu^\nu \frac{dz}{\sqrt{\nu-z}\sqrt{z-\mu}\sqrt{r-(z-\mu-r)}}\\
&\leq \frac{\nu\sqrt{r^{'}}}{\sqrt{2}\mu\sqrt{2r-(\nu-\mu)}}\int_\mu^\nu\frac{dz}{\sqrt{\nu-z}\sqrt{z-\mu}}.
\end{align*}
To compute the last integral above we make the change of variables $t= \frac{z-\mu}{\nu-\mu}$ to get
\[\int_\mu^\nu\frac{dz}{\sqrt{\nu-z}\sqrt{z-\mu}} = \int_0^1{dt}{\sqrt{t(t-1)}}= \pi.\]
Putting everything together we have obtained
\begin{equation}\label{lessthannintybound}
\frac{1}{2}\int_a^b\frac{dy}{\sqrt{y}\sqrt{b-y}\sin\alpha} \leq \frac{\pi\nu\sqrt{r^{'}}}{\sqrt{2}\mu\sqrt{2r-(\nu-\mu)}}.
\end{equation}
Case $2:r>r^\p.$
In this case $a=r-r^\p$ and we use the comparison $\sin\alpha_K=\sin\alpha.$  We set $K=1$ as its exact value is irrelevant in our estimates.
\begin{align*}
&\int_a^b\frac{dy}{\sqrt{y}\sqrt{b-y}\sin\alpha}\\
&\leq \int_a^b\frac{(\sin r^{'})(\sin d)dy}{\sqrt{y(b-y)}\sqrt{\cos r - \cos(r{'}+d)}\sqrt{\cos(r^{'} -d)- \cos r}}\\
&= \int_\mu^\nu\frac{2\sin r^{'} \sin z}{\sqrt{\nu -z}\sqrt{\nu+z}\sqrt{\cos r - \cos(z- \mu +r)}\sqrt{\cos(z+\mu-r) -\cos r}}dz\\
&\lesssim \frac{\sin \nu \sin r^{'}}{\sqrt{\mu}\sqrt{\cos(\nu+\mu-r) -\cos r}} \int_\mu^\nu\frac{dz}{\sqrt{\nu-z}\sqrt{\sin (\frac{z-\mu}{2})\sin(\frac{z+r+r^\p}{2})}}\\
&\lesssim \frac{\sin \nu \sin r^{'}}{\sqrt{\sin r}\sqrt{\mu}\sqrt{\cos(\nu+\mu-r) -\cos r}}\int_\mu^\nu\frac{dz}{\sqrt{\nu -z}\sqrt{z - \mu}}
\end{align*}
\begin{equation}\label{morethannintybound}
\sim \frac{\sin \nu \sin r^{'}}{\sqrt{\sin r}  \sqrt{\mu}\sqrt{\cos(\nu+\mu-r) -\cos r}}.
\end{equation}
Note that
\begin{align*}
&\lim_{\mu \rightarrow \nu}(\cos(\nu+\mu-r) -\cos r) = \cos((r-r^\p)-r^\p) - \cos((r-r^\p)+r^\p)\\
&= 2\sin r^\p\sin(r-r^\p).
\end{align*}
With this in mind, taking the limit as $\mu$ approaches $\nu$ in equations (\ref{lessthannintybound}) and (\ref{morethannintybound}) we get
\begin{align*}
\lim_{\eta^\p \rightarrow \eta}\int_a^b\frac{dy}{\sqrt{y}\sqrt{a-y}\sin\alpha} \lesssim \sqrt{\frac{r^{'}}{r}},
\end{align*}
as desired.
\end{proof}
Going back to (\ref{I-II-def}) and using (\ref{nonlinearitybound}) (or the alternative version given right after (\ref{nonlinearitybound})) we have
\begin{align*}
&|I| \lesssim \frac{\X}{\sqrt{f(r)}}\int_{-2-\eta}^\xi (f(r^\p))^{\delta-1}\B|_{\eta^\p = \eta} d\xi^\p\\
&\lesssim \frac{\X}{\sqrt{f(r)}} \big{(}\int_{-2-\eta}^{\xi} \B^2 d \xi^\p \big{)}^{1/2}\big(\int_{-2-\eta}^{\xi}(f(r^\p))^{2\delta-2}d\xi^\p\big)^{1/2}\\
&\sim\frac{\X}{\sqrt{f(r)}} \big{(}\int_{-2-\eta}^{\xi} \B^2 d \xi^\p \big{)}^{1/2}\big(\int_{r}^{r+t+1}{r^\p}^{2\delta-2}dr^\p\big)^{1/2}
\end{align*}
\begin{equation}\label{Ibound}
\lesssim_\delta r^{\delta-1}\X\epsilon.
\end{equation}
The bound on the integral involving $\B^2$ comes from the energy-flux inequality
\begin{align*}
\int_{C_t}(e-m)d\mu_C \leq E(t) < \epsilon^2.
\end{align*}
The last inequality is just the small energy hypothesis.\\
The more difficult step is obtaining a similar bound for $II.$ We start with a remark about symmetry.
\begin{remark}
By the same arguments as above $y_{\theta}((r,0),(r^\p, \theta^\p))= -f(r)\sqrt{4y-y_r^2} = -2f(r)d\sin\beta$ (note that we are working with $\theta =0$ and $\theta^\p\geq0$ so $y_\theta$ has to be negative and $y_{\theta^\p}$ positive). Now
\begin{align*}
&y_{\theta^\p}((r,0),(r^\p, \theta^\p))=\lim_{h \rightarrow 0}\frac{y((r,0),(r^\p, \theta^\p+h))-y((r,0),(r^\p, \theta^\p))}{h}\\
&=\lim_{h \rightarrow 0}\frac{y((r,-h),(r^\p, \theta^\p))-y((r,0),(r^\p,\theta^\p))}{h}=-y_{\theta}((r,0),(r^\p, \theta^\p)),
\end{align*}
and we get the important equality
\begin{align*}
f(r)\sin\beta=f(r^\p)\sin\alpha.
\end{align*}
Now recall that the Gauss-Bonnet theorem implies that $\lim_{\theta^\p\rightarrow0}(\alpha+\beta)=\pi.$ We want to examine the non-generic case where $\alpha$ and $\beta$ do not approach $0$ and $\pi$ more carefully.
\begin{align*}
\alpha(0)=\pi-\beta(0) \Rightarrow \sin\alpha(0)=\sin\beta(0).
\end{align*}
If $\alpha$ and $\beta$ are not $0$ or $\pi$ then $\sin\alpha$ and $\sin\beta$ are nonzero and the above equality implies that $f(r)=f(r^\p).$ Since $f$ is one to one (at least near the poles where all of our computations are done) $r=r^\p.$ We conclude by symmetry that $\alpha(0)=\beta(0)=\pi/2.$
\end{remark}
Going back to the main argument we divide the cone $K$ into two parts $K_1$ and $K_2$ corresponding to the regions $t-t^\p\geq r+r^\p$ and $t-t^\p \leq r+r^\p$ respectively. We also divide $II$ into the two corresponding integrals $II_1+II_2.$ Recall that in the flat case \cite{CT} proceed by writing
\begin{align*}
&\int_{\{\theta^{'}|(t-t^{'})^2 \geq d^2((r,\theta),(r^{'},\theta^{'}))\}\cap K_1}\frac{r^{'}}{\sqrt{(t-t^{'})^2- d^2((r,\theta),(r^{'},\theta^{'}))}}d\theta^{'}\\
&= 2\int_0^\mu \frac{r^\p}{\sqrt{2rr^\p(\cos\theta^\p-\cos\mu)}},
\end{align*}

where $\mu = \mu(r,r^\p,t-t^\p)$ is the positive angle such that $y(r,r^\p,\mu)= (t-t^\p)^2.$ The derivative inside the integral in $II_1$ is then written as
\begin{align*}
&\partial_{\eta}\Big[\int_0^\mu \frac{r^\p d\theta^\p}{\sqrt{2rr^\p(\cos\theta^\p-\cos\mu)}}\Big]\\
&=(\partial_{\eta}\mu)\partial_{\mu}\Big[\int_0^\mu \frac{r^\p d\theta^\p}{\sqrt{2rr^\p(\cos\theta^\p-\cos\mu)}}\Big] +  \int_0^\mu \partial_\eta\Big[\frac{r^\p}{\sqrt{2rr^\p(\cos\theta^\p-\cos\mu)}}\Big]d\theta^\p\\
&=(\partial_{\eta}\mu)\partial_{\mu}\Big[\int_0^\mu \frac{r^\p d\theta^\p}{\sqrt{2rr^\p(\cos\theta^\p-\cos\mu)}}\Big] -  \frac{1}{r\sqrt{2rr^\p}}\int_0^\mu \frac{r^\p d\theta^\p}{\sqrt{(\cos\theta^\p-\cos\mu)}}.
\end{align*}

Our first goal is to write $\partial_{\eta}[\int\frac{Q(U)w^{+}f(r^{'})}{\sqrt{(t-t^{'})^2- d^2((r,\theta),(r^{'},\theta^{'}))}}d\theta^{'}]$ as the sum of two similar terms. The difficulty lies in proving the following

\begin{lemma}
 $|\partial_{\eta}\frac{1}{\sqrt{y(\mu)-y(\theta^\p)}}| \lesssim \frac{1}{r\sqrt{rr^\p(\cos\theta^\p -\cos\mu)}}$ in $K_1.$
\end{lemma}

Of course here we treat $\mu$ as a constant so in particular $\partial_\eta = \partial_r$ as there is no time dependency.\\ \\
$\mathbf{Notation:}$ From now on we write $s$ for $t-t^\p.$\\ \\
We need two auxiliary lemmas.
\begin{lemma}
$\frac{\cos\theta^\p -\cos\mu}{y(\mu)-y(\theta^\p)}\sim\frac{1}{rr^\p}.$
\end{lemma}

\begin{proof}(of lemma 7)
 Let $m(\theta^\p):=\frac{\cos\theta^\p -\cos\mu}{y(\mu)-y(\theta^\p)},$ defined for $\theta^\p \in (0,\mu).$ To find upper and lower bounds for $m$ we need to consider three cases: $\theta^\p =0,~\theta^\p\rightarrow \mu,$ and $m^\p(\theta^\p)=0.$
First consider the case where $\theta^\p \rightarrow \mu:$
\begin{align*}
&\lim_{\theta^\p \rightarrow \mu}\frac{\cos \theta^\p - \cos \mu}{y(\mu) -y(\theta^\p)} = \lim_{\theta^\p \rightarrow \mu}\frac{\sin(\theta^\p)}{y_{\theta^\p}(\theta^\p)}\\
&\lim_{\theta^\p \rightarrow \mu}\frac{\sin\theta^\p}{2df(r^\p)\sin\alpha}
\end{align*}
\begin{equation}\label{disanceanglebound}
\sim \frac{\sin\mu}{sr^\p\sin\alpha(\mu)} \sim \frac{1}{rr^\p},
\end{equation}
from Corollary $1.$ The next case we consider is $m^\p(\theta^\p) =0.$ Note that
 \begin{align*}
 0=m^\p(\theta^\p)=\frac{1}{(y(\mu)-y(\theta^\p))^2}[-\sin\theta^\p(y(\mu)-y(\theta^\p))+y_{\theta^\p}(\theta^\p)(\cos\theta^\p-\cos\mu)],
 \end{align*}
 so at a point where the derivative is zero
 \begin{align*}
 m(\theta^\p)=\frac{\sin\theta^\p}{y_{\theta^\p}(\theta^\p)}=\frac{\sin\theta^\p}{2f(r^\p)d(\theta^\p)\sin\alpha(\theta^\p)}.
 \end{align*}
 But we already estimated this term in (\ref{disanceanglebound}).\\
 It remains to consider the $\theta^\p=0$ case. Note that $y(0)=y_0(0)=y_K(0)=(r-r^\p)^2.$ Here we are using comparison triangles with two sides of length $r$ and $r^\p$ meeting with an angle of $\theta^\p,$ and the comparison theorems imply $y_K\leq y\leq y_0.$ We get
 \begin{align*}
 &m(0) = \frac{1-\cos\mu}{y(\mu)-y(0)}\leq\frac{1-\cos\mu}{y_K(\mu)-y_K(0)}\lesssim \frac{1}{rr^\p}.\\
 \end{align*}
 This gives the upper bound. The lower bound computation is similar
 \begin{align*}
 &m(0) = \frac{1-\cos\mu}{y(\mu)-y(0)}\geq\frac{1-\cos\mu}{y_0(\mu)-y_0(0)}\geq \frac{c}{rr^\p}.
 \end{align*}
 \end{proof}
 \begin{lemma}
$(i)~ \frac{\sin\beta}{d} \leq \alpha_r \leq \frac{\sqrt{K}\sin\beta}{\sin \sqrt{K}d}.$ $(ii)~ \frac{-\sqrt{K}f(r)\cos\beta}{\sin \sqrt{K}d}\leq \alpha_{\theta^\p} \leq \frac{-f(r)\cos\beta}{d}.$\\
$(iii)~\frac{-\sqrt{K}f(r^\p)\cos\alpha}{\sin \sqrt{K}d}\leq \beta_{\theta^\p} \leq \frac{-f(r^\p)\cos\alpha}{d}.$
\end{lemma}
\begin{proof}(of lemma 8)
$(i)$ Let all variables be defined as in the following picture
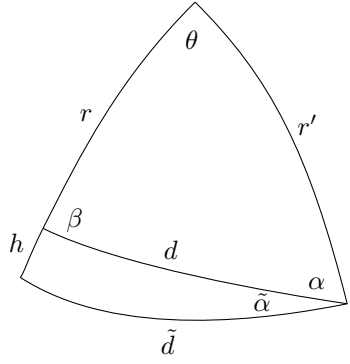
\begin{center}
\begin{tikzpicture}
 \draw (-1,0) .. controls (-2,-1) and (-2.5,-2) .. node [left = 1pt]{$r$} (-3,-3);
 \draw (-3,-3) .. controls (-3.2,-3.4) and (-3.25, -3.6) .. node [left = 4 pt, above=0.5] {$h$}(-3.3,-3.65);
 \draw (-1,0) .. node[left = 35, above =25] {$\theta$} controls (0,-1) and (0.5,-2) .. node [right = 1pt]{$r^\p$} node[below =60, right=5]{$\alpha$} node[below =67, left=2]{$\tilde{\alpha}$}(1,-4);
 \draw (-3,-3) ..  node[left = 45, above =13]{$\beta$} node [above= 10, left=3] {$d$} controls (-2,-3.5) and (0,-3.85) .. (1,-4);
 \draw (-3.3,-3.65) .. node [below= 8, left =3] {$\tilde{d}$} controls (-2,-4.5) and (0,-4.2) .. (1,-4);
\end{tikzpicture}
\end{center}
Note that we are working so close to the north pole that the curvature is positive. Since we are estimating the derivative $h$ is small so $\sin\tilde{\alpha}_0 \leq \sin\tilde{\alpha}\leq\sin\tilde{\alpha}_K.$
\begin{align*}
\alpha_r = \lim_{h\rightarrow 0} \frac{\tilde{\alpha}}{h} = \lim_{h \rightarrow 0}\frac{\sin\tilde{\alpha}}{h} \in (\lim_{h \rightarrow 0}\frac{\sin\tilde{\alpha}_0}{h},~\lim_{h \rightarrow 0}\frac{\sin\tilde{\alpha}_K}{h}).
\end{align*}
We compute the endpoints of this interval using laws of cosines, starting with the upper bound. To get the second line in the computation below we have used the trigonometric identity $\cos h-\cos(d+\tilde{d})= 2\sin(\frac{h-(d+\tilde{d})}{2})\sin(\frac{h+(d+\tilde{d})}{2}),$ and the fact that $\tilde{d}$ approaches $d$ as $h$ goes to zero:
\begin{align*}
&\lim_{h \rightarrow 0}\frac{\sin\tilde{\alpha}_K}{h} = \lim_{h \rightarrow 0}\frac{\sqrt{(\cos \sqrt{K}h - \cos\sqrt{K}(d+\tilde{d}))(\cos\sqrt{K}(\tilde{d}-d)-\cos\sqrt{K} h)}}{h\sin \sqrt{K}d \sin\sqrt{K}\tilde{d}}\\
&=\frac{\sqrt{2}}{\sin \sqrt{K}d}\lim_{h \rightarrow 0}\frac{\sqrt{\cos\sqrt{K}(\tilde{d}-d)-\cos \sqrt{K}h}}{h}\\
 &= \frac{2}{\sin \sqrt{K}d}\lim_{h \rightarrow 0}\frac{\sqrt{\sin\sqrt{K}(\frac{h+d-\tilde{d}}{2})\sin\sqrt{K}(\frac{h+\tilde{d}-d}{2})}}{h}\\
&=\frac{2\sqrt{K}}{\sin \sqrt{K}d}\sqrt{\frac{1}{4}\lim_{h\rightarrow 0}(\frac{h+d-\tilde{d}}{h})(\frac{h+\tilde{d}-d}{h})}= \frac{\sqrt{K}\sqrt{1-d_r^2}}{\sin \sqrt{K}d}= \frac{\sqrt{K}\sin\beta}{\sin \sqrt{K}d}.
\end{align*}
The lower bound computation is similar:
\begin{align*}
&\lim_{h\rightarrow 0}\frac{\sin\tilde{\alpha}_{0}}{h}= \lim_{h\rightarrow 0}\frac{\sqrt{(h^2-(d-\tilde{d})^2)((d+\tilde{d})^2-h^2)}}{2d\tilde{d}h}\\
&=\frac{1}{d}\lim_{h\rightarrow 0}\frac{\sqrt{h^2-(\tilde{d}-d)^2}}{h}=\frac{1}{d}\sqrt{1-d_{r}^2}=\frac{\sin\beta}{d}.
\end{align*}
This concludes the proof of part $(i).$\\
Part $(ii)$ now follows easily. By the symmetry of the argument in $(i)$ we have $\frac{\sin\alpha}{d}\leq\beta_{r^\p}\leq\frac{\sqrt{K}\sin\alpha}{\sin \sqrt{K}d}.$ We can finish the proof as follows
\begin{align*}
&y_{r^\p\theta^\p}=y_{\theta^\p r^\p},\\
&\partial_{r^\p}(2df(r)\sin\beta)=\partial_{\theta^\p}(2d\cos\alpha),\\
&2f(r)\cos\alpha\sin\beta+2df(r)\cos\beta\beta_{r^\p}=2f(r)\cos\alpha\sin\beta-2d\sin\alpha\alpha_{\theta^\p},\\
&2df(r)\cos\beta\beta_{r^\p}=-2d\sin\alpha\alpha_{\theta^\p}=-2d\frac{f(r)}{f(r^\p)}\sin\beta\alpha_{\theta^\p},\\
&\alpha_{\theta^\p}=\frac{-f(r^\p)\cos\beta\beta_{r^\p}}{\sin\beta}=\frac{-f(r)\cos\beta\beta_{r^\p}}{\sin\alpha}.
\end{align*}
Part $(ii)$ now follows from part $(i).$\\
To get $(iii),$ we use remark $7$ to write
\begin{align*}
&f(r^\p)\sin\alpha=f(r)\sin\beta,\\
&f(r^\p)\cos\alpha\alpha_{\theta^\p}= f(r)\cos\beta\beta_{\theta_\p},
\end{align*}
\begin{equation}\label{alphatobetader}
\beta_{\theta^\p}=\frac{f(r^\p)\cos\alpha}{f(r)\cos\beta}\alpha_{\theta^\p}.
\end{equation}
$(iii)$ now follow from $(ii).$
\end{proof}
To avoid repetition we also record the following observation which was already tacitly used in the proof of lemma $5.$
\begin{observation}
Suppose $f$ and $g$ are smooth functions such that $g$ does not vanish in the interior of the interval $[a,b],$ and that we seek to bound $\frac{f}{g}$ in absolute value over this interval. We set the derivative equal to zero and find that at an interior extremum we must have $\frac{f}{g}=\frac{f^\p}{g^\p}.$ It therefore suffices to consider the values of $\frac{f}{g}$ at the endpoints and also bound $\frac{f^\p}{g^\p}$ in $[a,b].$ Moreover if both $f$ and $g$ vanish at one (or both) of the endpoints, then by l'Hopital's rule estimating that fraction at the endpoint again reduces to estimating the derivative fraction at that endpoint.
\end{observation}
 We can now prove lemma $6$.
 \begin{proof} (of lemma 6)
 \begin{align*}
 &r\sqrt{rr^\p(\cos\theta^\p - \cos \mu)}\partial_{\eta}[\frac{1}{\sqrt{y(\mu) - y(\theta^\p)}}]\\
 &= \frac{r}{2}\sqrt{rr^\p(\cos\theta^\p - \cos \mu)}\frac{y_r(\theta^\p)-y_r(\mu)}{(y(\mu)-y(\theta^\p))^{3/2}}\\
 &\sim \frac{y_r(\theta)-y_r(\mu)}{r^\p(\cos\theta^\p-\cos\mu)}
 \end{align*}
 \begin{equation}\label{partialetabound}
 =\frac{2d(\theta^\p)\cos\beta(\theta^\p)-2d(\mu)\cos\beta(\mu)}{r^\p(\cos\theta^\p-\cos\mu)}.
 \end{equation}
To complete the proof of lemma $6$ we need to show that (\ref{partialetabound}) is bounded independently of $r,r^\p,$ and $d.$ Note that $\theta^\p$ is between $0$ and $\mu$ so in view of the observation we differentiate the numerator and denominator in (\ref{partialetabound}) with respect to $\theta^\p$ to get the derived fraction (up to a constant factor)
\begin{equation}\label{partialboundder}
\frac{f(r^\p)\sin\alpha\cos\beta-d\sin\beta\beta_{\theta^\p}}{r^\p\sin\theta^\p}.
\end{equation}
To bound this expression we substitute the two extreme values for $\beta_{\theta^\p}$ from part $(iii)$ of lemma $6.$ Setting $\beta_{\theta^\p}=\frac{-f(r^\p)\cos\alpha}{d}$ in (\ref{partialboundder}) we get
\begin{align*}
\frac{f(r^\p)\sin(\alpha+\beta)}{r^\p\sin\theta^\p}\sim\frac{\sin(\alpha+\beta)}{\sin\theta^\p},
\end{align*}
$\theta^\p \in [0,\pi].$ If $\theta\in [\frac{\pi}{2}-\delta_0,\frac{\pi}{2}+\delta_0]$ for some small $\delta_0>0$ then this fraction is bonded by a constant. Otherwise, we can use the observation again to replace this fraction by
\begin{align*}
\frac{\partial_{\theta^\p}\sin(\alpha+\beta)}{\partial_{\theta^\p}\sin\theta^\p}=\frac{\cos(\alpha+\beta)(\alpha+\beta)_{\theta^\p}}{\cos\theta^\p},
\end{align*}
which is bounded on $[0,\frac{\pi}{2}-\delta_0]$ and $[\frac{\pi}{2}+\delta_0,\pi]$ if we can bound $(\alpha+\beta)_{\theta^\p}.$ This derivative can be computed using the Gauss-Bonnet formula
\begin{equation}\label{gaussbonnet}
\alpha+\beta+\theta^\p = \pi + \int_0^{\theta^\p}\int_0^{\rho(\phi)}k(r^{\p\p})f(r^{\p\p})dr^{\p\p} d\phi,
\end{equation}
where $\rho(\phi)$ is the distance from the north pole the point on the geodesic connecting $(r,0)$ to $(r^\p, \phi)$ whose $\theta-$cordinate is $\phi^.$ In particular $\rho(\theta^\p) = r^\p$ and $\rho(0)=r:$\\

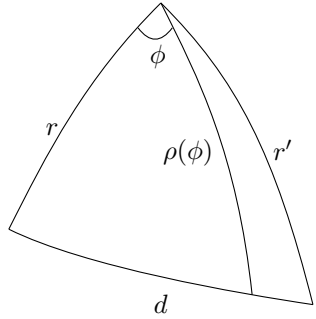
\begin{center}
\begin{tikzpicture}
 \draw (-1,0) .. controls (-2,-1) and (-2.5,-2) .. node [below = 5pt, left = 1pt]{$r$} (-3,-3);
 \draw (-1,0) .. controls (-0.5,-1) and (0,-2) .. node [left = 10pt, below=2pt]{$\rho(\phi)$} (0.2,-3.87);
 \draw (-1,0) .. node[left = 35, above =18] {$\phi$} controls (0,-1) and (0.5,-2) .. node [right = 12pt, below = 2pt]{$r^\p$} (1,-4);
 \draw (-3,-3) .. node [below=3] {$d$} controls (-2,-3.5) and (0,-3.85) .. (1,-4);
 \draw (-1.3,-0.33) .. controls (-1.15,-0.53) and (-1,-0.53) .. (-0.85,-0.33);
\end{tikzpicture}
\end{center}
Taking the derivative of both sides of (\ref{gaussbonnet}) we get
\begin{equation}\label{alphaplusbetader}
(\alpha+\beta)_{\theta^\p}= \int_0^{r^\p}k(r^{\p\p})f(r^{\p\p})dr^{\p\p} -1,
\end{equation}
which is bounded.\\
We consider the other extreme case of (\ref{partialboundder}) next. Setting $\beta_{\theta^\p}=\frac{-\sqrt{K}f(r^\p)\cos\alpha}{\sin\sqrt{K}d}$ and taking the difference with the other extreme case for $\beta_{\theta^\p}$ we see that we need to bound
\begin{align*}
\frac{f(r^\p)\cos\alpha\sin\beta(\frac{\sqrt{K}d}{\sin\sqrt{K}d}-1)}{r^\p\sin\theta^\p}.
\end{align*}
By corollary $1$ this is bounded by
\begin{align*}
\frac{r^\p(\frac{\sqrt{K}d}{\sin\sqrt{K}d}-1)}{d},
\end{align*}
which is bounded.\\
To finish bounding (\ref{partialetabound}) it remains to consider the $\theta^\p=0$ endpoint:
\begin{align*}
\frac{2|r-r^\p|\cos\beta(0)-2d(\mu)\cos\beta(\mu)}{r^\p(1-\cos\mu)}.
\end{align*}
We can again use the observation, this time for the variable $\mu \in [0,\pi].$ At $\mu=\pi$ the fraction is bounded, and at $\mu=0$ both the numerator and the denominator vanish, so by the observation is suffices to consider the differentiated fraction
\begin{align*}
\frac{f(r^\p)\sin\alpha\cos\beta-d\sin\beta\beta_{\mu}}{r^\p\sin\mu}.
\end{align*}
But this expression was already bounded in (\ref{partialboundder}).
\end{proof}
We are now in the position to estimate $II_1.$ We make one more change of variables just to be able to use the results from \cite{ST1} and \cite{CT} directly: $\nu = \cos \mu.$ Define
\begin{align*}
P(\nu):= \int_0^\mu\frac{d\theta^\p}{\sqrt{\cos\theta^\p-\cos\mu}}=\int_\nu^1\frac{dz}{\sqrt{1-z^2}\sqrt{z-\nu}}.
\end{align*}
We want to write $II_1$ in a form that resembles its flat analogue in \cite{CT}, so we start with a brief explanation of the estimates in that papers. Note that in the flat case $\nu = \frac{r^2+{r^\p}^2-s^2}{2rr^\p}.$ \cite{CT} write $II_1$ as ($Q$ is independent of $\theta^\p$ there)
\begin{align*}
&\int_{K_1(r,t)}Q(r^\p,t^\p)\partial_\eta\Big[\int_0^\mu\frac{d\theta^\p}{\sqrt{y(\mu)-y(\theta^\p)}}\Big]r^\p dr^\p dt^\p\\
&= \int_{K_1(r,t)}Q(r^\p,t^\p)\partial_\eta[\frac{1}{\sqrt{2rr^\p}}P(\nu)]r^\p dr^\p dt^\p\\
&=c_1\int_{K_1(r,t)}\frac{Q(r^\p,t^\p)}{r^{3/2}{r^\p}^{1/2}}P(\nu)r^\p dr^\p dt^\p \\
&+ c_2\int_{K_1(r,t)}\frac{(r-s)^2-{r^\p}^2}{2\sqrt{2}r^{5/2}{r^\p}^{3/2}}Q(r^\p,t^\p)P^\p(\nu)r^\p dr^\p dt^\p.
\end{align*}
The nonlinearity $Q$ is then replaced by its upper bound $\frac{\A\B}{r}$ and the two terms are then bounded separately (i.e. there are no further cancelations; see also \cite{ST1} pp. $251-252$ and \cite{SS1} sections $8.2$ and $8.3$). We would therefore like to set it as our goal to write $II_1$ as
\begin{align*}
\int_{K_1(r,t)}Q(r^\p,t^\p)\bigo(\frac{1}{r^{3/2}{r^\p}^{1/2}})P(\nu)f(r^\p)dr^\p dt^\p
\end{align*}
\begin{equation}\label{ctcomparison}
+ \int_{K_1(r,t)}Q(r^\p,t^\p)\bigo(\frac{(r-s)^2-{r^\p}^2}{2\sqrt{2}r^{5/2}{r^\p}^{3/2}})P^\p(\nu)f(r^\p)dr^\p dt^\p.
\end{equation}
Unfortunately, this will not be possible. The reason is that unlike the flat case $\frac{\cos\theta^\p-\cos\mu}{y(\mu)-y(\theta^\p)}$ is not independent of $\theta^\p$ and therefore some extra derivatives will appear in our computations. Nevertheless, we will be able to bound all these derivatives in a way that allows us to prove that
\begin{equation}\label{IIbound}
|II_1|\lesssim f^{\delta-1}(r)\X \epsilon.
\end{equation}
We start by estimating $\partial_\eta\nu.$ Taking derivatives form both sides of the equation $y((r,0),(r^\p,\mu))=s^2$ we get
\begin{align*}
&y_r+ \mu_ry_{\theta^\p}=0,\\
&\mu_ty_{\theta^\p}=2s,
\end{align*}
and therefore
\begin{align*}
&\mu_\eta=\frac{2s-y_r}{y_{\theta^\p}}=\frac{1-\cos\beta}{f(r)\sin\beta},
&-\nu_\eta=(\sin\mu)\mu_\eta=\frac{\sin\mu(1-\cos\beta)}{f(r)\sin\beta}.
\end{align*}
To bound this last term, note that by corollary $1$, $\frac{\sin\mu}{\sin\beta(\mu)}\sim\frac{s}{r^\p}.$ Moreover, since $\cos x$ is decreasing on $(0,\pi),$
\begin{align*}
&1-\cos\beta \leq 1-\cos\beta_K = 1- \frac{\cos\sqrt{K}r^\p-\cos\sqrt{K}r\cos\sqrt{K}s}{\sin\sqrt{K}r\sin\sqrt{K}s}\\
&= \frac{\cos\sqrt{K}(r-s)-\cos\sqrt{K}r^\p}{\sin\sqrt{K}r\sin\sqrt{K}s}\sim\frac{(r-s)^2-{r^\p}^2}{rs}.
\end{align*}
It follows that
\begin{align*}
|\nu_\eta|\lesssim \frac{(r-s)^2-{r^\p}^2}{r^2r^\p}.
\end{align*}
Using lemma $4$ and this last computation, $II_1$ can be bounded by
\begin{align*}
&\int_{K_1(r,t)} \frac{\|w^+Q\|_{L^\infty(\theta^\p)}}{r^{3/2}{r^\p}^{1/2}}P(\nu)f(r^\p)dr^\p dt^\p\\
&+ \int_{K_1(r,t)}\frac{\|w^{+}_\eta Q\|_{L^\infty(\theta^\p)}}{r^{1/2}{r^\p}^{1/2}}P(\nu)f(r^\p)dr^\p dt^\p\\
&+\int_{K_1(r,t)} (\partial_\eta\nu)\partial_\nu\Big[\int_0^\mu\frac{w^+ Qd\theta^\p}{\sqrt{y(\mu)-y(\theta^\p)}}\Big]f(r^\p)dr^\p dt^\p\\
&=\int_{K_1(r,t)} \bigo(\frac{Q}{r^{3/2}{r^\p}^{1/2}})P(\nu)f(r^\p)dr^\p dt^\p\\
\end{align*}
\begin{equation}\label{muetadichotomy}
+ \int_{K_1(r,t)} \bigo(\frac{(r-s)^2-{r^\p}^2}{r^2r^\p})\partial_\nu\Big[\int_0^\mu\frac{w^+ Q d\theta^\p}{\sqrt{y(\mu)-y(\theta^\p)}}\Big]f(r^\p)dr^\p dt^\p.
\end{equation}
This is already somewhat similar to (\ref{ctcomparison}).\\
With $z=\cos\theta^\p,$ $\nu=\cos\mu$ and $\phi=w^{+}Q$ we have
\begin{align*}
\int_0^\mu\frac{\phi d\theta^\p}{\sqrt{y(\mu)-y(\theta^\p)}}=\int_\nu^1\frac{\phi(z)dz}{\sqrt{1-z^2}\sqrt{y(\nu)-y(\theta^\p)}}.
\end{align*}
Letting $x^2(\theta^\p)=y(\theta^\p)-y(0),$ and $\lambda=x(\mu)$ this becomes
\begin{align*}
&\int_0^\lambda\frac{2x\phi dx}{y_{\theta^\p}(\theta^\p)\sqrt{\lambda^2-x^2}}=\int_0^\lambda\frac{d}{dx}(\sin^{-1}(\frac{x}{\lambda}))\frac{2x\phi}{y_{\theta^\p}(\theta^\p)}dx\\
&=\frac{\pi\lambda\phi(\nu)}{y_{\theta^\p(\mu)}}-2\int_0^\lambda\sin^{-1}(\frac{x}{\lambda})\frac{d}{dx}\big(\frac{x\phi}{y_{\theta^\p}(\theta^\p)}\big)dx.
\end{align*}
Differentiating we get
\begin{align*}
&\partial_\nu\int_\nu^1\frac{\phi d\theta^\p}{\sqrt{y(\mu)-y(\theta^\p)}}=\frac{d\lambda}{d\nu}\frac{d}{d\lambda}\Bigg[\frac{\pi\lambda\phi(\nu)}{y_{\theta^\p}(\mu)}-2\int_0^\lambda\sin^{-1}(\frac{x}{\lambda})\frac{d}{dx}\big(\frac{x\phi(x)}{y_{\theta^\p}(\theta^\p)}\big)dx\Bigg]\\
&=\frac{-2}{\lambda}\frac{d\lambda}{d\nu}\int_0^\lambda\frac{x}{\sqrt{\lambda^2-x^2}}\frac{d}{dx}\big(\frac{x\phi(z)}{y_{\theta^\p}(\theta^\p)}\big)dx\\
&=\frac{-2}{\lambda}\frac{d\lambda}{d\nu}\int_\nu^1\frac{\sqrt{y(\theta^\p)-y(0)}}{\sqrt{y(\mu)-y(\theta^\p)}\sqrt{1-z^2}}\frac{d}{d\theta^\p}\big(\frac{\phi(\theta^\p)\sqrt{y(\theta^\p)-y(0)}}{y_{\theta^\p}(\theta^\p)}\big)dz.
\end{align*}
Since $\frac{d\lambda}{d\nu}=\frac{-y_{\theta^\p}(\mu)}{2\lambda\sqrt{1-\nu^2}}$ and in view of lemma $5$ this last expression is of the same order as
\begin{align*}
\frac{y_{\theta^\p}(\mu)}{rr^\p(1-\nu)\sqrt{1-\nu^2}}\int_\nu^1\frac{\frac{d}{d\theta^\p}\big(\frac{\phi(\theta^\p)\sqrt{y(\theta^\p)-y(0)}}{y_{\theta^\p(\theta^\p)}}\big)}{\sqrt{(z-\nu)(1+z)}}dz.
\end{align*}
$y_{\theta^\p}=2df(r^\p)\sin\alpha\sim rr^\p\sin\mu=rr^\p\sqrt{1-\nu^2}$ and therefore the expression above is equivalent to
\begin{align*}
\frac{1}{1-\nu}\int_\nu^1\frac{\frac{d}{d\theta^\p}\big(\frac{\phi(\theta^\p)\sqrt{y(\theta^\p)-y(0)}}{y_{\theta^\p(\theta^\p)}}\big)}{\sqrt{(z-\nu)(1+z)}}dz=: A+B,
\end{align*}
according to whether the $\theta^\p-$derivative falls on $\phi$ or the other term. We start by bounding $B$ which is more difficult.
\begin{lemma}
$\frac{d}{d\theta^\p}\big(\frac{\sqrt{y(\theta^\p)-y(0)}}{y_{\theta^\p}(\theta^\p)}\big)\lesssim \frac{1}{\sqrt{rr^\p}(1+\cos\theta^\p)}.$
\end{lemma}
Assuming the lemma is correct, and in view of the integral equality $\int\frac{dz}{(1+z)^{3/2}\sqrt{z-\nu}}=\frac{2\sqrt{z-\nu}}{(1+\nu)\sqrt{z+1}}$, we get
\begin{align*}
|B|\lesssim \frac{\|Q\|_{L^\infty(\theta^\p)}}{1-\nu}\int_\nu^1\frac{1}{\sqrt{rr^\p(z-\nu)}(1+z)^{3/2}}dz
\end{align*}
\begin{equation}\label{Bbound}
=\bigo(\frac{\A\B}{r^\p\sqrt{rr^\p}(1-\nu)^{1/2}(1+\nu)}).
\end{equation}
We start the proof of the lemma:
\begin{proof}
\begin{align*}
\frac{d}{d\theta^\p}\big(\frac{\sqrt{y(\theta^\p)-y(0)}}{y_{\theta^\p}(\theta^\p)}\big)=\frac{y^2_{\theta^\p}-2y_{\theta^\p\theta^\p}(y(\theta^\p)-y(0))}{2y^2_{\theta^\p}(\theta^\p)\sqrt{y(\theta^\p)-y(0)}}\sim\frac{y^2_{\theta^\p}-2y_{\theta^\p\theta^\p}(y(\theta^\p)-y(0))}{(rr^\p)^{5/2}(1-\nu^2)\sqrt{1-\nu}},
\end{align*}
and therefore to prove the lemma it suffices to show that
\begin{align*}
\frac{y^2_{\theta^\p}-2y_{\theta^\p\theta^\p}(y(\theta^\p)-y(0))}{(rr^\p)^2(1-\cos\theta^\p)^{3/2}}\lesssim 1.
\end{align*}
According to observation 1 it suffices to bound the differentiated fraction
\begin{align*}
\frac{y_{\theta^\p\theta^\p\theta^\p}(\theta^\p)(y(\theta^\p)-y(0))}{(rr^\p)^2(1+\cos\theta^\p)^{1/2}(1-\cos\theta^\p)}\sim\frac{y_{\theta^\p\theta^\p\theta^\p}(\theta^\p)}{rr^\p\sqrt{1+\cos\theta^\p}}
\end{align*}
for $\theta^\p \in [0,\pi].$ Denoting $\partial_{\theta^\p}$ by $\prime$ and noting that $y=d^2$ we have $y^{\p\p\p}=6d^\p d^{\p\p} +2dd^{\p\p\p}.$ We first compute
\begin{equation}\label{dprimeddoubleprime}
d^\p d^{\p\p}=f^2(r^\p)\cos\alpha\sin\alpha\alpha^{\p}(\theta^\p).
\end{equation}
For the second term we have
\begin{equation}\label{ddtripleprime}
dd^{\p\p\p}=-f(r^\p)d\sin\alpha\alpha_{\theta^\p}^2+df(r^\p)\cos\alpha\alpha^{\p\p}.
\end{equation}
To compute the second derivative of $\alpha$ we differentiate the Gauss-Bonnet equation (\ref{gaussbonnet}) twice to see that $\alpha^{\p\p}+\beta^{\p\p}=0.$ We can also differentiate (\ref{alphatobetader}) to write $\beta^{\p\p}$ in terms of $\alpha^{\p\p}:$
\begin{align*}
-\alpha^{\p\p}=\beta^{\p\p}=\frac{f(r^\p)\cos\alpha}{f(r)\cos\beta}\alpha^{\p\p}+\Big(\frac{f(r^\p)\cos\alpha}{f(r)\cos\beta}\Big)^\p\alpha^\p.
\end{align*}
So
\begin{align*}
&\Big(1+\frac{f(r^\p)\cos\alpha}{f(r)\cos\beta}\Big)\alpha^{\p\p}=\frac{f(r)f(r^\p)\sin\alpha\cos\beta\alpha^\p-f(r)f(r^\p)\cos\alpha\sin\beta\beta^\p}{f^2(r)\cos^2\beta}\alpha^\p\\
&=\frac{f^2(r)\sin\beta\cos^2\beta\alpha^\p-f^2(r^\p)\cos^2\alpha\sin\beta\alpha^\p}{f^2(r)\cos^3\beta}\alpha^\p,
\end{align*}
so
\begin{align*}
&\alpha^{\p\p}=\frac{\sin\beta(f^2(r)\cos^2\beta-f^2(r^\p)\cos^2\alpha)}{f(r)\cos^2\beta(f(r)\cos\beta+f(r^\p)\cos\alpha)}(\alpha^{\p})^2\\
&=\frac{\sin\beta(f^2(r)\cos^2\beta-f^2(r^\p)\cos^2\alpha)}{f^2(r)\cos^3\beta(\alpha^\p+\beta^\p)}(\alpha^\p)^3.
\end{align*}
Plugging this into (\ref{ddtripleprime}) we get
\begin{align*}
&(r^2\cos^3\beta(\alpha^\p+\beta^\p))(3d^\p d^{\p\p}+dd^{\p\p\p})=\Bigg[3f^2(r)f^2(r^\p)\cos\alpha\sin\alpha\cos^3\beta(\alpha^\p)^2\\
&+3f^2(r)f^2(r^\p)\cos\alpha\sin\alpha\cos^3\beta\alpha^\p\beta^\p-f(r^\p)f^2(r)d\sin\alpha\cos^3\beta(\alpha^\p)^3\\
&-f(r^\p)f^2(r)d\sin\alpha\cos^3\beta(\alpha^\p)^2\beta^\p +f^2(r)f(r^\p)d\cos\alpha\sin\beta\cos^2\beta(\alpha^\p)^3
\end{align*}
\begin{equation}\label{thelongexpression}
-f^3(r^\p)d\sin\beta\cos^3\alpha(\alpha^\p)^3\Bigg].
\end{equation}
Recall that we want to bound $(3d^\p d^{\p\p}+dd^{\p\p\p})/rr^\p\sqrt{1+\cos\theta}.$ Our strategy will be to replace $\alpha^\p$ and $\beta^\p$ by and $-f(r)\cos\beta/d$ and $-f(r^\p)\cos\alpha/d$ respectively in the expression above, and use lemma 6 to bound the corresponding error in each term. First we make these substitutions:
\begin{align*}
&\frac{(\ref{thelongexpression})}{r^\p r^3\cos^3\beta}=\frac{1}{d^2}\times \Bigg[3f(r)f(r^\p)\cos^2\beta\cos\alpha\sin\alpha\\ &+3f^2(r^\p)\cos^2\alpha\cos\beta\sin\alpha +f^2(r)\sin\alpha\cos^3\beta\\
&+f(r)f(r^\p)\sin\alpha\cos\alpha\cos^2\beta-f^2(r)\sin\beta\cos\alpha\cos^2\beta+f^2(r^\p)\sin\beta\cos^3\alpha\Bigg]\\
&=\frac{1}{d^2}\times\Big[2f(r)f(r^\p)\cos\beta\cos\alpha\sin(\alpha+\beta)\\
&+f^2(r)\cos^2\beta\sin(\alpha+\beta)+f^2(r^\p)\cos^2\alpha\sin(\alpha+\beta)\Big]\\
&=\frac{\sin(\alpha+\beta)(f(r^\p)\cos\alpha+f(r)\cos\beta)^2}{d^2}=\frac{\sin(\alpha+\beta)(f(r)\cos\beta(\frac{\alpha^\p+\beta^\p}{\alpha^\p}))^2}{d^2}\\
&\sim\sin(\alpha+\beta)(\alpha^\p+\beta^\p)^2.
\end{align*}
Since $\alpha^\p+\beta^\p$ is bounded, with $\alpha^\p=\frac{-f(r)\cos\beta}{d}$ and $\beta^\p=\frac{-f(r^\p)\cos\alpha}{d},$
\begin{align*}
\frac{y^{\p\p\p}}{rr^\p\sqrt{1+\cos\theta^\p}}\lesssim\frac{\sin(\alpha+\beta)}{\sqrt{1+\cos\theta^\p}}\lesssim\frac{\sin(\alpha+\beta)}{\sin\theta^\p}.
\end{align*}
It suffices to bound this fraction on $[0,\pi/3]\cup [2\pi/3,\pi],$ and since $\alpha+\beta\rightarrow 0,\pi$ as $\theta^\p\rightarrow \pi,0$ in view of the Gauss-Bonnet theorem, it suffices to consider the differentiated fraction
\begin{align*}
\frac{\cos(\alpha+\beta)(\alpha^\p+\beta^\p)}{\cos\theta^\p},
\end{align*}
which is bounded.\\\\
It remains to do the error analysis resulting from setting $\alpha^\p=\frac{-f(r)\cos\beta}{d}=:a$ and $\beta^\p=\frac{-f(r^\p)\cos\alpha}{d}=:b$. To this end, we divide the left hand side of (\ref{thelongexpression}) by $f^3(r)f(r^\p)\cos^3\beta\sqrt{1+\cos\theta^\p},$ and use lemma 6 to bound each of the resulting 6 terms. For simplicity of notation we set $K=1$ and write $r$ and $r^\p$ instead of $f(r)$ and $f(r^\p)$ respectively.\\\\
First term:\\\\
Note that
\begin{align*}
\alpha_{\theta^\p}^2-a^2=(\alpha^\p-a)(\alpha^\p+a)\sim r^2\cos^2\beta\big(\frac{d-\sin d}{d^3}\big)\lesssim r^2\cos^2\beta,
\end{align*}
where the last bound is obtained by applying l`Hopital's rule to the term in the parenthesis. The first term is therefore bounded by
\begin{align*}
\frac{r^4{r^\p}^2\cos\alpha\sin\alpha\cos^5\alpha}{r^3r^\p\cos^3\beta\sqrt{1+\cos\theta^\p}}\big(\frac{d-\sin d}{d^3}\big)\lesssim\frac{rr^\p\sin\alpha}{\sqrt{1+\cos\theta^\p}}.
\end{align*}
To bound this last expression, we only need to consider the region $\theta^\p\geq\frac{\pi}{2},$ where we bound the fraction by
\begin{align*}
\frac{rr^\p\sin\alpha}{\sin\theta^\p}\lesssim\frac{r^2r^\p}{d(\theta^\p)}\leq\frac{r^2r^\p}{d(\frac{\pi}{2})}\leq\frac{r^2r^\p}{d_K(\frac{\pi}{2})}\sim\frac{r^2r^\p}{\sqrt{r^2+{r^\p}^2}}\lesssim1,
\end{align*}
where $d_K$ is the corresponding edge in the comparison triangle with sides $r$ and $r^\p$ meeting at angle $\theta^\p.$\\\\
Second term:\\
\begin{align*}
&\alpha\p\beta^\p-ab=\alpha^\p(\beta^\p-b)+b(\alpha^\p-a)\sim\frac{rr^\p\cos\beta\cos\alpha}{d}\big(\frac{1}{d}-\frac{1}{\sin d}\big)\\
&\lesssim rr^\p\cos\alpha\cos\beta.
\end{align*}
It follows that the second term can be bounded by
\begin{align*}
\frac{{r^\p}^2\sin\alpha}{\sqrt{1+\cos\theta^\p}},
\end{align*}
which is bounded as before.\\\\
Third term:\\
\begin{align*}
&d\big({\alpha^\p}^3-a^3\big)=d\big(\alpha^\p({\alpha^\p}^2-a^2)+{\alpha^\p}^2(\alpha^\p-a)\big)\sim\frac{r^3\cos^3\beta}{d}\big(\frac{1}{d}-\frac{1}{\sin d}\big)\\
&\lesssim r^3\cos^3\beta,
\end{align*}
so the third term is bounded by
\begin{align*}
\frac{r^2\sin\alpha}{\sqrt{1+\cos\theta^\p}}\lesssim 1.
\end{align*}
Fourth and fifth terms:\\\\
Here again we have a factor of $d{\alpha^\p}^3,$ so these are bounded exactly in the same way as the third term.\\\\
Sixth term:\\
\begin{align*}
&d\big({\alpha^\p}^2\beta^\p-a^2b\big)=d\big({\alpha^\p}^2(\beta^\p-b)+b({\alpha^\p}^2-a^2)\big)\sim\frac{r^\p r^2\cos\alpha\cos^2\beta}{d}\big(\frac{1}{d}-\frac{1}{\sin d}\big)\\
&\lesssim r^\p r^2\cos\alpha\cos^2\beta,
\end{align*}
and therefore this term is bounded by
\begin{align*}
\frac{rr^\p\sin\alpha}{\sqrt{1+\cos\theta^\p}}\lesssim 1.
\end{align*}
\end{proof}
Bounding $A$ is easier. Since $y_{\theta^\p}\sim rr^\p\sin\theta^\p,~\frac{\sqrt{y(\theta^\p)-y(0)}}{y_{\theta^\p}}\lesssim \frac{1}{\sqrt{rr^\p(1+\nu)}}.$ Moreover since
\begin{align*}
Q(U)=B(U)(U_\xi,U_\eta)+\frac{1}{f^2}(lAU,lAU),
\end{align*}
$Q_{\theta^\p}$ is bounded just as $Q$ by $\frac{\A\B}{f(r^\p)}.$ It follows that $A$ is also bounded by the right hand side of (\ref{Bbound}).\\

We continue to bound $II$ over $K_1.$ Keep in mind that in this part of the cone

\begin{align*}
&|\nu-1|=|\cos\mu-\cos0|\geq c\big|\frac{y(\mu)-y(0)}{rr^\p}\big|= c\frac{(\xi-\xi^\p)(\eta -\eta^\p)}{rr^\p},\\
&|\nu+1|=|\cos\mu-\cos\pi|\geq c\big|\frac{y(\mu)-y(\pi)}{rr^\p}\big|=c\frac{(\eta^\p-\xi)(\eta-\xi^\p)}{rr^\p}.
\end{align*}

We now go back to (\ref{muetadichotomy}). The first term can be bounded just as in \cite{ST1}. Note that according to lemma $3.2$ in \cite{CT}
\begin{align*}
&P(\nu)=\bigo\Big(\log(1+\frac{1}{\sqrt{|1+\nu|}})\Big)=\bigo\Big(\log(1+c\sqrt{\frac{rr^\p}{(\eta^\p-\xi)(\eta-\xi^\p)}})\Big)\\
&=\bigo\Big(\log(1+c\sqrt{\frac{r}{(\eta^\p-\xi)}})\Big).
\end{align*}
Therefore keeping in mind the bound (\ref{nonlinearitybound}) on the nonlinearity,
\begin{align*}
&\Bigg|\int\int_{K_1}\frac{PQ}{r^{3/2}{r^\p}^{1/2}}f(r^\p)dr^\p dt^\p\Bigg|\lesssim r^{-3/2}\int\int_{K_1}\log(1+\frac{1}{\sqrt{\nu+1}})\frac{\A\B}{\sqrt{r^\p}}dr^\p dt^\p\\
&\lesssim \X r^{-3/2}\int_\xi^\eta\log(1+c\sqrt{\frac{r}{(\eta^\p-\xi)}})\Big(\int_{-2-\eta}^\xi (\eta^\p-\xi^\p)^{\delta-1}\B d\xi^\p\Big)d\eta^\p\\
&\lesssim \frac{\X}{r^{3/2}}\int_\xi^\eta\log(1+c\sqrt{\frac{r}{(\eta^\p-\xi)}})\times\\
&~~~~~~~~~~~~~~~~~~~~~\Big(\int_{-2-\eta}^\xi (\eta^\p-\xi^\p)^{2\delta-2} d\xi^\p\Big)^{1/2}\Big(\int_{-2-\eta}^\xi\B^2d\xi^\p\Big)^{1/2}d\eta^\p\\
&\lesssim \epsilon\X r^{-3/2}\int_\xi^\eta(\eta^\p-\xi)^{\delta -1/2}\log(1+c\sqrt{\frac{r}{(\eta^\p-\xi)}})d\eta^\p\\
&\lesssim \epsilon\X r^{-3/2}\Bigg[(\eta^\p-\xi)^{\delta+1/2}\log(1+c\sqrt{\frac{r}{\eta^\p-\xi}})\Big|_\xi^\eta+\tilde{c}\int_\xi^\eta\frac{\sqrt{r}(\eta^\p-\xi)^{\delta-1}}{1+c\sqrt{\frac{r}{\eta^\p-\xi}}}d\eta^\p\Bigg]\\
&\lesssim\epsilon\X r^{-3/2}[r^{\delta+1/2}+\sqrt{r}(\eta^\p-\xi)^\delta\big|_\xi^\eta]\lesssim\epsilon\X r^{\delta-1}.
\end{align*}
To bound the second term in (\ref{muetadichotomy}) we need to consider (\ref{Bbound}).
\begin{align*}
&\frac{(r-s)^2-{r^\p}^2}{r^2r^\p}\frac{1}{\sqrt{rr^\p}(1+\nu)(1-\nu)^{1/2}}\\
&\lesssim\frac{(\xi-\xi^\p)(\eta^\p-\xi)(rr^\p)^{3/2}}{r^2r^\p\sqrt{rr^\p}\sqrt{(\eta-\eta^\p)(\xi-\xi^\p)}(\eta^\p-\xi)(\eta-\xi^\p)}\\
&=\frac{(\xi-\xi^\p)^{1/2}(rr^\p)^{3/2}}{r^2r^\p\sqrt{rr^\p}(\eta-\xi^\p)\sqrt{\eta-\eta^\p}}\leq\frac{(rr^\p)^{3/2}}{r^2r^\p\sqrt{rr^\p}\sqrt{(\eta-\xi^\p)(\eta-\eta^\p)}}\\
&\leq\frac{1}{r\sqrt{(\eta^\p-\xi^\p)(\eta-\eta^\p)}}\lesssim\frac{1}{r{r^\p}^{1/2}\sqrt{\eta-\eta^\p}},
\end{align*}
and therefore the second term in (\ref{muetadichotomy}) can be bounded by
\begin{align*}
&r^{-1}\int\int_{K_1}\frac{\A\B}{\sqrt{r^\p(\eta-\eta^\p)}}d\xi^\p d\eta^\p\lesssim \X r^{-1}\int_{\xi}^{\eta}\frac{1}{\sqrt{\eta-\eta^\p}}\Big{(}\int_{-2-\eta}^{\xi}(\eta^\p-\xi^\p)^{\delta-1}\B d\xi^\p \Big{)}d\eta^\p\\
&\lesssim \epsilon r^{-1}\X\int_\xi^\eta\frac{(\eta^\p-\xi)^\delta}{\sqrt{(\eta^\p-\eta)(\eta^\p-\xi)}}d\eta^\p\lesssim\epsilon r^{\delta-1}\X\int_\xi^\eta\frac{1}{\sqrt{(\eta^\p-\eta)(\eta^\p-\xi)}}d\eta^\p.\\
&\lesssim \epsilon r^{\delta-1}\X
\end{align*}
The final step for obtaining Holder estimates is bounding $II$ in $K_2.$ Since the region of integration for $\theta^\p$ is fixed in this case (i.e. $0$ to $\pi$ independently of $s$ and $d$) we can use a direct comparison with the flat case and use the results from \cite{CT}. Keeping in mind the bound (\ref{nonlinearitybound}) on the nonlinearity, the term that needs to be bounded is
\begin{align*}
&\int_{K_2}\partial_\eta\Big[\int_0^\pi\frac{Qw^{+}f(r^\p)}{\sqrt{s^2-d^2}}d\theta^\p\Big] d\xi^\p d\eta^\p\\
&=\int_{K_2}\int_0^\pi\frac{(\partial_\eta w^{+})Qf(r^\p)}{\sqrt{s^2-d^2}}d\theta^\p d\xi^\p d\eta^\p + \int_{K_2}\int_0^\pi\partial_{\eta}\Big[\frac{1}{\sqrt{s^2-d^2}}\Big]Qw^{+}f(r^\p)d\theta^\p d\xi^\p d\eta^\p\\
&=\bigo\Bigg( \int_{K_2}\A\B\int_0^\pi \Big(\frac{1}{\sqrt{s^2-d^2}}+\Bigg|\partial_\eta \Big[\frac{1}{\sqrt{s^2-d^2}}\Big]\Bigg|\Big)d\theta^\p d\xi^\p d\eta^\p\Bigg).~~~~~~~~~~~~~~~~~~(*)
\end{align*}
Since $d_0\geq d$ we expect to be able to bound $\partial_\eta \Big[\frac{1}{\sqrt{s^2-d^2}}\Big]$ by $\partial_\eta \Big[\frac{1}{\sqrt{s^2-d_0^2}}\Big]$ (here the comparison triangle is the triangle with the angle at the pole equal to $\theta^\p$ and the sides equal to $r$ and $r^\p$). This is the next step.
\begin{lemma}
$\partial_\eta \Big[\frac{1}{\sqrt{s^2-d^2}}\Big]= \bigo\Bigg(\partial_\eta \Big[\frac{1}{\sqrt{s^2-d_0^2}}\Big]\Bigg).$
\end{lemma}
\begin{proof}
\begin{align*}
&\partial_\eta \Big[\frac{1}{\sqrt{s^2-d^2}}\Big]= \frac{-1}{2(s^2-d^2)^{3/2}}\Big((s+d)\partial_\eta(s-d)+(s-d)\partial_\eta(s+d)\Big).\\
&\partial_\eta \Big[\frac{1}{\sqrt{s^2-d_0^2}}\Big]= \frac{-1}{2(s^2-d_0^2)^{3/2}}\Big((s+d_0)\partial_\eta(s-d_0)+(s-d_0)\partial_\eta(s+d_0)\Big).
\end{align*}
 Taking note that $s\geq d_0 \geq d \geq 0$ and the fact that $\partial_\eta(s\pm d)=1\pm \cos\beta \geq 0$ we get
\begin{align*}
&\frac{\partial_\eta \Big[\frac{1}{\sqrt{s^2-d^2}}\Big]}{\partial_\eta \Big[\frac{1}{\sqrt{s^2-d_0^2}}\Big]}\lesssim\frac{\partial_\eta(s-d)}{\partial_\eta(s-d_0)}+ \frac{\partial_\eta(s+d)}{\partial_\eta(s+d_0)}\\
&= \frac{1-\cos\beta}{1-\cos\beta^\p_0} + \frac{1+\cos\beta}{1+\cos\beta^\p_0}.
\end{align*}
Here $\beta^\p_0$ is the corresponding angle (facing $r^\p$) in the flat triangle with sides of lengths $r,~r^\p,$ and $d_0.$ To bound the first term, note that for fixed $\theta^\p$ (and hence fixed $d$) the numerator is bounded by $1-\cos\beta_K$ where $\beta_K$ is the spherical angle facing $r^\p$ in the triangle with sides $r,~r^\p,$ and $d.$ In the following computation we set $K=1$ for simplicity of notation.
\begin{align*}
&\frac{1-\cos\beta}{1-\cos\beta^\p_0}\leq \frac{1-\cos\beta_K}{1-\cos\beta^\p_0}\\
&=\Big(\frac{\sin r\sin d +\cos r\cos d-\cos r^\p}{2rd_0-r^2-d^2_0+{r^\p}^2}\Big)\Big(\frac{2rd_0}{\sin r\sin d}\Big)\\
&\lesssim \Big(\frac{\cos(d-r)-\cos r^\p}{{r^\p}^2-(d_0-r)^2}\Big)(\frac{d_0}{d})\sim \Big|\frac{\sin(\frac{(d-r-r^\p)}{2})\sin(\frac{(d+r^\p-r)}{2})}{(d_0-r-r^\p)(d_0+r^\p-r)}\Big|(\frac{d_0}{d})\\
&\sim\Bigg|\Big(\frac{d-r-r^\p}{d_0-r-r^\p}\Big)\Big(\frac{d+r^\p-r}{d_0+r^\p-r}\Big)\Big(\frac{d_0}{d}\Big)\Bigg|.
\end{align*}
We need to bound these three terms as $\theta^\p$ ranges in the interval $[0,\pi].$ Consider the first two terms first. To bound these in the interior of the interval, we set the derivative equal to zero to find that at a point where the derivative is zero $\frac{d-r-r^\p}{d_0-r-r^\p}=\frac{\partial_{\theta^\p}(d-r-r^\p)}{\partial_{\theta^\p}(d_0-r-r^\p)}=\frac{f(r)\sin\beta}{r\sin\beta^\p_0}\sim\frac{\sin\beta}{\sin\beta^\p_0},$ and similarly for the second term. Since $d=d_0$ at the end points $\theta^\p=0,\pi,$ in view of l'Hopital's rule, the first two terms are bounded by $\frac{\sin\beta}{\sin\beta^\p_0}$ or $1$ at the end points as well. According to the flat and spherical laws of sines $\frac{\sin\beta}{\sin\beta^\p_0}\sim\frac{d_0}{d}$ (note that in our comparison triangles the angles at the north poles are equal). Therefore to bound $\frac{1-\cos\beta}{1-\cos\beta^\p_0}$ we only need to bound $\frac{d_0}{d}.$ Let $d_K$ be the length of the side facing $\theta^\p$ in the spherical triangle with sides $r$ and $r^\p.$ Then
\begin{align*}
&\frac{d^2_0}{d^2}\leq\frac{d_0^2}{d^2_K} \sim \frac{d^2_0}{{\sin}^2 d_K}=\frac{d_0^2}{(1-\cos d_K)(1+\cos d_K)} \lesssim \frac{d_0^2}{1-\cos d_K}\\
& = \frac{r^2+{r^\p}^2-2rr^\p\cos\theta^\p}{1-\cos r\cos r^\p-\sin r\sin r^\p\cos\theta^\p}\\
&=\frac{(r-r^\p)^2+2rr^\p(1-\cos\theta^\p)}{(1-\cos (r-r^\p))+\sin r\sin r^\p(1-\cos\theta^\p)}\\
&\lesssim 1+ \frac{(r-r^\p)^2}{1-\cos(r-r^\p)}\lesssim 1.
\end{align*}
To complete the proof of the lemma it only remains to bound $\frac{1+\cos\beta}{1+\cos\beta^\p_0}.$ This is similar to the computation we just did. Again for fixed $\theta^\p$ $1+\cos\beta \leq 1+\cos\beta_0,$ so by a similar argument as before
\begin{align*}
&\frac{1+\cos\beta}{1+\cos\beta^\p_0}\leq \frac{1+\cos\beta_0}{1+\cos\beta^\p_0}\\
&\lesssim \Bigg|\Big(\frac{r+d-r^\p}{r+d_0-r^\p}\Big)\Big(\frac{r+r^\p+d}{r+r^\p+d_0}\Big)\Big(\frac{d_0}{d}\Big)\Bigg|.
\end{align*}
The exact same argument as before shows that this is bounded.
\end{proof}
We can now finish the proof of our Holder estimates. The lemma implies that
\begin{align*}
&\int_{K_2}\A\B\int_0^\pi \Big(\frac{1}{\sqrt{s^2-d^2}}+\partial_\eta \Big[\frac{1}{\sqrt{s^2-d^2}}\Big]\Big)d\theta^\p d\xi^\p d\eta^\p\\
&\lesssim \int_{K_2}\A\B\int_0^\pi \Big(\frac{1}{\sqrt{s^2-d_0^2}}+\partial_\eta \Big[\frac{1}{\sqrt{s^2-d_0^2}}\Big]\Big)d\theta^\p d\xi^\p d\eta^\p.
\end{align*}
According to lemmas $8.5$ and $8.6$ in \cite{SS1}
\begin{align*}
\Bigg|\int_0^\pi \partial_\eta\Big(\frac{1}{\sqrt{s^2-d_0^2}}\Big)d\theta^\p\Bigg| \lesssim \frac{1}{|\eta-\eta^\p||\eta^\p-\xi|^{1/2}(r+r^\p)^{1/2}}.
\end{align*}
On the other hand since in $K_2$ there holds $s-d\geq t-t^\p-r-r^\p = \xi-\eta^\p$ and $s+d\geq s \geq r+r^\p,$ we have
\begin{align*}
\int_0^\pi\frac{d\theta^\p}{\sqrt{s^2-d_0^2}}\leq \frac{\pi}{\sqrt{(r+r^\p)(\xi-\eta^\p)}}\lesssim\frac{1}{|\eta-\eta^\p||\eta^\p-\xi|^{1/2}(r+r^\p)^{1/2}}.
\end{align*}
Plugging these back into the inequality we had obtained for $(*)$ we get
\begin{align*}
(*)\lesssim \int\int_{K_2}\frac{\A\B}{(\eta-\eta^\p)\sqrt{(r+r^\p)(\xi-\eta^\p)}}d\xi^\p d\eta^\p.
\end{align*}
We bound $\A$ by ${r^\p}^{\delta-1/2}\X$ and estimate the integral with respect to $\xi^\p$ as follows
\begin{align*}
&\int_{-2-\eta^\p}^{\eta^\p}\frac{\B {r^\p}^{\delta-1/2}}{\sqrt{r+r^\p}}d\xi^\p\\
&\leq\Bigg(\int_{-2-\eta^\p}^{\eta^\p}\B^2d\xi^\p\Bigg)^{1/2}\Bigg(\int_0^{1+\xi}\frac{{r^\p}^{2\delta-1}}{r+r^\p}dr^\p\Bigg)^{1/2}\\
&\leq\epsilon\Bigg(\int_0^{1+\xi}\frac{{r^\p}^{2\delta-1}}{r+r^\p}dr^\p\Bigg)^{1/2}.
\end{align*}
Keeping in mind that $|\xi|$ is small the integral can be bounded as
\begin{align*}
&\int_0^{1+\xi}\frac{{r^\p}^{2\delta-1}}{r+r^\p}dr^\p=\int_0^r\frac{{r^\p}^{2\delta-1}}{r+r^\p}dr^\p+\int_r^{1+\xi}\frac{{r^\p}^{2\delta-1}}{r+r^\p}dr^\p\\
&\leq\frac{1}{r}\int_0^r{r^\p}^{2\delta-1}dr^\p+\int_r^{1+\xi^\p}{r^\p}^{2\delta-2}dr^\p \lesssim r^{2\delta-1}.
\end{align*}
It follows that
\begin{align*}
&(*) \lesssim r^{\delta-1/2}\epsilon\X\int_{-1}^\xi\frac{d\eta^\p}{(\eta-\eta^\p)\sqrt{\xi-\eta^\p}} \leq r^{\delta-1/2}\epsilon\X\int_{-1}^\xi\frac{d\eta^\p}{(t-\eta^\p)\sqrt{\xi-\eta^\p}}\\
&= -2r^{\delta-1}\epsilon\X\tan^{-1}\Big(\sqrt{\frac{\xi-\eta^\p}{r}}\Big)\Bigg|^\xi_{-1} \lesssim r^{\delta-1}\epsilon\X.
\end{align*}
This completes the proof of (\ref{IIbound}) with $II_1$ replaced by $II,$ which together with (\ref{Ibound}) imply
\begin{align*}
\X \lesssim 1+\epsilon\X.
\end{align*}
For small enough $\epsilon$ we get the boundedness of $\X.$ As shown at the beginning of this section this implies Holder continuity of $u.$
\section{Higher Regularity}
\begin{center}

\end{center}

\vspace*{0.5cm}
We start by noting that for $r$ small enough, $u_1^2+u_2^2$ is small, and therefore $u_1$ and $u_2$ can be used as coordinates on the target sphere:
\begin{align*}
&u_1^2+u_2^2=2\int_0^r(u_1,u_2)\cdot(\partial_ru_1, \partial_r u_2)dr\\
&\leq\Bigg(\int|\nabla u|^2rdr\Bigg)^{1/2}\Bigg(\frac{|Au|^2}{r}dr\Bigg)^{1/2}\leq \mathrm{energy} \leq \epsilon^2.
\end{align*}
Letting $v=(u_1,u_2)$ the equation for $u$ becomes
\begin{equation}\label{vequation}
\left\{\begin{array}{rcl}v_{tt}-v_{rr}-\frac{f^\p}{f}v_r+\frac{l^2}{f^2}v=Q\\Q:=\big(-u_\xi\cdot u_\eta +\frac{l^2}{f^2}|v|^2\big)v\end{array}\right\}.
\end{equation}

\begin{proposition}
With $D_i$ denoting covariant differentiation with respect to the push forward of $\partial_i,$ the following intertwining relations hold:
\begin{align*}
&(i)\big[\partial_r+\frac{l}{f}\big]\big[\partial_r^2+\frac{f^\p}{f}\partial_r-\frac{l^2}{f^2}\big]-\big[\partial_r^2+\frac{f^\p}{f}\partial_r-\frac{(l-1)^2}{f^2}\big]\big[\partial_r+\frac{l}{f}\big]\\
&=\Big(\frac{(1-{f^\p}^2)+2l(f^\p-1)+ff^{\p\p}}{f^2}\Big)\Big(\partial_r+\frac{l}{f}\Big).\\
&(ii)\big[D_r+\frac{l}{f}\big]\big[\partial_t^2-\partial_r^2-\frac{f^\p}{f}\partial_r+\frac{l^2}{f^2}\big]-\big[D_t^2-D_r^2-\frac{f^\p}{f}D_r+\frac{(l-1)^2}{f^2}\big]\big[\partial_r+\frac{l}{f}\big]\\
&=D_t(II(\partial_t,\partial_r))-D_r(II(\partial_r,\partial_r))+\frac{l(\partial_\xi\partial_\eta-D_\xi D_\eta)}{f}\\
&-\Big(\frac{(1-{f^\p}^2)+2l(f^\p-1)+ff^{\p\p}}{f^2}\Big)\Big(\partial_r+\frac{l}{f}\Big).
\end{align*}
\end{proposition}
\begin{remark}
If we work on the the Minkowski plane with $l=0,$ the first part of the proposition reduces to
\begin{align*}
\partial_r\Box=(\Box+\frac{1}{r^2})\partial_r.
\end{align*}
This is used in \cite{CT} to deal with the $\theta-$invariant case. In the corotational case $l=1$ we have
\begin{align*}
(\partial_r+\frac{1}{r})(\partial_t^2-\partial_r^2-\frac{1}{r}\partial_r+\frac{1}{r^2})=(\partial_t^2-\partial_r^2-\frac{1}{r}\partial_r)(\partial_r+\frac{1}{r}).
\end{align*}
This is used in \cite{ST2} to prove higher regularity. The proposition is a generalization of these intertwining relations to a curved background with arbitrary rotation number $l$ and their extension to the covariant case.
\end{remark}
\begin{proof}
\begin{align*}
&\big[ \partial_r+\frac{l}{f}\big] \big[\partial^2_r+\frac{f^\p}{f}\partial_r-\frac{l^2}{f^2}\big]\\
&= \partial_r^3+\Big(\frac{f^{\p\p}f-{f^\p}^2}{f^2}\Big)\partial_r + \frac{f^\p}{f}\partial^2_r+\frac{2l^2f^\p}{f^3}-\frac{l^2}{f^2}\partial_r+\frac{l}{f}\partial^2_r+\frac{lf^\p}{f^2}\partial_r-\frac{l^3}{f^3}\\
&= \partial^3_r+\Big(\frac{f^\p+l}{f}\Big)\partial^2_r+\Big(\frac{f^{\p\p}f-{f^\p}^2+l(f^\p-l)}{f^2}\Big)\partial_r+\frac{l^2(2f^\p-l)}{f^3}.\\
\end{align*}
And
\begin{align*}
&\big[\partial^2_r+\frac{f^\p}{f}\partial_r-\frac{(l-1)^2}{f^2}\big]\big[\partial_r+\frac{l}{f}\big]\\
&=\partial^3_r+\frac{l}{f}\partial^2_r-\frac{2lf^\p}{f^2}\partial_r+\frac{l(2{f^\p}^2-f^{\p\p}f)}{f^3}+\frac{f^\p}{f}\partial_r^2+
\frac{lf^\p}{f^2}\partial_r- \frac{{lf^\p}^2}{f^3}\\
&-\frac{(l-1)^2}{f^2}\partial_r-\frac{l(l-1)^2}{f^3}\\
&=\partial^3_r+\Big(\frac{l+f^\p}{f}\Big)\partial_r^2-\frac{lf^\p+(l-1)^2}{f^2}\partial_r+\frac{l({f^\p}^2-f^{\p\p}f)-l(l-1)^2}{f^3}.
\end{align*}
Therefore
\begin{align*}
&\big[\partial_r+\frac{l}{f}\big]\big[\partial_r^2+\frac{f^\p}{f}\partial_r-\frac{l^2}{f^2}\big]-\big[\partial_r^2+\frac{f^\p}{f}\partial_r-\frac{(l-1)^2}{f^2}\big]\big[\partial_r+\frac{l}{f}\big]\\
&=\Big(\frac{(1-{f^\p}^2)+2l(f^\p-1)+ff^{\p\p}}{f^2}\Big)\Big(\partial_r+\frac{l}{f}\Big).
\end{align*}
For part (ii) let $\mathrm{proj}$ denote the projection on the tangent bundle of the target.
\begin{align*}
&\big[D_r+\frac{l}{f}\big]\big[\partial_t^2-\partial_r^2-\frac{f^\p}{f}\partial_r+\frac{l^2}{f^2}\big]\\
&=\mathrm{proj}(\partial_t\partial_{tr})-D_r(D_r\partial_r+II(\partial_r,\partial_r))-\frac{f^\p}{f}D_r\partial_r+\Big(\frac{{f^\p}^2-f^{\p\p}f}{f^2}\Big)\partial_r+\frac{l^2}{f^2}D_r\\
&-\frac{2l^2f^\p}{f^3}+\frac{l\partial_t^2}{f}-\frac{l\partial_r^2}{f}-\frac{lf^\p}{f^2}\partial_r+\frac{l^3}{f^3}\\
&=D_t^2\partial_r-D_r^2\partial_r-\frac{f^\p}{f}D_r\partial_r+\frac{l(\partial_t^2-\partial_r^2)}{f}+\frac{l^2}{f^2}D_r+\Big(\frac{{f^\p}^2-f^{\p\p}f-lf^\p}{f^2}\Big)\partial_r\\
&+\frac{l^2(l-2f^\p)}{f^3}+D_t(II(\partial_t,\partial_r))-D_r(II(\partial_r,\partial_r)).
\end{align*}
And
\begin{align*}
&\big[D_t^2-D_r^2-\frac{f^\p}{f}D_r+\frac{(l-1)^2}{f^2}\big]\big[\partial_r+\frac{l}{f}\big]\\
&=D_t^2\partial_r+\frac{lD^2_t}{f}-D^2_r\partial_r-\frac{l}{f}D^2_r+\frac{2lf^\p}{f^2}D_r+\frac{l(f^{\p\p}f-2{f^\p}^2)}{f^3}-\frac{f^\p}{f}D_r\partial_r-\frac{lf^\p}{f^2}D_r\\
&+\frac{l{f^\p}^2}{f^3}+\frac{(l-1)^2}{f^2}\partial_r+\frac{l(l-1)^2}{f^3}\\
&= D_t^2\partial_r-D_r^2\partial_r-\frac{f^\p}{f}D_r\partial_r+\frac{l(D_t^2-D_r^2)}{f}+\frac{lf^\p}{f^2}D_r+\frac{(l-1)^2}{f^2}\partial_r\\
&+\frac{l(f^{\p\p}f-{f^\p}^2+(l-1)^2)}{f^3}.
\end{align*}
Therefore noting that on the level of functions $D_r$ is the same as $\partial_r$
\begin{align*}
&\big[D_r+\frac{l}{f}\big]\big[\partial_t^2-\partial_r^2-\frac{f^\p}{f}\partial_r+\frac{l^2}{f^2}\big]-\big[D_t^2-D_r^2-\frac{f^\p}{f}D_r+\frac{(l-1)^2}{f^2}\big]\big[\partial_r+\frac{l}{f}\big]\\
&=D_t(II(\partial_t,\partial_r))-D_r(II(\partial_r,\partial_r))+\frac{l(\partial_\xi\partial_\eta-D_\xi D_\eta)}{f}\\
&-\Big(\frac{(1-{f^\p}^2)+2l(f^\p-1)+ff^{\p\p}}{f^2}\Big)\Big(\partial_r+\frac{l}{f}\Big).
\end{align*}
\end{proof}
In view of the first intertwining relation, we apply the differential operator $\partial_r +\frac{l}{f}$ to the equation for $v$ and let $w=v_r+\frac{lv}{f}$ to get:
\begin{equation}\label{wequation}
\left\{\begin{array}{rcl}&w_{tt}-w_{rr}-\frac{f^\p}{f}w_r+\frac{(l-1^2)w}{f^2}=F\\&F = \bigo\Big(|v|\big[|w_\xi||u_\eta|+|w_\eta||u_\xi|\\&~~~~~~~~~~~~~~~~~~~~~~~~+\frac{|v_\eta||v_\xi|}{r}+ \frac{|v|(|u_\eta|+|u_\xi|)}{r^2}+\frac{|v|^2}{r^3}\\&~~~~~~~~~~~~~~~~~~~~~~~~+\frac{|v||v_r|}{r^2}\big]+\frac{|v|^2|w|}{r^2}+|w||u_\xi||u_\eta|+|w|\Big) \end{array}\right\}.
\end{equation}
To see that $F$ satisfies this estimate note that
\begin{align*}
&F-\frac{({f^\p}^2-1)+2l(1-f^\p)-ff^\p}{r^2}w= \Big(\partial_r+\frac{l}{f}\Big)\big((-u_\eta\cdot u_\xi +\frac{l^2|v|^2}{f^2})v\big)\\
&=(-u_\eta\cdot u_\xi +\frac{|v|^2}{f^2})w+(-u_{r\xi}\cdot u_\eta-u_\xi \cdot u_{r\eta}+\frac{2f^\p l^2|v|^2}{f^3}+\frac{2l^2v\cdot v_r}{f^2})v.
\end{align*}
Now $|u_{r\xi}| \lesssim |v_{r\xi}|$ and $|u_{r\eta}| \lesssim |v_{r\eta}|.$ Moreover since $v_r = w - \frac{lv}{f},$ we have the bound:
\begin{align*}
|u_{r\xi}|\lesssim |w_\xi| + \frac{|v_\xi|}{f}+\frac{|f^\p v|}{f^2},
\end{align*}
and similarly
\begin{align*}
|u_{r\eta}|\lesssim |w_\eta| + \frac{|v_\eta|}{f}+\frac{|f^\p v|}{f^2}.
\end{align*}
The bound on $F$ now follows from the following estimates coming from the expansion for $f$ in the section on the existence of stationary maps
\begin{equation}\label{fbounds}
\frac{(1-f^\p)}{r^2}\lesssim 1,~\frac{1-{f^\p}^2}{r^2}\lesssim 1,~\frac{ff^{\p\p}}{r^3}\lesssim \frac{1}{r}.
\end{equation}
To prove higher regularity we want to bound $w_\eta$ in the same way as we bounded $u_\eta.$ To this end we define
\begin{equation}\label{atildebtilde}
\left\{\begin{array}{rcl}&\tilde{\A}^2 := f(r)(\tilde{e}+\tilde{m})= \frac{f}{2}(|\partial_\eta w|^2 + \frac{(l-1)^2|w|^2}{f^2})\\
&\tilde{B}^2 := f(r)(\tilde{e}-\tilde{m})= \frac{f}{2}(|\partial_\xi w|^2 + \frac{(l-1^2)|w|^2}{f^2})\\
&\Y(\tilde{t}) = \sup_{Z(\tilde{t})}\big\{(f(r))^{(\frac{1}{2}-\alpha)}\tilde{\A}(t,r)\big\},~\alpha\in (0,1/2)~fixed\\
& \Y = \Y(0) \end{array}\right\}.
\end{equation}
The main difference here is that since the energy of $w$ is no longer small the corresponding flux cannot be bounded by a small constant. This problem will surface in bounding the terms involving $w_\xi.$ There are two approaches for dealing with this issue. One is to perform an integration by parts and move the $\xi$ derivative on the fundamental solution. This involves finding estimates on the second derivative of the fundamental solution, and in the flat case it is carried out in \cite{CT}. The simpler method is based on covariantly differentiating the equation for $v.$ This will allow us to bound the terms involving $w_\xi$ in terms of $\Y$ and the other terms appearing in the nonlinearity. In the flat case this is done in \cite{SS1}, and this is the approach we choose to follow.\\
Note that since $w$ vanishes at $r=0,$ we can bound it as
\begin{equation}\label{wbound}
|w|\lesssim r^\alpha \Y.
\end{equation}
Moreover, multiplying $w=v_r+\frac{lv}{f}$ by $e^{\int\frac{ldr}{f}}$ we get $\partial_r(e^{\int\frac{ldr^\p}{f}}v)=e^{\int\frac{ldr}{f}}w$ and therefore
\begin{align*}
e^{\int\frac{ldr}{f}}v=\int_0^re^{\int\frac{ldr^{\p\p}}{f}} w dr^\p.
\end{align*}
Together with (\ref{wbound}) this implies that
\begin{equation}\label{vbound}
\frac{|v|}{r}\lesssim r^\alpha \Y,
\end{equation}
and since $v_r= w-\frac{lv}{r},$
\begin{equation}\label{vrbound}
|v_r|\lesssim r^\alpha\Y.
\end{equation}
Applying $D_r+\frac{l}{f}$ to (\ref{vequation}) and using part (ii) of the proposition above we get:
\begin{equation}\label{wcovariantequation}
\left\{\begin{array}{rcl}&D_t^2w-D^2_r w -\frac{f^\p}{f}D_rw+\frac{(l-1)^2w}{f^2}= G\\
&G = \bigo\Big(\frac{|w||v|^2}{r^2}+|u_\xi||u_\eta||w|\\
&
+((dII)(u_t)(u_t,u_r)-(dII)(u_r)(u_r,u_r))\\
&+\frac{|(\partial_\xi \partial_\eta - D_\xi D_\eta)v|}{r}+ |w|\Big)\end{array}\right\}.
\end{equation}
Here we have used the fact that since $v$ is normal to the target, $D_r(fv)= fD_rv$ for any function $f$. Similarly since the second fundamental form is normal to the the target $D_{\partial_k}II(u)(u_i,u_j)= (dII)(u_k)(u_i,u_j).$ \\
Since
\begin{align*}
&(dII)(u_t)(u_t,u_r)-(dII)(u_r)(u_r,u_r)\\
&= \frac{1}{2}\big[(dII)(u_\xi)(u_\eta,u_r)-(dII)(u_\eta)(u_\xi,u_r)\big]=\bigo(|u_\eta||u_\xi||v_r|),
\end{align*}
$G$ is
\begin{equation}\label{Gbound}
\bigo\Big(\frac{|w||v|^2}{r^2}+|u_\xi||u_\eta||w|+|u_\eta||u_\xi||v_r|+\frac{|(\partial_\xi \partial_\eta - D_\xi D_\eta)v|}{r}+  |w|\Big).
\end{equation}
Note that $w_\xi$ does not appear in this expression which is what we had hoped to gain by differentiating covariantly.\\ \\
As mentioned before, the main difficulty is dealing with terms involving $w_\xi.$ We need to gain control over the flux associated with $w$ and for this we define
\begin{align*}
\Z^2:= \sup_{(t,r)\in K}\int_{\{\eta^\p=\eta, -2-\eta\leq\xi^\p\leq\xi\}}\Big(|w_\xi|^2+\frac{(l-1)^2|w|^2}{f^2}\Big) r^\p d\xi^\p.
\end{align*}
The key step in the proof of higher regularity is the following
\begin{lemma}
$\Z\lesssim \Y+1.$
\end{lemma}
\begin{proof}
The proof is an energy conservation argument based on the divergence theorem. Let $n=l-1$ and multiply equation (\ref{wcovariantequation}) by $D_tw$ to get
\begin{align*}
&\mathrm{div}_{t,r,\theta^\p}\Big(\frac{|D_tw|^2+|D_rw|^2+|\frac{nw}{f}|^2}{2}, -D_tw\cdot D_rw,0\Big)=G\cdot D_tw.
\end{align*}
Therefore with $K$ denoting the backward light cone (only in the $r$ and $t$ variables), $C$ its lateral boundary and $B_t$ the cross section at time $t$
\begin{align*}
&\int_KG\cdot D_tw~ r^\p dr^\p dt^\p = \int_C\Big(|D_\xi w|^2+\frac{(l-1)^2|w|^2}{f^2}\Big) r^\p d\mu_C\\
\end{align*}
\begin{equation}\label{gronwall}
+\frac{1}{2}\int_{B_{-1}}(|D_tw|^2+|D_rw|^2+\frac{(l-1)^2|w|^2}{f^2})r^\p dr^\p.
\end{equation}
Here we have integrated and canceled out the contribution of the $\theta$ variable. With $C_t$ denoting the truncated lateral boundary between $B_{-1}$ and $B_t,$ another application of the divergence theorem gives, for any $-1<\tau<0$
\begin{align*}
&\int_{B_\tau}(|D_tw|^2+|D_rw|^2+\frac{(l-1)^2|w|^2}{f^2})r^\p dr^\p\\
&\lesssim \int_{B_{-1}}(|D_tw|^2+|D_rw|^2+\frac{(l-1)^2|w|^2}{f^2})r^\p dr^\p\\
&+\int_{C_\tau}(|D_\xi w|^2+\frac{(l-1)^2|w|^2}{f^2})r^\p d\xi^\p\\
&+\int_K|G|^2r^\p dr^\p dt^\p +\int_{K_\tau}|D_tw|^2dr^\p dt^\p.
\end{align*}
Defining $\Lambda(t)=\int_{K_t}|D_tw|^2$ and integrating the inequality above we get
\begin{align*}
&\Lambda(t)\lesssim 1+\int_K|G|^2r^\p dr^\p dt^\p\\
&+\int_{-1}^t\int_{C_\tau}(|D_\xi w|^2+\frac{(l-1)^2|w|^2}{f^2})r^\p d\xi^\p d\tau+\int_{-1}^t\Lambda(\tau)d\tau,
\end{align*}
and therefore by Gronwall's inequality
\begin{align*}
\Lambda(t)\lesssim1+\int_K|G|^2)r^\p dr^\p dt^\p+ \int_{-1}^t\int_{C_\tau}(|D_\xi w|^2+\frac{(l-1)^2|w|^2}{f^2})r^\p d\xi^\p d\tau.
\end{align*}
Putting these together with (\ref{gronwall}) we get
\begin{align*}
&\int_{C}(|D_\xi w|^2+\frac{(l-1)^2|w|^2}{f^2}r^\p d\xi^\p\\
&\lesssim 1+\int_K|G|^2 r^\p d\xi^\p d\eta^\p  +\int_{-1}^0\Big(\int_{C_t}(|D_\xi w|^2+\frac{(l-1)^2|w|^2}{f^2})r^\p d\xi^\p\Big)dt.
\end{align*}
Similarly for any $-1<\tau<0$
\begin{align*}
&\int_{C_{\tau}}(|D_\xi w|^2+\frac{(l-1)^2|w|^2}{f^2})r^\p d\xi^\p\\
&\lesssim 1+\int_K|G|^2 r^\p d\xi^\p d\eta^\p  +\int_{-1}^\tau\Big(\int_{C_t}(|D_\xi w|^2+\frac{(l-1)^2|w|^2}{f^2})r^\p d\xi^\p\Big)dt,
\end{align*}
and by Gronwall's inequality
\begin{align*}
\int_{C}(|D_\xi w|^2+\frac{(l-1)^2|w|^2}{f^2})r^\p d\xi^\p\lesssim 1+ \int_K|G|^2r^\p d\xi^\p d\eta^\p.
\end{align*}
We now claim that
\begin{align*}
\int_{K}|G|^2r^\p d\xi^\p d\eta^\p \lesssim 1+ \Y^2.
\end{align*}
By the definition of covariant differentiation $|\partial_\xi w|^2\lesssim |D_\xi w|^2 + |u_\xi|^2|w|^2$ and therefore assuming the claim
\begin{align*}
\int_{C}(|\partial_\xi w|^2+\frac{(l-1)^2|w|^2}{f^2})r^\p d\xi^\p \lesssim 1+ \Y^2.
\end{align*}
Taking the supremum over $r$ and $t$ we obtain the desired bound
\begin{align*}
\Z \lesssim \Y +1.
\end{align*}
To prove the claim, we treat the terms in (\ref{Gbound}) separately. First by (\ref{wbound}), (\ref{vrbound}), and the bounds from the previous section
\begin{align*}
\int_K|u_\xi|^2|u_\eta|^2(|w|^2+|v_r|^2)r^\p d\xi^\p d\eta^\p \lesssim \Y^2\int_K\frac{|u_\xi|^2r^\p}{{r^\p}^{(2-2\delta-2\alpha)}}d\xi^\p d\eta^\p.
\end{align*}
According to equation (\ref{fluxbound}) from the previous section, if $1-\delta-\alpha<\delta$ the above integral is bounded.\\
Next since $|v|\lesssim r^\delta,$
\begin{align*}
\Big(\frac{|w||v|^2}{r^2}\Big)^2\lesssim r^{2\alpha+4\delta-4}\Y^2.
\end{align*}
Again if $\alpha+2\delta-1>0,$ the integral of ${r^\p}^{2\alpha+4\delta-4}$ over $K$ with respect to the measure $r^\p dt^\p dr^\p$ is finite and this takes care of the $\frac{|w||v|^2}{r^2}$ term. $|w|^2\lesssim r^{2\alpha}\Y^2$ which integrates to a finite multiple of $\Y^2,$ so it only remains to consider $\frac{\big|\partial_\xi\partial_\eta v- D_\xi D_\eta v\big|}{r}.$ For this we need a precise geometric computation. Remember that for us $v(\xi,\eta)=(v_1(\xi,\eta),v_2(\xi,\eta))$ is a map from $\RR^2$ to $\RR^2,$ where the target $\RR^2$ is equipped with the metric
\begin{align*}
g=\frac{1}{(1-x^2-y^2)}\Bigl(\begin{matrix} 1-y^2 & xy\\ xy & 1- x^2\end{matrix}\Bigr).
\end{align*}
This is the metric on $S^2$ under the parametrization
\begin{align*}
(x,y)\rightarrow (x,y,\sqrt{1-x^2-y^2}).
\end{align*}
For future use, we note that the inverse of $g$ is given by
\begin{align*}
g^{-1}=\Bigl(\begin{matrix}1- x^2 & -xy\\ -xy & 1- y^2\end{matrix}\Bigr).
\end{align*}
Now with $x_1=x$ and $x_2=y$
\begin{align*}
&D_\xi D_\eta v = D_\xi(\frac{\partial v}{\partial \eta}) = D_\xi(\sum_i \frac{\partial v_i}{\partial\eta} \frac{\partial}{\partial x_i})\\
&=\sum_i\frac{\partial^2v_i}{\partial\xi\partial\eta}\frac{\partial}{\partial x_i}+\sum_{i,j}\frac{\partial v_j}{\partial\xi}\frac{\partial v_i}{\partial\eta}\nabla_{\frac{\partial}{\partial x_j}}\frac{\partial}{\partial x_i},
\end{align*}
whereas
\begin{align*}
\partial_\xi\partial_\eta v= \sum_i\frac{\partial^2v_i}{\partial\xi\partial\eta}\frac{\partial}{\partial x_i}.
\end{align*}
It follows that
\begin{align*}
(D_\xi D_\eta-\partial_\xi\partial_\eta)v= \sum_{i,j}\frac{\partial v_j}{\partial\xi}\frac{\partial v_i}{\partial\eta}\nabla_{\frac{\partial}{\partial x_j}}\frac{\partial}{\partial x_i}.
\end{align*}
Note that everywhere in these computations $\frac{\partial}{\partial x_i}= \frac{\partial}{\partial x_i}(v).$ It follows that if we can show that
\begin{align*}
\nabla_{\frac{\partial}{\partial x_j}}\frac{\partial}{\partial x_i} = \bigo(|X|) ~~\mathrm{as}~|X|=|(x,y)|\rightarrow 0, \forall i,j,
\end{align*}
then
\begin{align*}
\frac{\big|(D_\xi D_\eta-\partial_\xi\partial_\eta)v\big|}{r}\lesssim \frac{|v_\eta||v_\xi||v|}{r},
\end{align*}
which in view of (\ref{vbound}) can be dealt with in the same was as $|u_\xi||u_\eta||w|.$ For this it suffices to show that the Christoffel symbols $\Gamma^m_{ij}(X)$ are $\bigo(|X|)$ for all $i,j,$ and $m.$ Since
\begin{align*}
\Gamma^m_{ij} = \frac{1}{2}\sum_k\Big{\{}\frac{\partial g_{jk}}{\partial x_i}+\frac{\partial g_{ki}}{\partial x_j}-\frac{\partial g_{ij}}{\partial x_k}\Big{\}}g^{km},
\end{align*}
and since the components of $g^{-1}$ are bounded near the origin, it suffices to prove that $\frac{\partial g_{ki}}{\partial x_j}(X) = \bigo(|X|)$ for all $i,j,$ and $k.$ This is a direct derivative computation. By the symmetry of the components of $g$ in $x$ and $y$ the following three computations finish the proof:
\begin{align*}
&\frac{\partial g_{xx}}{\partial x}= \frac{\partial}{\partial x}\big[\frac{1-y^2}{1-y^2-x^2}\big]= \frac{2x(1-y^2)}{(1-x^2-y^2)^2} = \bigo(|X|),\\
&\frac{\partial g_{xx}}{\partial y}= \frac{\partial}{\partial y}\big[\frac{1-y^2}{1-y^2-x^2}\big]= \frac{2yx^2}{(1-x^2-y^2)^2} = \bigo(|X|),\\
&\frac{\partial g_{xy}}{\partial x}= \frac{\partial}{\partial x}\big[\frac{xy}{1-y^2-x^2}\big]= \frac{y(1+x^2-y^2)}{(1-x^2-y^2)^2} = \bigo(|X|).
\end{align*}
\end{proof}
Now notice that $W_1:=w\cos (l-1)\theta$ and $W_2:=w\sin(l-1)\theta$ satisfy $\Box W_1=F\cos (l-1)\theta$ and $\Box W_2=F\sin(l-1)\theta.$ Since $w=W_1\cos (l-1)\theta+W_2 \sin(l-1)\theta$ it suffices to obtain holder bounds on $W_1$ and $W_2$ instead of $w.$ This implies that there is no loss of generality in assuming that $w$ satisfies $\Box w = F.$ The bounds (\ref{vbound}), (\ref{vrbound}), and (\ref{wbound}) together with (\ref{wequation}) imply that $F$ is bounded by a combination of terms which are bounded by one of $\frac{\A\B\Y}{r},\frac{\tilde{\A}\B}{r},\frac{\tilde{\B}\A}{r}$ or $r^\alpha\Y.$ Therefore if we choose $\alpha < \delta,$ by the same exact argument as in the previous section
\begin{align*}
\Y\lesssim 1+ \epsilon(\Z+\Y).
\end{align*}
Indeed the first two terms yield a factor of $\epsilon$ because $\B$ appears in them. For the third term the argument from the previous section shows that $|\partial_\eta w|\lesssim r^{\delta-1}\Z,$ so we get the desired bound if we choose $\alpha<\delta$ and $r$ so small that $r^{\delta-\alpha}<\epsilon.$ Finally recall that in the previous section where the nonlinearity was bounded by $\frac{\A\B}{r},$ we replaced $\A$ by its upper bound $r^{\delta-\frac{1}{2}}$ which yielded $|u_\eta|\lesssim \epsilon\X r^{1-\delta}$. Since $r^\alpha\Y\lesssim r^{\delta-\frac{3}{2}}\Y,$ by the same argument this term contributes $r^{\delta-\alpha}\Y$ which is good as long as $\alpha<\delta$. Together with the lemma this implies that upon choosing $\epsilon$ small enough, $\Y$ and $\Z$ are bounded. It follows that $|w|,$ and in particular $|v_r|,$ are bounded. Moreover the bound on $u_\eta$ from the previous section implies that for a fixed $r=R, |v_\eta(t,R)|\lesssim R^{\delta-1},$ and therefore
\begin{align*}
|v_\eta(t,r)|\lesssim R^{\delta-1}+\int_R^r|\partial_\eta v_r(t,r^\p)|dr^\p\lesssim R^{\delta-1}+\int_R^r r^{\alpha-1}dr^\p\lesssim 1.
\end{align*}
Since $\partial_r$ and $\partial_\eta$ are linearly independent, we conclude that the full gradient of $v$ is bounded, and higher regularity follows.
\section{Stability}
\begin{center}

\end{center}

\vspace*{0.5cm}
In this section we prove a stability result for the stationary maps of the first section. The proof here is almost word by word the same as that in \cite{ST1} and I reproduce it only for completeness. I will also provide more details in a few places. We start by introducing some notation.\\
Let $\cS$ be the set of minima of $\cGo$ in $X_l.$ Let $Y:=H^1(S^2)\times L^2 (S^2)$ with the product norm, and let $d$ denote the associated distance function $d(u,v)=\|u-v\|_Y$ for $u,v \in Y.$ For any compact $K$ in $Y$ we define as usual $d(u,K):=\inf_{v\in K}d(u,v).$\\
Let $\mathfrak{S}:= \{(v,\omega Av)\}_{v\in\cS}.$ Our results from the elliptic section show that $\mathfrak{S}$ is compact in $Y.$ Finally for any finite energy map $u:S^2\times \RR \rightarrow S^2$ let $\uu^t:=(u(t),\partial_t u(t)) \in Y.$
\begin{theorem}
There exists $\eta>0$ such that for all $T^*>0,$ if $u$ is a classical solution to our Cauchy problem on $[0,T^*)$ satisfying $\sup_{t\in [0,T*)}d(\uu^t,\mathfrak{S})<\eta,$ then $u$ is smooth in $S^2\times [0,T^*].$
\end{theorem}
\begin{proof}
Since the equation is radial we know that the first singularity, if there is one, will happen at one of the poles, which we take to be the north pole without loss of generality. For $t\in [0,T^*)$ and $R$ small, let
\begin{align*}
D_R(t):=\{(r^\p,\theta^\p,t)\big|0\leq r^\p\leq R\},
\end{align*}
and
\begin{align*}
E_R(u(t)):=\frac{1}{2}\int_{D(t)}(|\nabla u(t)|^2+|u_t(t)|^2)dg_{S^2}.
\end{align*}
Local conservation of energy implies
\begin{equation}\label{energyinequality}
E_R(u(T_1))\geq E_{R+T_1-T_2}(u(T_2)),
\end{equation}
for $T_1\leq T_2\leq T^*.$ The hypothesis of the theorem implies that for every $t$ we can find $v^t\in\cS$ such that
\begin{align*}
\|u(t)-v^t\|_{H^1(S^2)}+\|u_t(t)-\omega Av^t\|_{L^2(S^2)}<\eta^2.
\end{align*}
Of course the same inequality holds if we replace $S^2$ by $D_{T*-t}(t)$ and therefore
\begin{align*}
2E_{T^*-t}(u(t))\leq \|\nabla v^t\|^2_{L^2(D_{T^*-t}(t))} +\|\partial_t v^t\|^2_{L^2(D_{T^*-t}(t))} +\eta^2.
\end{align*}
It follows that since $\cS$ is compact we can find a $T^\p$ close enough to $T^*$, such that if we choose $\eta$ very small we have $E_{T^*-t}(u(t))<\epsilon/2$ for all $t \in [T^\p,T^*).$ Therefore if $\delta$ is a very small positive number we have $E_{T^*-T^\p+\delta}(u(T^\p))<\epsilon.$ It follow from (\ref{energyinequality}) that for all $T\in[T^\p,T^*)$
\begin{align*}
E_{\delta}(u(T))\leq E_{T^*-T+\delta}(u(T))\leq E_{T^*-T^\p+\delta}(u(T^\p))<\epsilon.
\end{align*}
Since this inequality holds for all $T$ near $T^*,$ according to our regularity theorem from the hyperbolic section the solution cannot develop a singularity at $T^*.$
\end{proof}
We are now ready for the stability theorem
\begin{theorem}
Let $u^0 \in \cS$ be fixed and let $\tu$ be the corresponding stationary and equivariant map
\begin{align*}
\tu (x,t) := e^{A\omega t}u^0(x).
\end{align*}
Then $\tu$ is stable in the following sense: Let $\eta$ be as in the previous theorem. For every $\epsilon \in (0,\eta)$ there exists a $\delta$ such that if $(f,g)$ is a pair of equivariant initial data for our Cauchy problem satisfying
\begin{align*}
d((f,g),(u^0,\omega A u^0))<\delta,
\end{align*}
then the classical solution $u$ for this pair of initial data defined on $[0,T^*)\times S^2$ remains $\epsilon-$close to $\cS$ in the energy norm for all $t\in[0,T^*),$ i.e.
\begin{align*}
d(\uu^t,\mathfrak{S})<\epsilon,~~~~~~~~ \forall t\in[0,T^*).
\end{align*}
\end{theorem}
Note that combined with the previous theorem this theorem implies that $T^*=\infty$ so $u$ is a globally smooth solution.
\begin{proof}
We argue by contradiction. If the statement of the theorem is false then for some $\epsilon < \eta$ we can find a sequence of equivariant initial data $(f_n,g_n),$ with corresponding solutions $u_n,$ satisfying $d((f_n,g_n),(u^0,\omega Au^0))<\frac{1}{n},$ and a sequence of times $t_n$ such that for all $v \in \cS$
\begin{equation}\label{contradiction}
\|u_n(t_n)-v\|_{H^1}+ \|\partial_t u_n(t_n)-\omega Av\|_{L^2}\geq \epsilon.
\end{equation}
We claim that $\{u_n(t_n)\}$ is a minimizing sequence for $\cGo$ in $X_l.$ Our claim together with the main theorem of the elliptic section imply that $\{u_n(t_n)\}$ will have a subsequence converging to a point in $\cS.$ We will use this to contradict (\ref{contradiction})\\
To prove the claim we first introduce the conserved quantity charge $Q(u(t)):= \int_{S^2}Au(t)\cdot u_t(t) dg_{S^2}= \frac{1}{l}\int_{S^2}(u_\theta\cdot u_t)(t) dg_{S^2}.$ Since $\tu_t = A\omega \tu$ the charge at zero is $Q(\tu(0))= \omega\int_{S^2}|A\tu|^2 dg_{S^2}$ and therefore $E(\tu(0))= \cGo(u^0)+ \omega Q(\tu(0)).$ Conservation of energy now implies
\begin{align*}
&E(u_n(t_n))= E(u_n(0))\leq E(\tu(0))+\frac{c}{n}\\
&= \cGo(u^0)+\omega Q(\tu(0))+\frac{c}{n},
\end{align*}
for a generic constant c.\\
Let $D(u(t)):=\frac{1}{2}\int_{S^2}|u_t(t)-\omega Au(t)|^2dg_{S^2}$ For a general $u$ we have
\begin{align*}
&\cGo(u)+\omega Q(u)+ D(u)\\
&= \frac{1}{2}\int_{S^2}\Big(|\nabla u|^2-\omega^2|Au|^2+2\omega Au\cdot u_t+|u_t|^2+\omega^2|Au|^2-2\omega Au\cdot u_t\Big)dg_{S^2}\\
&= E(u),
\end{align*}
and therefore by conservation of charge
\begin{align*}
&E(u_n(t_n))= \cGo(u_n(t_n))+\omega Q(u_n(t_n)) + D(u_n(t_n))\\
&= \cGo(u_n(t_n))+\omega Q(u_n(0)) + D(u_n(t_n))\\
&\geq \cGo(u_n(t_n))+\omega Q(\tu(0))+D(u_n(t_n))-\frac{\tilde{c}}{n},
\end{align*}
for some generic constant $\tilde{c}.$ With $d(l,\omega)$ denoting the greatest lower bound for $\cGo$ on $X_l,$ our two conservation inequalities imply
\begin{align*}
&d(l,\omega)\leq\cGo(u_n(t_n))\leq D(u_n(t_n))+ \cGo(u_n(t_n))\leq \cGo(u^0)+\frac{C}{n}\\
&=d(l,\omega)+\frac{C}{n},
\end{align*}
which implies that $\{u_n(t_n)\}$ is a minimizing sequence for $\cGo$ in $X_l,$ and that $D(u_n(t_n))\rightarrow 0$ as $n\rightarrow \infty.$ By the main theorem of the elliptic section, there exists a $v \in \cS$ such that $\|u_n(t_n)-v\|_{H^1}\rightarrow 0$. Moreover
\begin{align*}
\|\partial_tu_n(t_n)-\omega A v\|^2_{L^2}\leq D(u_n(t_n))+\|\omega A u_n(t_n)-\omega A v\|^2_{L^2}\rightarrow 0,
\end{align*}
giving us the desired contradiction.
\end{proof}
\acknowledgements{I would like to thank my advisor Professor Joachim Krieger for introducing me to this topic, and for many helpful discussions. I would also like to thank Professor Tahvildar-Zadeh for many fruitful discussions.}
   
\end{document}